\documentclass[a4paper,11pt]{amsart}

\usepackage{amsmath}
\usepackage{amsfonts}	
\usepackage{mathrsfs}
\usepackage{mathtools}	
\usepackage{savesym}
\usepackage{amssymb}  	
\usepackage{esvect}		
\usepackage{ifthen}		
\usepackage{color} 	 	
\usepackage{graphicx}  	
\usepackage{caption}  	
\usepackage{amscd} 	  	
\usepackage{float} 	  	
\usepackage[svgnames]{xcolor}

\usepackage[shortlabels,inline]{enumitem} 
\usepackage{tikz}    
\usepackage{tikz-cd} 

\usepackage[a4paper,%
			margin=1in,%
			top=1.2in,
			bottom=1.2in,
			]{geometry}

\usepackage[numbers,sort]{natbib} 	

\usepackage{aliascnt} 
\usepackage[pdfusetitle,
		bookmarks=true,
		pagebackref=false,
		colorlinks=true,
		linkcolor=Navy,
		citecolor=Navy,
		urlcolor=black]{hyperref}

\usepackage{amsthm}

\theoremstyle{plain}
\newaliascnt{theorem}{dummy}
\newtheorem{theorem}[theorem]{Theorem}
\aliascntresetthe{theorem}

\newaliascnt{proposition}{dummy}
\newtheorem{proposition}[proposition]{Proposition}
\aliascntresetthe{proposition}

\newaliascnt{corollary}{dummy}
\newtheorem{corollary}[corollary]{Corollary}
\aliascntresetthe{corollary}

\newaliascnt{lemma}{dummy}
\newtheorem{lemma}[lemma]{Lemma}
\aliascntresetthe{lemma}

\newaliascnt{conjecture}{dummy}

\aliascntresetthe{conjecture}

\theoremstyle{definition}
\newaliascnt{definition}{dummy}
\newtheorem{definition}[definition]{Definition}
\aliascntresetthe{definition}

\newaliascnt{example}{dummy}
\newtheorem{example}[example]{Example}
\aliascntresetthe{example}

\theoremstyle{remark}
\newaliascnt{remark}{dummy}
\newtheorem{remark}[remark]{Remark}
\aliascntresetthe{remark}

\numberwithin{equation}{section} 

\allowdisplaybreaks

\newcommand{\calA}{\mathcal{A}}
\newcommand{\calB}{\mathcal{B}}
\newcommand{\calC}{\mathcal{C}}
\newcommand{\calD}{\mathcal{D}}
\newcommand{\calE}{\mathcal{E}}
\newcommand{\calF}{\mathcal{F}}
\newcommand{\calG}{\mathcal{G}}
\newcommand{\calI}{\mathcal{I}}

\newcommand{\calM}{\mathcal{M}}
\newcommand{\calP}{\mathcal{P}}

\newcommand{\bbQ}{\mathbb{Q}}
\newcommand{\bbR}{\mathbb{R}}

\newcommand{\bbE}{\mathbb{E}}
\newcommand{\bbZ}{\mathbb{Z}}
\newcommand{\bbN}{\mathbb{N}}
\newcommand{\bbP}{\mathbb{P}}

\newcommand{\bfP}{\mathbf{P}}
\newcommand{\bfE}{\mathbf{E}}
\newcommand{\bfI}{\mathbf{I}}

\DeclareMathOperator*{\intxn}{\cap}
\DeclareMathOperator*{\Intxn}{\bigcap}
\DeclareMathOperator*{\union}{\cup}
\DeclareMathOperator*{\Union}{\bigcup}

\DeclareMathOperator{\diam}{diam}
\DeclareMathOperator{\diag}{diag}
\DeclareMathOperator{\supp}{supp}
\DeclareMathOperator{\Var}{Var}

\newcommand{\Hof}[1]{H\!\left(#1\right)}

\newcommand{\Eof}[2][]{\bfE_{#1}\!\left(#2\right)}
\newcommand{\Iof}[2][]{\bfI_{#1}\!\left(#2\right)}
\newcommand{\bbPof}[2][]{\bbP_{#1}\!\left\{#2\right\}}
\newcommand{\bbEof}[2][]{\bbE_{#1}\!\left(#2\right)}

\providecommand{\abs}[1]{\lvert#1\rvert}
\providecommand{\Abs}[1]{\left\lvert#1\right\rvert}
\providecommand{\norm}[2][]{\lVert#2\rVert_{#1}}
\providecommand{\Norm}[2][]{\left\lVert#2\right\rVert_{#1}}

\newcommand{\euclid}[1][d]{\mathbb{R}^{#1}}
\newcommand{\innp}[2]{\langle #1, #2 \rangle}

\newcommand{\wt}[1]{\widetilde{#1}}
\newcommand{\wh}[1]{\widehat{#1}}

\newcommand{\ul}[1]{\underline{#1}}
\newcommand{\ol}[1]{\overline{#1}}
\newcommand{\bigO}[2][]{O_{#1}\!\left(#2\right)}

\newcommand{\mFor}{\quad\text{for }}
\newcommand{\mAnd}{\quad\text{ and }\quad}

\newcommand{\by}[1]{\text{(by #1)}}
\newcommand{\dimH}{\dim_{\mathrm{H}}}
\newcommand{\dimP}{\dim_{\mathrm{P}}}

\newcommand{\HcalA}[1]{\mathrm{H}^{\mathcal{A}}_{#1}}

\newcommand{\betaOXj}[3]{\beta^{#1,\xi_{#3}}_{#2}}
\newcommand{\PiBall}[3][]{B^{\Pi_{#1}}\!\left(#2,#3\right)}
\newcommand{\ugamma}[2]{\ol{\gamma}_{#1,#2}^{\omega}(x)}
\newcommand{\lgamma}[2]{\ul{\gamma}_{#1,#2}^{\omega}(x)}
\newcommand{\cgamma}[2]{\gamma_{#1,#2}^{\omega}(x)}
\newcommand{\indicator}[1]{\textbf{\large 1}_{#1}}
\newcommand{\aev}{\text{-a.e.\@ }}
\newcommand{\calN}{\mathcal{N}}
\newcommand{\bfQ}{\mathbf{Q}}
\providecommand{\flr}[1]{\lfloor#1\rfloor}
\providecommand{\Flr}[1]{\left\lfloor#1\right\rfloor}

\newcommand{\calU}{\mathcal{U}}

\newcommand{\HcalCAn}[1]{\mathrm{H}^{\mathcal{C},\mathcal{A},n}_{#1}}
\newcommand{\hcalCA}[1]{\mathrm{h}^{\mathcal{C},\mathcal{A}}_{#1}}
\newcommand{\fof}[2][]{f_{#1}\left(#2\right)}
\newcommand{\hRWphiA}{h_{RW}(\Phi, \calA)}
\newcommand{\diff}{\mathrm{d}}
\newcommand{\wrt}{w.r.t.\@ }
\newcommand{\bfr}{\mathbf{r}}
\newcommand{\bft}{\mathbf{t}}

\title{Dimension of diagonal self-affine measures with exponentially separated projections}
\author{Zhou Feng}
\address{Faculty of Mathematics \\
	Technion - Israel Institute of Technonlogy\\
	Haifa, Israel}
\curraddr{}
\email{\href{mailto: zfeng@campus.technion.ac.il}{zfeng@campus.technion.ac.il}}
\thanks{}

\subjclass[2020]{28A80, 37C45}
\keywords{Self-affine measure, Hausdorff dimension, Lyapunov dimension, disintegration}
\thanks{This research was supported by the Israel Science Foundation (grant No.\ 619/22).}

\begin{document}
	
\begin{abstract}
	Let $ \mu $ be a self-affine measure associated with a diagonal affine iterated function system (IFS) $ \Phi = \{ (x_{1}, \ldots, x_{d}) \mapsto ( r_{i, 1}x_{1} + t_{i,1}, \ldots, r_{i,d}x_{d} + t_{i,d}) \}_{i\in\Lambda} $ on $ \mathbb{R}^{d} $ and a probability vector $ p = (p_{i})_{i\in\Lambda}$. For $ 1 \leq j \leq d $, denote the $ j $-th the Lyapunov exponent by $ \chi_{j} := \sum_{i\in\Lambda} - p_{i} \log | r_{i,j} |$, and define the IFS induced by $ \Phi $ on the $j$-th coordinate as $ \Phi_{j} := \{ x \mapsto r_{i,j}x + t_{i,j}\}_{i\in\Lambda}$. We prove that if $ \chi_{j_{1}} \neq \chi_{j_{2}} $ for $ 1 \leq j_{1} < j_{2} \leq d $, and $ \Phi_{j}$ is exponentially separated for $ 1 \leq j \leq d $, then the dimension of $ \mu $ is the minimum of $ d $ and its Lyapunov dimension. This confirms a conjecture of Rapaport~\cite{Rapaport2023} by removing the additional assumption that the linear parts of the maps in $ \Phi $ are contained in a 1-dimensional subgroup. One of the main ingredients of the proof involves disintegrating $ \mu $ into random measures with convolution structure. In the course of the proof, we establish new results on dimension and entropy increase for these random measures.
\end{abstract}

\maketitle

\section{Introduction}\label{sec:intro}

\subsection{Background and main results} Computing the dimension of self-affine fractals remains a fundamental open problem in fractal geometry; see \cite{BaranyEtAl2023a,Falconer2003}. This paper focuses on determining the dimension of diagonal self-affine measures under mild assumptions. 

An affine iterated function systems (IFS) is a nonempty finite collection $ \Phi = \{ \varphi_{i}(x) = A_{i} x + t_{i} \}_{i\in\Lambda}$ of contracting affine maps on $ \euclid[d] $. It is well known \cite{Hutchinson1981} that there is a unique nonempty compact $ K_{\Phi} $, called the \textit{self-affine set}, satisfying $ K_{\Phi} = \union_{i\in\Lambda} \varphi_{i}(K_{\Phi}) $. Given a probability vector $ p = (p_{i})_{i\in\Lambda} $, the associated \textit{self-affine measure} $ \mu $ is the unique Borel probability measure on $ \euclid[d] $ such that $ \mu = \sum_{i \in \Lambda} p_{i} \varphi_{i} \mu $, where $ \varphi_{i} \mu = \mu \circ \varphi_{i}^{-1} $ denotes the pushforward measure. When the linear parts $ \{A_{i}\}_{i\in\Lambda} $ are diagonal matrices, $ \Phi $ and $ \mu $ are referred to as \textit{diagonal}. In recent years, the exact dimensionality of self-affine measures has been established (see \cite{FengHu2009} for diagonal case and \cite{BaranyKaeenmaeki2017,Feng2023a} for general case). That is, there exists a number $ \dim \mu $, called the \textit{dimension} of $ \mu $, such that
\begin{equation*}
	\lim_{r\to 0} \frac{\log B(x,r)}{\log r} = \dim \mu \mFor \mu \aev x,
\end{equation*}
where $ B(x,r)$ denotes the closed ball centered at $x$ with radius $ r $.

The dimension theory of self-affine sets and measures has been extensively studied. Notably, Falconer~\cite{Falconer1988} introduced the \textit{affinity dimension} $ \dim_{A} \Phi $ which only depends on the linear parts $ \{A_{i}\}_{i\in\Lambda}$. He proved that if $ \norm{A_{i}} < 1/2 $ for all $ i \in \Lambda $, then for Lebesgue almost all translations $ \{t_{i}\}_{i\in\Lambda} $,
\begin{equation}\label{eq:Fal-eq}
	\dimH K_{\Phi} = \min \left\{d, \dim_{A}\Phi \right\},
\end{equation}
where $ \dimH $ denotes the Hausdorff dimension. (In fact, Falconer proved this for $ \norm{A_{i}} < 1/3 $; Solomyak~\cite{Solomyak1998} later showed that $ \norm{A_{i}} < 1/2$ suffices.) Similar results for self-affine measures were obtained by Jordan, Pollicott and Simon~\cite{JordanEtAl2007}, who showed that, under the same norm condition, for Lebesgue almost all $ \{t_{i}\}_{i\in\Lambda} $,
\begin{equation}\label{eq:JPS-eq}
	\dim \mu = \min \left\{d, \dim_{L}(\Phi, p) \right\},
\end{equation}
where $ \dim_{L}(\Phi, p)$ is the \textit{Lyapunov dimension} defined in \eqref{eq:def-LyDim}.

While the above results provide a characterization of typical cases, finding explicit and verifiable conditions for \eqref{eq:Fal-eq} and \eqref{eq:JPS-eq} to hold remains an open challenge. Recently, significant progress has been made in this direction, particularly under the assumption that $\{A_{i}\}_{i\in\Lambda}$ is strongly irreducible (see \cite{BaranyEtAl2019,HochmanRapaport2022,MorrisShmerkin2019} for $ d = 2 $ and \cite{Rapaport2024,MorrisSert2023} for $ d = 3$). 

Diagonal systems, which contrast with and complement the strongly irreducible case, form a significant subclass of IFSs that have been studied since the 1980s~\cite{Bedford1984,McMullen1984}. In this paper, we consider a diagonal affine IFS on $ \euclid[d]$:
\begin{equation}\label{eq:def-diagIFS}
	\Phi = \{ \varphi_{i}(x) = A_{i} x + t_{i} \}_{i \in \Lambda },
\end{equation}
where $ A_{i} = \diag (r_{i,1}, \ldots, r_{i,d})$ ($ 0 < \abs{r_{i,j}} < 1$) are diagonal matrices, and $ t_{i} = (t_{i,1}, \ldots, t_{i, d}) \in \euclid[d]$. Let $ K_{\Phi}$ denote the corresponding self-affine set. Given a probability vector $ p = (p_{i})_{i\in\Lambda}$, let $ \mu $ be the self-affine measure associated with $ \Phi $ and $p$. To state the results concerning the dimensions of $ K_{\Phi} $ and $ \mu $, we introduce some definitions. For $ 1 \leq j \leq d $, denote the \textit{$j$-th Lyapunov exponent} by $\chi_{j} := \sum_{i\in\Lambda} - p_{i} \log \abs{r_{i,j}} $, and define the \textit{IFS induced by $ \Phi $ on the $j$-th coordinate} as $ \Phi_{j} := \{ x \mapsto r_{i,j} x + t_{i,j} \}_{i\in\Lambda} $. Without loss of generality, we assume after possibly permuting the coordinates that $ \chi_{1} \leq \cdots \leq \chi_{d} $. The \textit{Lyapunov dimension} for $ \Phi $ and $ p $ is given by
\begin{equation}\label{eq:def-LyDim}
	\dim_{L}(\Phi, p) = f_{\Phi}(H(p)),
\end{equation}
where $ \Hof{p} := \sum_{i\in\Lambda} - p_{i} \log p_{i} $ is the \textit{entropy}, and $ f_{\Phi} \colon [0, \infty) \to [0, \infty) $ is a function defined as
\begin{equation}\label{eq:f-LyaDim}
	f_{\Phi}(x) = \begin{cases}
		j +\frac{x - \sum_{b=1}^{j}\chi_{b} }{\chi_{j+1}} & \text{if } x \in \left[ \sum_{b=1}^{j}\chi_{b}, \sum_{b=1}^{j+1}\chi_{b}\right) \text{ for some } 0 \leq j \leq d - 1; \\[0.7em]
		d \frac{x }{\sum_{b=1}^{d}\chi_{b}} & \text{if } x \in \left[ \sum_{b=1}^{d} \chi_{b}, \infty\right) .
	\end{cases}
\end{equation}

Next, we introduce the mild separation conditions, originally arising from Hochman's seminal work~\cite{Hochman2014}. Given two affine maps $ \psi_{1}, \psi_{2}  \colon \bbR \to \bbR $ with $ \psi_{i}(x) = s_{i}x + b_{i}$ for $i = 1, 2 $, define
\begin{equation*}
	d(\psi_{1}, \psi_{2}) := \begin{cases}
		\infty & \text{if } s_{1} \neq s_{2}; \\
		\abs{b_{1}-b_{2}} & \text{otherwise}.
	\end{cases}
\end{equation*}
For an affine IFS $ \Psi = \{\psi_{i}\}_{i\in \Lambda} $ on $ \bbR $ and $ n \in \bbN $, denote $ \psi_{u} = \psi_{u_{1}} \cdots \psi_{u_{n}} $ for $ u = u_{1} \cdots u_{n} \in \Lambda^{n} $. Define
\begin{equation} \label{eq:def-ESC-sep}
	\Delta_{n}(\Psi) = \min \{ d(\psi_{u}, \psi_{v}) \colon u, v\in \Lambda^{n}, \, u \neq v \}
\end{equation}
and
\begin{equation}\label{eq:def-Diop-sep}
	S_{n}(\Psi) = \min \{ d(\psi_{u}, \psi_{v}) \colon u, v\in \Lambda^{n}, \, \psi_{u} \neq \psi_{v} \},
\end{equation}
with the convention $ \min \emptyset = 0 $.

\begin{definition} \label{def:SepConds}
	Let $ \Psi $ be an affine IFS on $ \bbR $. We call $ \Psi $ \textit{exponentially separated} (resp.\ \textit{Diophantine}) if there exists $ c > 0 $ so that $ \Delta_{n}(\Psi) > c^{n} $ (resp.\ $S_{n}(\Psi) > c^{n} $) for infinitely many $ n \in \bbN $. We say $ \Psi $ has \textit{no exact overlaps} if $ \Delta_{n}(\Psi) > 0 $ for all $ n \in \bbN $, or equivalently, the semigroup generated by $ \Psi $ is free.
\end{definition}

\begin{remark}\label{rml:seps}
	It follows from \autoref{def:SepConds} that $ \Psi $ is exponentially separated if and only if $ \Psi $ is both Diophantine and has no exact overlaps. Furthermore, $ \Psi $ is Diophantine if it is defined by algebraic parameters (see \cite{Hochman2014}).
\end{remark}

Very recently, Rapaport~\cite{Rapaport2023} made a breakthrough in the dimension theory of diagonal self-affine sets and measures. Specifically, \cite[Theorem 1.3]{Rapaport2023} establishes that \eqref{eq:Fal-eq} holds if, for each $ 1 \leq j_{1} < j_{2} \leq d $ there is $ i \in \Lambda $ so that $ \abs{r_{i,j_{1}}} \neq \abs{r_{i,j_{2}}}$, and $ \Phi_{j} $ is exponentially separated for $ 1 \leq j \leq d $. This builds on an analogous result regarding the dimension of $ \mu $ (\cite[Theorem 1.7]{Rapaport2023}) under the additional assumption that the linear parts of $ \Phi$ lie within a $ 1$-dimensional subgroup. That is, there exist $ c_{1}, \ldots, c_{d} > 0 $ such that
\begin{equation*}
	(\abs{r_{i,1}}, \ldots, \abs{r_{i,d}} ) \in \left\{ (c_{1}^{t}, \ldots, c_{d}^{t}) \colon t\in \bbR \right\} \text{ for all } i \in \Lambda.
\end{equation*}
This assumption is satisfied, in particular, when $ A_{i} $ is the same for all $ i \in \Lambda $. Regarding this assumption, Rapaport pointed out that his argument crucially depends on it, but he expects the result remains true without it (see \cite[Remark 1.8]{Rapaport2023}). Our main result confirms his conjecture by removing the additional assumption.

\begin{theorem} \label{thm:main-mu}
	If $ \chi_{1} < \cdots < \chi_{d}$ and $ \Phi_{j} $ is exponentially separated for $ 1 \leq j \leq d$, then 
	\begin{equation*}
		\dim \mu = \min \left\{d, \dim_{L}(\Phi, p)\right\}.
	\end{equation*}
\end{theorem}

Before discussing the proof of \autoref{thm:main-mu} in \autoref{subsec:AboutPf}, we provide some remarks on the assumptions and discuss several applications.

\begin{remark}\label{rmk:LyExp-Distinct}
	Due to the phenomenon of saturation (see \cite[Example 1.2]{Hochman2017}), it is not hard to find examples showing that the assumption $ \chi_{1} < \cdots < \chi_{d} $ cannot be dropped. For the reader's convenience, we give one such example. Let $ \lambda \in \bbQ \intxn (1/\sqrt{2}, 1) $ and $ n > 2 $ such that $ \lambda^{n} < 1/3 $. Define $\Psi = \{ \psi_{0}(x) = \lambda x, \psi_{1}(x) = \lambda x + 1 \} $. Consider the IFS $ \Phi = \left\{ \varphi_{u}  \right\}_{u\in \{0,1\}^{n}} $ on $ \euclid[2] $ given by $ \varphi_{0\cdots 0}(x,y) = (\lambda^{n} x + \psi_{1\cdots 1}(0), \lambda^{n} y) $, $ \varphi_{1\cdots 1}(x,y) = (\lambda^{n} x, \lambda^{n} y + \psi_{1\cdots 1}(0) ) $ and $ \varphi_{u}(x,y) = (\lambda^{n}x + \psi_{u}(0), \lambda^{n}y + \psi_{u}(0))$ for $ u \notin \{ 0\cdots 0, 1 \cdots 1 \}$. Let $ \mu $ the self-affine measure associated with $ \Phi $ and the uniform probability vector $ p $ on $ \{0,1\}^{n}$. Since the orthogonal projection of $ \Phi $ onto the line $ \{(t,-t) \colon t \in \bbR \} $ generates a Cantor set, it follows from $ \lambda^{n} < 1/3 $ and $ \lambda > 1 / \sqrt{2} $ that
	\begin{equation*}
		\dim \mu \leq  1 + \frac{\log 3}{ - n \log \lambda } < 2 = \min \left\{ 2, \frac{\log 2}{-\log \lambda}\right\} = \min \left \{2, \dim_{L}(\Phi, p) \right\}. 
	\end{equation*}
	On the other hand, by \autoref{rml:seps} and $ \lambda \in \bbQ $, the IFS $ \Phi_{1} = \Phi_{2} = \Psi^{n} = \{ \psi_{u}\}_{u\in\{0,1\}^{n}}$ is exponentially separated.
\end{remark}

\begin{remark}
		Various carpet-like examples (see e.g.\ \cite{Bedford1984,McMullen1984,LalleyGatzouras1992,Baranski2007,FengWang2005,Fraser2012a}) indicate that it is necessary to assume that $ \Phi_{j}$ has no exact overlaps for $ 1 \leq j \leq d $. One may expect that the result remains true under this necessary assumption. Recently, Rapaport and Ren~\cite{RapaportRen2024} verified this conjecture for homogeneous diagonal IFSs with rational translations.\footnote{The author believes that incorporating the results from \cite{FengFeng2024} into~\cite{RapaportRen2024} can relax the assumption of rational translations to algebraic translations.} However, even when $ d = 1 $, this conjecture is considered one of the major open problems in fractal geometry and well beyond our reach (see \cite{Hochman2018,Varju2023}).
\end{remark}

\subsection{Applications}

By \autoref{rml:seps}, the following is a direct application of \autoref{thm:main-mu}. 

\begin{corollary} \label{cor:alg-param}
	Suppose $ \chi_{1} < \cdots < \chi_{d}$. If for $ 1 \leq j \leq d $, $ \Phi_{j} $ is defined by algebraic parameters and has no exact overlaps, then $ \dim \mu = \min \left\{ d, \dim_{L} (\Phi, p)\right\}$.
\end{corollary}

Below we determine the dimension of a concrete new example by \autoref{cor:alg-param}.

\begin{example}
	Let $ a, b \in (1/2,1)$ be distinct algebraic numbers such that $ P(a,b) \neq 0 $ for each two\nobreakdash-variable polynomial $ P $ with coefficients in $ \{0, \pm 1 \}$ and $ P(0, 0) = 1 $. For example, choose $ a = q_{1}/q_{2}, b = q_{2}/q_{3} \in \bbQ$, where $ q_{1}, q_{2}, q_{3} $ are distinct prime numbers. Let $ \mu $ be the self-affine measure associated with the IFS $ \Phi = \{ (x,y) \mapsto (\alpha x, \beta y), (x,y) \mapsto (\beta x + 1, \alpha y + 1)\}$ on $ \euclid[2]$ and the probability vector $ p = (p_{1}, 1 - p_{1})$ with $ p_{1} \in (0, 1/2)$. Then $ \dim \mu = \min \{ 2, \dim_{L} (\Phi, p)\}$.
\end{example}


Next, we give a result about the typical validity of \eqref{eq:JPS-eq} in the spirit as \cite[Theorem 1.8]{Hochman2014}. By $ \dimP $ we denote the packing dimension. Recall that $ \dimH E \leq  \dimP E  $ for $ E \subset \euclid[d]$. For $ m \geq 2 $, let $ \Delta^{m-1}$ denote the set of probability vectors in $ \euclid[m] $.

\begin{corollary} \label{prop:typical-Dim}
Let $ m \geq 2 $ and let $ \bft = (t_{i,j})_{1 \leq i \leq m, 1 \leq j \leq d }\in \euclid[dm] $ such that $ t_{i_{1},j} \neq t_{i_{2},j} $ for $ 1 \leq i_{1} < i_{2} \leq  m $ and $ 1 \leq j \leq  d $.  For $ \bfr = (r_{i,j})_{1 \leq i \leq m, 1 \leq  j \leq d } \in ( (-1,1) \setminus \{0\} )^{dm} $ and $ p \in \Delta^{m-1} $, let $ \mu_{\bfr,p}$ denote the self-affine measure associated with the IFS $ \Phi_{\bfr} = \left \{ (x_{j})_{1 \leq j \leq d} \mapsto (r_{i,j}x_{j}+ t_{i,j})_{1 \leq j \leq d }\right\}_{i=1}^{m} $ and the probability vector $ p $. Then, there exists $ \calE_{\bft} \subset  ( (-1,1) \setminus \{0\} )^{dm} $ with $ \dimP \calE_{\bft} \leq dm - 1 $ such that for $ \bfr \notin \calE_{\bft} $, there exists $ \calF_{\bfr} \subset \Delta^{m-1} $ with $ \dimP \calF_{\bfr} \leq m - 2  $ so that for $ p \notin \calF_{\bfr} $, 
$ \dim \mu_{\bfr,p} = \min \left\{ d, \dim_{L}(\Phi_{\bfr}, p) \right\} $.
\end{corollary}

\begin{proof} 
	For $ \bfr = (r_{i,j})_{1 \leq i \leq m, 1 \leq  j \leq d } $ and $1 \leq j \leq d$, let $\bfr_{j} = (r_{i,j})_{i=1}^{m} $. Consider the IFS $\Phi(\bfr_{j}) = \left\{ x \mapsto r_{i,j}x + t_{i,j} \right\}_{i=1}^{m}$ on $\bbR$, with its coding map denoted by $\Pi_{\Phi(\bfr_{j})}$ (see \eqref{eq:def-coding-map}). For distinct sequences $x = (x_{k}), y = (y_{k}) \in \{1,\ldots, m\} ^{\bbN}$, there exists $n \in \bbN$ such that $x_{n} \neq y_{n}$ and $x_{k} = y_{k}$ for $k < n$. This gives:
	\begin{align*}
		\Delta_{x,y}(\bfr_{j}) & :=  \Pi_{\Phi(\bfr_{j})}(x) - \Pi_{\Phi(\bfr_{j})}(y) \\
		& = r_{x_{1},j} \cdots r_{x_{n-1},j} \left( (t_{x_{n},j} - t_{y_{n},j}) + \sum_{k=n}^{\infty} \left( r_{x_{n},j} \cdots r_{x_{k},j} t_{x_{k+1},j} - r_{y_{n},j} \cdots r_{y_{k},j} t_{y_{k+1},j} \right) \right).
	\end{align*}
	Since $t_{1,j}, \ldots, t_{m,j}$ are distinct, we have $t_{x_{n}} - t_{y_{n}} \neq 0$. Consequently, $\Delta_{x,y}(\bfr_{j}) \neq 0$ if the norm of $\bfr_{j}$ is sufficiently small, ensuring that the summation in the above expression is small, depending on $(t_{i,j})_{i=1}^{m}$. Thus, $\Delta_{x,y}(\bfr_{j})$ is a nonzero real analytic function of $\bfr_{j}$ on each connected component of $((-1,1) \setminus \{0\})^{m}$. By applying \cite[Theorem 1.10]{Hochman2017}, for each $1 \leq j \leq d$, there exists $\calE_{j} \subset ((-1,1) \setminus \{0\})^{m}$ with $\dimP \calE_{j} \leq m - 1$ such that $\Phi(\bfr_{j})$ is exponentially separated for $\bfr_{j} \notin \calE_{j}$. Define
	\begin{equation*}
		\calE_{\bft}' = \Union_{j=1}^{d} \left\{ \bfr \in ((-1,1) \setminus \{0\})^{dm} \colon \bfr_{j} \in \calE_{j} \right\},
	\end{equation*}
	and
	\begin{equation*}
		\calE =  \Union_{1 \leq j_{1} < j_{2} \leq d} \left\{ (r_{i,j})_{1 \leq i \leq m, 1 \leq j \leq d} \in ((-1,1) \setminus \{0\})^{dm} \colon \abs{r_{i, j_{1}}} = \abs{r_{i, j_{2}}} \text{ for } 1 \leq i \leq m  \right\}.
	\end{equation*}
	Set $ \calE_{\bft} := \calE_{\bft}' \union \calE $. Thus, for  $ \bfr \notin \calE_{\bft}$ and $ 1 \leq j \leq d $, $ \Phi(\bfr_{j} ) $ is exponentially separated. Since $ \dimP \calE_{\bft}'\leq dm -1 $ and $ \dimP \calE \leq dm - m$, we have $ \dimP \calE_{\bft} \leq dm - 1 $.
	
	For $ \bfr =  (r_{i,j})_{1 \leq i \leq m, 1 \leq  j \leq d } \in ((-1,1) \setminus \{0\})^{dm} \setminus \calE_{\bft} $ and $ 1 \leq j_{1} < j_{2} \leq d $, define a vector $ v_{j_{1}, j_{2}} := (\log \abs{r_{i,j_{1}}} - \log \abs{r_{i,j_{2}}})_{i=1}^{m} $. Then $  v_{j_{1}, j_{2}} \neq 0 $ by $ \bfr \notin \calE $. If $ v_{j_{1},j_{2}} $ is parallel to $ (1, \ldots, 1) $, then $ \Delta^{m-1} \intxn v_{j_{1},j_{2}}^{\perp} = \emptyset $, where $ v_{j_{1},j_{2}}^{\perp} $ denotes the orthogonal complement of $ v_{j_{1},j_{2}} $. Define
	\begin{equation*}
		\calF_{\bfr}' = \Union \left\{ v_{j_{1},j_{2}}^{\perp} \colon  1 \leq j_{1} < j_{2} \leq d \text{ and } v_{j_{1},j_{2}} \text{ is not parallel to } (1, \ldots, 1) \right\}.
	\end{equation*}
	 Set $ \calF_{\bfr} := \Delta^{m-1} \intxn \calF_{\bfr}' $. Then $ \dimP \calF_{\bfr} \leq m - 2 $. For $ p \notin \calF_{\bfr} $, the Lyapunov exponents of $ \mu_{\bfr,p}$ are distinct. The proof is finished by \autoref{thm:main-mu}.
\end{proof}

We determine the measures of full dimension on certain overlapping diagonal self-affine sets (see \cite{DasSimmons2017,GatzourasPeres1996,MorrisSert2022,BarralFeng2011,KenyonPeres1996a,Kaeenmaeki2004} for further discussion on this topic). A measure $ \nu $ on $ K_{\Phi} $ is called an \textit{ergodic measure of full dimension} if $ \dim \nu = \dimH K_{\Phi} $ and $ \nu = \Pi \bar{\nu} $, where $ \Pi $ is the coding map in \eqref{eq:def-coding-map}, and $ \bar{\nu} $ is an ergodic shift-invariant measure  on $ \Lambda^{\bbN} $. Let $ S_{d}$ denote the symmetric group over $ \{1, \ldots, d \}$. For $ \sigma \in S_{d} $, $ i \in \Sigma $ and $ s \geq 0 $, define
\begin{equation}\label{eq:SVF-sigma}
	\phi^{s}_{\sigma}(i) = \begin{cases}
		\abs{r_{i,\sigma(1)}} \cdots \abs{r_{i, \sigma(\flr{s})}} \cdot \abs{r_{i,\sigma(\flr{s}+1)}}^{s-\flr{s}} & \text{if } s < d;\\
		\abs{r_{i,1}\cdots r_{i,d}}^{s/d} & \text{if } s \geq d.
	\end{cases}
\end{equation}
By \cite[Theorem 2.1]{Fraser2015a}, the affinity dimension $ \dim_{A} \Phi $ is the unique $ s \geq 0 $ such that
\begin{equation}\label{eq:dimA-Sd}
	\max_{\sigma \in S_{d}} \sum_{i\in \Lambda} \phi_{\sigma}^{s}(i) = 1.
\end{equation}

\begin{corollary}\label{coro:MFD}
	Let $ \Phi $ be as in \eqref{eq:def-diagIFS} with $ d = 2 $. Suppose $ \abs{r_{i,1}} \neq \abs{r_{i,2}}$ for some $ i \in \Lambda $, and $ \Phi_{1}, \Phi_{2} $ are exponentially separated. Define $ \Sigma := \left\{ \sigma \in S_{2} \colon \sum_{i\in \Lambda} \phi^{\dim_{A} \Phi}_{\sigma}(i) = 1 \right\} $, which is nonempty by \eqref{eq:dimA-Sd}. If $ 0 < \dim_{A} \Phi < 2 $, then the ergodic measures of full dimension on $ K_{\Phi} $ are precisely the self\nobreakdash-affine measures associated with $\Phi $ and the probability vectors $ (\phi^{\dim_{A} \Phi}_{\sigma}(i))_{i\in\Lambda}$ for $ \sigma \in \Sigma $. In particular,  $ \Sigma = S_{2} $ when $ (\abs{r_{i,1}})_{i\in\Lambda} $ is a permutation of $ (\abs{r_{i,2}})_{i\in\Lambda}$.
\end{corollary}

\begin{proof} 
	We first show that the ergodic equilibrium states for the singular value function of diagonal matrices are Bernoulli.
	 Let $ \nu $ be an ergodic shift-invariant measure  on $ \Lambda^{\bbN} $. The Lyapunov dimension $ \dim_{L} \nu $ is defined as the unique $ s \geq 0 $ satisfying
	\begin{equation}\label{eq:dimL-nu}
		h(\nu) + \lim_{n\to\infty} \frac{1}{n} \int \log \phi^{s}(A_{x|n}) \, \diff \nu(x) = 0,
	\end{equation}
	where $ h(\nu) $ denotes the measure-theoretic entropy (see \cite{Walters1982}), $ A_{x|n} = A_{x_{1}} \cdots A_{x_{n}} $ for $ x = (x_{n}) \in \Lambda^{\bbN}$, and $ \phi^{s}(A)$ is the singular value function of $ A $ (see \cite{Falconer1988}).  For $ k 
	\in \bbZ \intxn [0, d] $, it is well known~\cite{Falconer1988} that $ \phi^{k}(A) = \norm{A^{\wedge k}}$, where $ \wedge$ denotes the exterior product, $ A^{\wedge k} $ is the linear map induced by $ A $ on $ \wedge^{k} \bbR^{d} $ as $ A^{\wedge k}( v_{1} \wedge \cdots \wedge v_{k} ) := (Av_{1})\wedge \cdots \wedge (Av_{k}) $ for $ v_{1}, \ldots, v_{k} \in \bbR^{d} $, and $ \norm{\cdot}$ is the standard Euclidean operator norm on $ \wedge^{k} \euclid[d] $. Since $  \wedge^{k} \euclid[d] = \mathrm{span}\left\{e_{\sigma(1)} \wedge \cdots \wedge e_{\sigma(k)} \colon \sigma \in S_{d} \right\} $ is a finite-dimensional vector space, where $ e_{1}, \ldots, e_{d} $ denote the standard basis of $ \euclid[d]$, we have for $ k \in \bbZ \intxn [0, d] $,
	\begin{align*}
		\lim_{n\to\infty} \frac{1}{n} \int \log \phi^{k}(A_{x|n}) \, \diff \nu(x) & = \lim_{n\to\infty} \frac{1}{n} \int \log \norm{A^{\wedge k}_{x|n}} \, \diff \nu(x) \\ 
		& = \lim_{n\to\infty} \frac{1}{n} \int \max_{\sigma \in S_{d}} \log  \Norm{A^{\wedge k}_{x|n} \left (e_{\sigma(1)} \wedge \cdots \wedge e_{\sigma(k) }\right)} \, \diff \nu(x) \\
		& = \max_{\sigma \in S_{d}} \sum_{i\in \Lambda} \nu([i]) \log \phi_{\sigma}^{k}(i),
	\end{align*}
	where $ [i] := \{ (x_{n}) \in \Lambda^{\bbN} \colon x_{1} = i \} $, while in the last equality we have used that $ \{A_{i}\}_{i\in\Lambda} $ are diagonal, and $ \nu $ is shift-invariant. From this and $ \phi^{s}(A) = \left(\phi^{\flr{s}}(A)\right)^{\flr{s}+1-s} \left(\phi^{\flr{s}+1}(A)\right)^{s - \flr{s}} $ for $ s \geq 0 $, it follows that
	\begin{equation}\label{eq:LyExpDiag}
		\lim_{n\to\infty} \frac{1}{n} \int \log \phi^{s}(A_{x|n}) \, \diff \nu(x) = \max_{\sigma \in S_{d}} \sum_{i\in \Lambda} \nu([i]) \log \phi_{\sigma}^{s}(i).
	\end{equation}
 	Let $ \beta_{\nu}$ denote the Bernoulli measure on $ \Lambda^{\bbN} $ with marginal $ (\nu([i]))_{i\in\Lambda}$. It is well known (see e.g.\ \cite{Walters1982}) that $ h(\nu) \leq h(\beta_{\nu}) $, with equality if and only if $ \nu = \beta_{\nu} $. From this, \eqref{eq:dimL-nu} and \eqref{eq:LyExpDiag}, it follows that $ \dim_{L} \nu \leq \dim_{L} \beta_{\nu} $,
	with equality if and only if $ \nu = \beta_{\nu} $. Combining this with $ \dim \Pi \nu \leq \dim_{L} \nu \leq \dim_{A} \Phi $ (see \cite{JordanEtAl2007}), \cite[Theorem 1.3]{Rapaport2023}, and $ \dim_{A} \Phi < 2 $ yields
	\begin{equation}\label{eq:DimRel}
		\dim \Pi \nu \leq \dim_{L} \nu \leq \dim_{L} \beta_{\nu} \leq \dim_{A} \Phi = \dim_{H} K_{\Phi},
	\end{equation}
	where the second inequality is strict unless $ \nu = \beta_{\nu} $, that is, $ \nu $ is Bernoulli. 
	
	Write $ s_{0} := \dim_{A}\Phi $. By Gibbs' inequality (see e.g. \cite[Lemma 9.9]{Walters1982}) and \eqref{eq:dimA-Sd}, the probability vectors $ p_{\sigma} := (\phi_{\sigma}^{s_{0}}(i))_{i\in\Lambda} $ for $ \sigma \in \Sigma $ are precisely the probability vectors $ q = (q_{i})_{i\in\Lambda}$ satisfying
	\begin{equation*}
	\sum_{i\in\Lambda} - q_{i} \log q_{i} + \max_{\sigma \in S_{d}} \sum_{i\in \Lambda} q_{i} \log \phi_{\sigma}^{s_{0}}(i) = \max_{\sigma \in S_{d}} \log \sum_{i\in \Lambda}  \phi_{\sigma}^{s_{0}}(i) = 0.
	\end{equation*}
	By \eqref{eq:dimL-nu} and \eqref{eq:LyExpDiag}, this implies that $ \dim_{L} (\Phi, p_{\sigma}) = \dim_{A} \Phi $ for $ \sigma \in \Sigma $.
	
	Let $ \sigma \in \Sigma $, and let $ \mu_{\sigma} $ be the self-affine measure associated with $ \Phi $ and $ p_{\sigma}$. By \eqref{eq:DimRel} and $ \dim_{L} (\Phi, p_{\sigma}) = \dim_{A} \Phi $, it suffices to prove that
		$\dim \mu_{\sigma} = \dim_{L} (\Phi, p_{\sigma})$.
	From \autoref{thm:main-mu}, it remains to verify that $ \chi_{\sigma(1)}(p_{\sigma}) \neq \chi_{\sigma(2)}(p_{\sigma})$. If there exists $ \alpha > 0 $ such that $ \abs{r_{i,\sigma(1)}} / \abs{r_{i,\sigma(2)}} = \alpha $ for all $ i \in \Lambda $, then $ \alpha \neq 1 $ since $\abs{r_{i,\sigma(1)}} \neq \abs{r_{i,\sigma(2)}} $ for some $ i \in \Lambda $, implying $ \chi_{\sigma(1)}(p_{\sigma}) \neq \chi_{\sigma(2)}(p_{\sigma})$. Now suppose there exist some $ i_{1} \neq i_{2} \in \Lambda $ such that $  \abs{r_{i_{1},\sigma(1)}} / \abs{r_{i_{1},\sigma(2)}} \neq  \abs{r_{i_{2},\sigma(1)}} / \abs{r_{i_{2},\sigma(2)}} $. Define $ t := s_{0} $ if $ s_{0} \in (0,1] $ and $ t := 2 - s_{0} $ if $ s_{0} \in (1, 2)$. Then
	\begin{align*}
		t \left(\chi_{\sigma(2)}(p_{\sigma}) - \chi_{\sigma(1)}(p_{\sigma}) \right) & = \sum_{i\in\Lambda} p_{\sigma}(i) \log  \frac{\abs{r_{i,\sigma(2)}}^{t}}{\abs{r_{i,\sigma(1)}}^{t}} \\ & <  \log  \sum_{i\in\Lambda} p_{\sigma}(i)  \frac{\abs{r_{i,\sigma(2)}}^{t}}{\abs{r_{i,\sigma(1)}}^{t}} \\& \leq \log \sum_{i\in\Lambda} \phi^{s_{0}}_{\sigma}(i) = 0,
	\end{align*}
	where the strict inequality is by the concavity of $ \log(\cdot) $ and $ \abs{r_{i_{1},\sigma(1)}} / \abs{r_{i_{1},\sigma(2)}} \neq \abs{r_{i_{2},\sigma(1)}} / \abs{r_{i_{2},\sigma(2)}} $, while the last inequality follows from $ \max_{\sigma'\in S_{2}} \sum_{i\in\Lambda} \phi^{s_{0}}_{\sigma'}(i) = \sum_{i\in\Lambda} \phi^{s_{0}}_{\sigma}(i) = 1 $. Since $ t > 0 $, we conclude that $\chi_{\sigma(1)}(p_{\sigma}) \neq \chi_{\sigma(2)}(p_{\sigma}) $, completing the proof.
\end{proof}

Recently, Py\"{o}r\"{a}l\"{a}~\cite{Pyoeraelae2025} determined the dimension of orthogonal projections of planar diagonal self-affine measures under an irrationality condition (see \cite{PeresShmerkin2009,HochmanShmerkin2012,FalconerKempton2017,FergusonEtAl2015,BaranyEtAl2023} for earlier results). Building on this, we combine \cite[Theorem 1.1]{Pyoeraelae2025} with \autoref{coro:MFD} to obtain the dimension of orthogonal projections for a class of overlapping self-affine sets.

\begin{corollary}
	Let $ \Phi $ be as in \eqref{eq:def-diagIFS} with $ d = 2 $. Suppose $ \abs{r_{i,1}} \neq \abs{r_{i,2}}$ for some $ i \in \Lambda $, and $ \Phi_{1}, \Phi_{2} $ are exponentially separated. Suppose further that there exist $ (i_{1},i_{2}) \in \Lambda^{2} $ and $ (j_{1}, j_{2}) \in \{1,2\}^{2} $ such that $ \log \abs{r_{i_{1},j_{1}}} / \log \abs{r_{i_{2},j_{2}}}  \notin \bbQ  $. Then $ \dimH \pi(K_{\Phi}) = \min \{1, \dim_{A} \Phi \} $ for each orthogonal projection $ \pi $ onto a line not parallel to the coordinate axes. For the orthogonal projection $ \pi_{j} $ onto the $j$-th coordinate axis with $j = 1,2$, $ \dimH \pi_{j}(K_{\Phi}) = \min\{ 1, \dim_{A} \Phi_{j}\}$. 
\end{corollary}

\subsection{About the proof} \label{subsec:AboutPf}
\autoref{thm:main-mu} is reduced from \autoref{thm:main-A} which concerns the dimension of a disintegration of the measure $ \mu $. This disintegration is defined as follows. For any partition $ \calE $ of a set $ X $, let $ \calE(x)$ denote the unique element of $ \calE $ containing $ x \in X $. Given $ u = u_{1} \cdots u_{n} \in \Lambda^{n}$, define $ \varphi_{u} = \varphi_{u_{1}} \circ \cdots \circ \varphi_{u_{n}}$. Fix $ N \in \bbN $. Define the partition $ \Gamma $ of $ \Lambda^{\bbN}$ by
\begin{equation}\label{eq:def-gamma}
	\Gamma(x) = \Gamma(y) \text{ if and only if } A_{\varphi_{x|N}} = A_{\varphi_{y|N}} \mFor x,y \in \Lambda^{\bbN},
\end{equation}
where $ A_{\psi}$ denotes the linear part of an affine map $ \psi$, and $ x | N $ represents the first $ N $ digits of $ x \in \Lambda^{\bbN} $. Endow $ \Lambda^{\bbN}$ with the product topology, and let $ \sigma $ be the shift map defined by $ \sigma((x_{k})_{k=1}^{\infty}) = (x_{k+1})_{k=1}^{\infty}$. Set $ T = \sigma^{N}$ and $\calA = \vee_{n=0}^{\infty} T^{-n} \Gamma $. Let $ \{ \beta^{\calA}_{x}\}_{x\in\Lambda^{\bbN}}$ be the disintegration of the Bernoulli measure $ \beta := p^{\bbN} $ on $ \Lambda^{\bbN}$ with respect to $ \calA $; see \autoref{subsec:some-disint} for further details. Define the quotient space $ \Omega = \Lambda^{\bbN} /\calA \cong \{1, \ldots, \abs{\Gamma} \}^{\bbN} $, and  endow it with the pushforward measure $ \bfP $ of $ \beta $ under the natural projection $ x \mapsto \calA(x)$. For $ \omega \in \Omega $, define $ \beta^{\omega} = \beta^{\calA}_{x}$ whenever $ \omega = \calA(x)$ for some $ x \in \Lambda^{\bbN}$. Then
\begin{equation}\label{eq:SymDisint}
	\beta = \int_{\Lambda^{\bbN}} \beta^{\calA}_{x} \, \diff \beta(x) = \int_{\Omega} \beta^{\omega} \, \diff \bfP(\omega).
\end{equation}
Let $ \Pi \colon \Lambda^{\bbN} \to \euclid $ be the coding map associated with $ \Phi $, defined by,
\begin{equation}\label{eq:def-coding-map}
	\Pi(x) = \lim_{n\to\infty} \varphi_{x_{1}} \circ \cdots \circ \varphi_{x_{n}}(0) \mFor x = (x_{n})_{n=1}^{\infty} \in \Lambda^{\bbN}.
\end{equation}
It is well known that $ \mu = \Pi \beta$.  For $ \omega \in \Omega $, define $ \mu^{\omega} = \Pi \beta^{\omega}$. Applying $ \Pi $ to \eqref{eq:SymDisint} yields the desired disintegration:
\begin{equation}\label{eq:Rd-disint}
	\mu = \int_{\Omega} \mu^{\omega} \, \diff \bfP(\omega).
\end{equation}

	Recently, similar disintegration techniques have been widely applied to study various properties of self-conformal measures; see e.g.\ \cite{GalicerEtAl2016,SagliettiEtAl2018,AlgomEtAl2022,BakerBanaji2024,KaeenmaekiOrponen2023,SolomyakSpiewak2023}. Notably, Saglietti, Shmerkin and Solomyak~\cite{SagliettiEtAl2018} established the typical absolute continuity of self-similar measures on the line. From this, \autoref{prop:typical-Dim} and \cite{Solomyak2022}, it seems possible to show the typical absolute continuity of diagonal self-affine measures, but we do not pursue this here. The idea of disintegrating stationary measures into well-behaved random measures was introduced by Galicer, Saglietti, Shmerkin and Yavicoli~\cite{GalicerEtAl2016}.
	
	While many prior works are motivated by the infinite convolution structure of random measures, our primary goal is to construct minimal cut-sets $ \calU_{n}$ of the finite words over $ \Lambda $. These cut-sets ensure that the cylinder sets $ \{\Pi([u])\}_{u\in\calU_{n}} $ have comparable diameters respectively along each coordinate. Such minimal cut-sets are naturally found in conformal settings (see \cite{Hochman2014,Rapaport2024a}) or under the specific assumptions on the linear parts of $ \Phi $ (see \cite{Rapaport2023}). However,  achieving this in general non\nobreakdash-homogeneous affine settings is almost impossible. Consequently, the additional assumption is crucial in \cite{Rapaport2023}. Later in this subsection, we further illustrate how the disintegration method underpins our approach.


	
	As a starting point, we establish the exact dimensionality of $ \mu^{\omega} $ for $ \bfP \aev \omega $; see \autoref{thm:L-Y-formula} for a detailed statement. \autoref{thm:L-Y-formula} is a version of \cite[Theorem 2.11]{FengHu2009} (see also \cite[Theorem 1.4]{Feng2023a}) in the context of disintegrations.
 
 \begin{theorem}\label{thm:ExactDim-Intro}
	There exists $ \dim \calA \geq 0 $ such that for $ \bfP \aev \omega $, $ \mu^{\omega}$ is exact dimensional with dimension given by $\dim \calA$. Furthermore, $ \dim \calA $ satisfies a Ledrappier-Young type formula \eqref{eq:LY-formula}.
\end{theorem}

It is well known \cite{Young1982} that for an exact dimensional measure $ \theta $, commonly used notions of dimension coincide. In particular, $ \dim \theta = \lim_{n\to\infty} \frac{1}{n} \Hof{ \theta, \calD_{n}} $, where $ \calD_{n} $ denotes the dyadic partition of $ \euclid[d] $. For the basics of entropy, please refer to \autoref{subsec:cond-e}. By \eqref{eq:Rd-disint} and the concavity of entropy, we obtain
\begin{equation}\label{eq:mu>A}
	\dim \mu = \lim_{n\to\infty} \frac{1}{n} \Hof{\int \mu_{\omega}\, \diff \bfP(\omega), \calD_{n}} \geq \lim_{n\to\infty} \int \frac{1}{n} \Hof{\mu^{\omega}, \calD_{n}} \, \diff \bfP(\omega) = \dim \calA.
\end{equation}

We are now ready to state the main theorem regarding the dimension of $ \mu^{\omega}$. For $ 1\leq j \leq d$, let $ \pi_{j} $ denote the orthogonal projection from $ \euclid $ to the $ j $-th coordinate axis. For $ n \in \bbN $, let $ \calC_{n} $ be the partition of $ \Lambda^{\bbN} $ such that $ \calC_{n}(x) = \calC_{n}(y) $ if and only if $  \varphi_{x|n} = \varphi_{y|n} $ for $ x, y \in \Lambda^{\bbN}$. The conditional entropy $ \Hof{\cdot, \cdot \mid \cdot }$ is defined in \eqref{eq:def-Hof}. 

\begin{theorem} \label{thm:main-A}
	Suppose $ \chi_{1} < \cdots < \chi_{d}$, and $ \Phi_{j} $ is Diophantine and for $ 1 \leq j \leq d$. Suppose further that for $ x, y \in \Lambda^{\bbN}, n \in \bbN $ and $ 1 \leq j \leq d $, $ \pi_{j} \varphi_{x|n} = \pi_{j} \varphi_{y|n} $ implies $ \varphi_{x|n} = \varphi_{y|n}$. Then
	\begin{equation*}
		\dim \calA = \min \{d, \fof[\Phi]{h_{RW}(\Phi,\calA)}\},
	\end{equation*}
	 where $ \fof[\Phi]{\cdot} $ is as defined in \eqref{eq:f-LyaDim}, and
		\begin{equation}\label{eq:def-hRW-calA}
			h_{RW}(\Phi, \calA) = \lim_{n\to\infty} \frac{1}{nN} \Hof{ \beta, \calC_{nN} \mid \wh{\calA} } = \inf_{n} \frac{1}{nN} \Hof{ \beta, \calC_{nN} \mid \wh{\calA} }.
		\end{equation}
	The limit exists by subadditivity (see \eqref{eq:subadd}).
\end{theorem}


\begin{proof}[Reduction of \autoref{thm:main-mu} from \autoref{thm:main-A}] Since $ \Phi_{j}$ is exponentially separated for $ 1  \leq j \leq d $, the assumptions of the theorem are satisfied, and $ \calC_{nN} = \vee_{i=0}^{nN-1} \sigma^{-i} \calP $, where $ \calP $ denotes the partition of $ \Lambda^{\bbN} $ based on the first digit. Note that
	$ \wh{\calA} = (\vee_{i=0}^{n-1} T^{-i}\wh{\Gamma}) \vee T^{-n} \wh{\calA} $, and $ \beta $ is Bernoulli. Then
	\begin{align*}
		\Hof{ \beta, \vee_{i=0}^{nN-1} \sigma^{-i} \calP \mid \wh{\calA} } & = \Hof{ \beta, \vee_{i=0}^{nN-1} \sigma^{-i} \calP \mid \vee_{i=0}^{n-1} T^{-i}\Gamma} & \by{\autoref{lem:list-identities}\ref{itm:condindp-info}} \\
		& = \Hof{\beta,\vee_{i=0}^{nN-1} \sigma^{-i} \calP } - \Hof{\beta, \vee_{i=0}^{n-1} T^{-i}\Gamma }  & \by{\autoref{lem:list-identities}\ref{itm:ChainRule-H}} \\
		& = (nN)\Hof{p} - \Hof{\beta, \vee_{i=0}^{n-1} T^{-i}\Gamma }.
	\end{align*}
	Since $ \{ A_{\varphi_{i}}\}_{i\in\Lambda}$ are commutative, by \eqref{eq:def-gamma} we have $ \Hof{\beta, \vee_{i=0}^{n-1} T^{-i}\Gamma } \leq n \log \Abs{\Gamma} \leq 2 n \abs{\Lambda} \log N $. From this, \eqref{eq:def-hRW-calA} and the above equation it follows that
	\begin{equation*}
	\Abs{\hRWphiA - \Hof{p}}\leq 2\abs{\Lambda} \frac{\log N}{N}.
	\end{equation*}
 	From this, \eqref{eq:mu>A}, \autoref{thm:main-A},  and \eqref{eq:def-LyDim}, letting $ N \to \infty $ yields that
	\begin{equation*}
		\dim \mu \geq \dim \calA = \min\left\{d, f_{\Phi}(\hRWphiA)\right\} \rightarrow \min\left\{d, \dim_{L}(\Phi, p) \right\}.
	\end{equation*}
	This completes the proof since $ \dim \mu \leq \min\left\{d, \dim_{L}(\Phi, p) \right\} $ always holds.
\end{proof}

We prove \autoref{thm:main-A} by following the approach of Rapaport~\cite{Rapaport2023}. The proof relies on two key ingredients: a Ledrappier-Young type formula and an entropy increase result. For the first ingredient, we establish a Ledrappier-Young type formula for certain disintegrations of self-affine measures in \autoref{thm:L-Y-formula}, a result may be of independent interest. Based on an argument inspired by ideas from~\cite{BaranyEtAl2016}, this formula reduces the general case to the one where  the entropy increase result can be applied.

The proof of the entropy increase result involves analyzing the multi-scale entropy of repeated self-convolutions of a measure with nonnegligible entropy, as well as the component measures of $ \mu $, along certain nonconformal partitions. In \cite{Rapaport2023}, the assumption that the linear parts of $ \Phi $ stay in a 1-dimensional subgroup is used to find minimal cut-sets $ \calU_{n}, \, n \in \bbN $ of $ \Lambda^{\ast} $ such that
\begin{equation}\label{eq:1Dim-result}
	A_{\varphi_{u}} \approx A_{\varphi_{v}} \mFor u, v \in \calU_{n},
\end{equation}
where $\approx $ means being entrywise comparable. These cut-sets are essential for estimating the asymptotic entropies of components of $ \mu $ within the desired error (see~\cite[Section 4]{Rapaport2023}). For each $ \mu^{\omega} $, there are natural partitions $ \calE^{\omega}_{n}, n \in\bbN $ (see \eqref{eq:def-E-omega}). Motivated by this and \eqref{eq:1Dim-result}, we consider the random measures $ \mu^{\omega}$ and establish the entropy increase result accordingly. However, difficulties arise because $ \mu^{\omega}$ is only dynamically self-affine (see \eqref{eq:dyn-selfaff}), and the partitions $ \calE^{\omega}_{n} $ depend on $ \omega $. To address this, we utilize the dynamics on $ (\Omega, \bfP) $ to prove appropriate modifications of the required lemmas. Based these lemmas, it is not difficult to adapt the arguments in \cite{Rapaport2023} to derive \autoref{thm:omega-e-increase}, a version of the entropy increase result for random measures.

\subsection{Structure of the article} In \autoref{sec:prelim}, we introduce the basics of the conditional entropies and disintegrations. \autoref{sec:L-Y} is devoted to proving the Ledrappier-Young type formula for random measures, thereby showing \autoref{thm:ExactDim-Intro}. In \autoref{sec:dis-linear}, we define the disintegrations with respect to the linear parts of the IFS. Sections \ref{sec:self-conv} and \ref{sec:components} are prepared for the entropy increase result which itself  is proved in \autoref{sec:EntropyIncrease}. Finally, \autoref{thm:main-A} is proved in \autoref{sec:pf-main-A}.

\subsection{Acknowledgement}

I would like to thank Ariel Rapaport for suggesting the problem, pointing out the useful references~\cite{Rapaport2023,SagliettiEtAl2018}, and providing helpful comments on an early version of this paper.


\section{Preliminaries}\label{sec:prelim}

In this section, we introduce the necessary notations and setup, present the basics of conditional information theory, and discuss key properties of specific disintegrations.

\subsection{Notations} 

Throughout this paper, the base of $ \log(\cdot) $ and $ \exp(\cdot) $ is 2. For $ n \in \bbN $, we define $ [n] = \{1 , \ldots, n\}$, with convention $ [0] = \emptyset $. The normalized counting measure on $ [n]$ is denoted by $ \#_{n} $, that is, $ \#_{n}(\{k\}) = 1/n $ for $ k \in [n]$. For a finite set $ \calE $,  we use $ \# \calE $ or $ \abs{\calE}$ to represent its cardinality.  By $ E \subsetneq F $ we mean that $ E $ is a proper subset of $ F $.

For a metric space $ X $, let $ \calB(X)$ denote the Borel $\sigma$-algebra on $ X $, and $ \calM(X) $ the set of all Borel probability measures on $ X $. By $ \calM_{c}(X)$ we denote the members of $ \calM(X)$ with compact support. For $ \theta \in \calM(X) $ and $ E \subset X $, the restriction of $ \theta $ to $ E $ is written as $ \theta|_{E} $, and the normalized restriction is $ \theta_{E} = \theta|_{E} / \theta(E) $ if $ \theta(E) > 0 $.

Following \cite[Section 2.1]{Rapaport2023}, we use the convenient notation $ \ll $. Given $ R_{1}, R_{2} \geq 1 $, we write $ R_{1} \ll R_{2}$ to indicate that $ R_{2} $ is large with respect to (w.r.t.)\ $ R_{1}$. Similarly, given $ \varepsilon_{1}, \varepsilon_{2} \in (0,1) $, we write $ R_{1} \ll \varepsilon_{1}^{-1} $, $ \varepsilon_{2}^{-1} \ll R_{2}$ and $ \varepsilon_{1}^{-1} \ll \varepsilon_{2}^{-1}$ to respectively indicate $ \varepsilon_{1}$ is small \wrt $ R_{1}$, $ R_{2}$ is large \wrt $ \varepsilon_{2} $, and $ \varepsilon_{2} $ is small w.r.t.\@ $ \varepsilon_{1}$. The relation $ \ll $ is clearly transitive. For example, the statement ``Let $ m \geq 1 $, $ \ell \geq L(m) \geq 1$, $ k \geq K(m,\ell) \geq 1 $ and $ \varepsilon \leq \varepsilon_{0}(m, \ell, k)$ be given.'' is equivalent to ``Let $ \varepsilon \in (0, 1)$ and $ m, \ell, k \geq 1$ be with $ m \ll \ell \ll k \ll \varepsilon^{-1}$.''

\subsection{The setup} We fix a diagonal affine IFS $ \Phi =  \{ \varphi_{i}(x) = A_{i}x + t_{i}\}_{i\in\Lambda}$ on $ \euclid $, where $ A_{i} = \diag(r_{i,1}, \ldots, r_{i,d})$ with $ r_{i,j} \in (-1,1) \setminus \{0\} $, and $ t_{i} = (t_{i,j})_{j=1}^{d} \in \euclid $. The associated self-affine set is $ K_{\Phi}$. We fix a probability vector $ p = (p_{i})_{i \in \Lambda}$, and $ \mu $ is the corresponding self-affine measure. Let $ \Pi \colon \Lambda^{\bbN} \to K_{\Phi} $ denote the coding map defined as in \eqref{eq:def-coding-map}. It is well known that $ \mu = \Pi \beta $, where $ \beta := p^{\bbN} $ is the Bernoulli measure on $ \Lambda^{\bbN} $.  For $ 1 \leq j \leq d $, the $ j $-th Lyapunov exponent is $ \chi_{j} := \sum_{i\in\Lambda} - p_{i} \log \abs{r_{i,j}} $. As explained in \autoref{rmk:LyExp-Distinct}, we always assume $ \chi_{1} < \cdots < \chi_{d} $. Without loss of generality, we also assume $ \diam( K_{\Phi}) \leq 1 $, where $ \diam(\cdot)$ denotes the diameter in Euclidean metric.

For $ i \in \Lambda$ and $ j \in [d]$, define $ \varphi_{i,j} \colon \bbR \to \bbR $ by $ \varphi_{i,j}(x) = r_{i,j}x + t_{i,j}$. For $ \emptyset \neq J \subset [d] $, the IFS induced by $ \Phi $ on $ \bbR^{J} $ is defined as
\begin{equation}\label{eq:def-Phi-J}
	\Phi_{J} = \{ \varphi_{i,J}\}_{i\in\Lambda}, \text{ where }	\varphi_{i,J}\left( (x_{j})_{j\in J} \right) = (\varphi_{i,j}(x_{j}))_{j\in J} \text{ for } i \in \Lambda.
\end{equation}
For $ 1 \leq j \leq d $, we write $ \Phi_{j}$ in place of $ \Phi_{[j]}$. It follows that $ \Phi = \Phi_{[d]} $ and $ \varphi_{i} = \varphi_{i,[d]}$ for $ i \in \Lambda $.

The collection of all finite words over $ \Lambda $ is denoted by $ \Lambda^{\ast}$, including the empty word $ \varnothing $. Write $ \abs{I} := n $ if $ I \in \Lambda^{n} $ and $ \abs{\varnothing} := 0 $. For $ x = (x_{i})_{i=1}^{\infty} \in \Lambda^{\bbN} $ and $ n \in \bbN $, let $ x|n = x_{1}\cdots x_{n} $ and $ x|0 = \varnothing $. For $ I \in \Lambda^{\ast}$, the cylinder set is $ [I] :=\{x\in \Lambda^{\bbN} \colon x|\abs{I} = I \}$ . For $ I = i_{1} \cdots i_{n} \in \Lambda^{n}$ and $ 1 \leq j \leq d $, define
\begin{equation}\label{eq:def-AIj}
	\varphi_{I} = \varphi_{i_{1}} \circ \cdots \varphi_{i_{n}},\quad A^{I} = A_{i_{1}} \cdots A_{i_{n}},\quad A_{j}^{I} = r_{i_{1}, j} \cdots r_{i_{n}, j},
\end{equation}
and
\begin{equation*}
	\lambda^{I}_{j} := \Abs{A_{j}^{I}} \mAnd \chi^{I}_{j} := - \log \lambda^{I}_{j}.
\end{equation*}
Let $ \{e_{1}, \ldots, e_{d}\}$ be the standard basis of $ \euclid[d]$. For $ J \subset [d]$, let $ \pi_{J} $ denote the orthogonal projection onto $ \mathrm{span} \{e_{j}\}_{j\in J} $, that is,
\begin{equation*}
	\pi_{J}(x) = \sum_{j\in J} \innp{e_{j}}{x} e_{j} \mFor x \in \euclid,
\end{equation*}
where $ \innp{\cdot}{\cdot}$ is the standard inner product on $ \euclid $. In particular, $ \pi_{\emptyset} $ is the zero map and $ \pi_{[d]}$ is the identity map on $ \euclid $.

\subsection{Conditional expectation, information and entropy} \label{subsec:cond-e}

Let $ (X, \calB, \theta)$ be a probability space. For a sub-$\sigma$-algebra $ \calF $ of $\calB$, the \textit{conditional expectation} of an integrable function $ f $ given $ \calF $ is denoted by $ \Eof{\theta, f \mid \calF}$. For a countable ($ \calB $-measurable) partition $ \xi $ of $ X $, the \textit{conditional information} of $ \xi $ given $ \calF$ is defined as
\begin{equation}\label{eq:def-Iof}
	\Iof{\theta, \xi \mid \calF} = \sum_{A \in \xi} - \indicator{A} \log \Eof{\theta, \indicator{A} \mid \calF},
\end{equation}
where $ \indicator{S}$ denotes the indicator function of a set $ S $. The \textit{conditional entropy} of $ \xi $ given $ \calF $ is 
\begin{equation}\label{eq:def-Hof}
	\Hof{\theta, \xi \mid \calF } := \int \Iof{\theta, \xi \mid \calF } \, \diff \theta = \int \sum_{A\in\xi} - \Eof{\theta, \indicator{A} \mid \calF} \log \Eof{\theta, \indicator{A} \mid \calF} \, \diff \theta.
\end{equation}
If $ \calF = \calN $, the trivial $\sigma$-algebra consisting of sets of  $\theta$\nobreakdash-measure 0 or 1, the above quantities reduce to their unconditional counterparts:
\begin{equation*}
	\Iof{\theta, \xi} = \Iof{\theta, \xi \mid \calN } \mAnd \Hof{\theta, \xi} = \Hof{\theta, \xi \mid \calN }.
\end{equation*}
For $ S \subset \calB $, let $ \wh{S}$ denote the $\sigma$-algebra generated by $ S $. Given a countable partition $ \eta $, we write
\begin{equation}
	\Iof{\theta, \xi\mid \eta} = \Iof{\theta, \xi \mid \wh{\eta}} \mAnd \Hof{\theta, \xi \mid \eta } = \Hof{\theta, \xi \mid \wh{\eta}}.
\end{equation}
In this case, the conditional entropy satisfies
\begin{equation*}
	\Hof{\theta, \xi \mid \eta } = \sum_{A\in\eta} \theta(A) \cdot \Hof{\theta_{A}, \xi },
\end{equation*}
where $ \theta_{A} := \theta(A)^{-1} \theta|_{A} $ for $ A \in \eta $ with $ \theta(A) > 0 $.

The following lemma summarizes key identities and properties of conditional information; see  \cite{Parry1981,Walters1982} for details. For countable partitions $ \eta_{1}, \ldots, \eta_{n} $, let $ \eta_{1} \vee \cdots \vee \eta_{n} = \vee_{i=1}^{n} \eta_{i}  = \left\{ \intxn_{i=1}^{n} A_{i} \colon A_{i} \in \eta_{i}, \, 1 \leq i \leq n\right\}$. For $ \sigma$-algebras $ \calF_{1}, \calF_{2}, \ldots$, let $ \calF_{1} \vee \calF_{2} \vee \cdots $ or $ \vee_{i} \calF_{i} $ denote the $\sigma$-algebra generated by $\union_{i}\calF_{i}$. Below we take the convention $ 0 / 0 = 0 $.

\begin{lemma}\label{lem:list-identities}
	Let $ T $ be a measurable map from a separable probability space $ (X, \calB, \theta)$ to another measurable space $ (Y, \calB')$. Let $ A \in \calB $. Let $ \xi, \eta, \zeta $ be countable partitions of $ X $, and let $ \calE $ be a countable partition of $ Y $, such that $ \Hof{\theta, \xi }, \Hof{\theta,\eta}, \Hof{\theta, \zeta}, \Hof{T\theta,\calE} < \infty $. Let $ \calF, \calF_{1},\calF_{2},\ldots $ be sub-$\sigma$-algebras of $ \calB $, and let $ \calG $ be a sub-$\sigma$-algebra of $ \calB'$. Then the following hold.
	
	\begin{enumerate}[(i)]
		\item \label{itm:list-pushfoward-E}  $ \Eof{ T\theta, g \mid \calG } \circ T = \Eof{ \theta, g \circ T \mid T^{-1} \calG }$ for $ g \in L^{1}(Y,\calB', T\theta)$.
		
		\item \label{itm:list-pushfoward-I} $ \Iof{T\theta, \calE \mid \calG} \circ T = \Iof{\theta, T^{-1}\calE \mid T^{-1} \calG} $.
		
		\item \label{itm:list-pushfoward-H} $ \Hof{T\theta, \calE \mid \calG} = \Hof{\theta, T^{-1}\calE \mid T^{-1} \calG} $.
		
		\item \label{itm:ChainRule-I}  $ \Iof{\theta, \xi \vee \eta \mid \calF } = \Iof{ \theta, \xi \mid \calF} + \Iof{\theta, \eta \mid \calF \vee \wh{\xi}} $.
		
		\item \label{itm:ChainRule-H} $ \Hof{\theta, \xi \vee \eta \mid \calF } = \Hof{ \theta, \xi \mid \calF} + \Hof{\theta, \eta \mid \calF \vee \wh{\xi}} $.
		
		\item \label{itm:condindp-exp} If $ \theta(A\intxn F_{1} \intxn F_{2}) /\theta(F_{1}\intxn F_{2}) = \theta(A\intxn F_{1}) /\theta(F_{1}) $ for $ F_{1} \in \calF, F_{2} \in \calF_{2}$, then
		\begin{equation*}
			\Eof{ \theta, \indicator{A} \mid \calF_{1} \vee \calF_{2}} = \Eof{ \theta, \indicator{A} \mid \calF_{1} }.
		\end{equation*}
		
		\item \label{itm:condindp-info} If  $ \theta(A\intxn F_{1} \intxn F_{2}) /\theta(F_{1}\intxn F_{2}) = \theta(A\intxn F_{1}) /\theta(F_{1}) $ for $ A \in \xi, F_{1} \in \calF, F_{2} \in \calF_{2}$, then
		\begin{equation*}
			\Iof{\theta, \xi \mid \calF_{1} \vee \calF_{2}} = \Iof{\theta, \xi \mid \calF_{1}} \mAnd \Hof{\theta, \xi \mid \calF_{1} \vee \calF_{2}} = \Hof{\theta, \xi \mid \calF_{1}}.
		\end{equation*} 
		
		\item \label{itm:increasing}
		If $ \calF_{n} \subset \calF_{n+1} $ for $ n \in \bbN $ and $ \calF_{n} \uparrow \calF $, then $ \sup_{n} \Iof{ \theta, \xi \mid \calF_{n}} \in L^{1}(\theta)$, and $ \Iof{ \theta, \xi, \calF_{n}}$ converges $ \theta $ a.e.\@ and in $ L^{1}(\theta) $ to $ \Iof{\theta, \xi\mid \calF}$. In particular, $ \lim_{n\to\infty} \Hof{\theta, \xi \mid \calF_{n}} = \Hof{\theta, \xi \mid \calF}$.
	\end{enumerate}
\end{lemma}

Next, we present several useful inequalities for estimating conditional entropy. For partitions $ \xi$ and $\eta $, we say $ \eta $ refines $ \xi $, denoted by $ \xi \prec \eta $, if each member of $ \eta $ is a subset of some member of $ \xi $.

\begin{lemma}\label{lem:list-ests}
	Let $ (X, \calB) $ be a measurable space, and let $ \theta, \theta_{1}, \ldots, \theta_{n}$ be probability measures on $(X, \calB)$. Let $\xi, \eta$ be countable partitions of $X$, and let $ \calF_{1},\calF_{2}$ be sub-$\sigma$-algebras of $\calB$. Then the following hold.
	\begin{enumerate}[(i)]
		\item \label{itm:est-triv-UB} $ \Hof{\theta, \xi} \leq \log \# \{ A \in \xi \colon \theta(A) > 0 \} $.
		
		\item \label{itm:est-monotone} If $ \xi \prec \eta $ and $ \calF_{1} \subset \calF_{2} $, then $ \Hof{\theta, \xi \mid \calF_{2}} \leq \Hof{\theta, \xi \mid \calF_{1}} \leq \Hof{\theta,\eta \mid \calF_{1}}$.
		
		\item \label{itm:concav-aconvex} If $ q = (q_{i})_{i=1}^{n}$ is a probability vector and $ \theta = \sum_{i=1}^{n} q_{i} \theta_{i} $, then
			\begin{equation*}
			 \sum_{i=1}^{n} q_{i} \Hof{\theta_{i}, \xi \mid \eta } \leq \Hof{\theta, \xi\mid \eta}  \leq \sum_{i=1}^{n} q_{i} \Hof{\theta_{i}, \xi \mid \eta } + H(q). 
			\end{equation*}
			
		\item \label{itm:commensure} Given $ C \geq 1$, we say that $ \xi $ and $ \eta$ are $ C $-commensurable if for each $ A \in \xi$ and $ B \in \eta $,
		\begin{equation*}
			\#\{ A' \in \xi \colon A' \intxn B \neq \emptyset \} \leq C \mAnd	 \#\{ B' \in \xi \colon B' \intxn A \neq \emptyset \} \leq C.
		\end{equation*}
		If $ \xi $ and $ \eta$ are $ C $-commensurable, then $ \Abs{ \Hof{\theta, \xi} - \Hof{\theta, \eta} } \leq \log C $.
		
	\end{enumerate}
	
\end{lemma}

\subsection{Conditional measures and some disintegrations} \label{subsec:some-disint}

We begin with a foundational result from Rohlin's theory of conditional measures; for further details, refer to \cite{Rohlin1952, EinsiedlerWard2011}.

\begin{theorem}[{Rohlin~\cite{Rohlin1952}}] \label{thm:rohlin}
	Let $ X , Y $ be Euclidean spaces or product spaces of countably many finite sets. Let $ \eta $ be a partition induced by a Borel measurable map $\pi \colon X \to Y $, that is, $ \eta = \{ \pi^{-1}(y) \colon y\in Y \}$. Let $ \theta $ be a Borel probability measure on $ X $. Then for $ \theta \aev x $ there exists a probability measure $ \theta^{\eta}_{x}$ supported on $ \eta(x)$. These measures are uniquely determined up zero $ \theta $-measure by the properties: if $ A \subset X $ is Borel measurable, then $ x \mapsto \theta^{\eta}_{x}(A)$ is $ \wh{\eta}$-measurable, and $ \theta(A) = \int \theta^{\eta}_{x}(A) \, \diff \theta(x)$. This means $ \theta = \int \theta^{\eta}_{x} \, \diff \theta(x) $ in the sense that $ \int\int f(y) \, \diff \theta^{\eta}_{x}(y) \diff \theta(x) $ for $ f \in L^{1}(X, \calB(X), \theta )$.
\end{theorem}

The family of measures $\{\theta^{\eta}_{x}\}_{x\in X} $ is called the \textit{system of conditional measures of $ \theta $ associated with $ \eta $} or the \textit{disintegration of $ \theta $ with respect to $ \pi $}.

Next, we introduce certain disintegrations and present some of their properties. Fix $ N \in \bbN $. Let $ \Gamma $ be a partition of $ \Lambda^{\bbN} $ such that for $ x ,y \in \Lambda^{\bbN}$, $ x|N = y|N $ implies $ \Gamma(x) = \Gamma(y) $. Set $ T = \sigma^{N} $ and $ \calA = \vee_{i=0}^{\infty} T^{-i} \Gamma $. Define the quotient space $ \Omega := \Lambda^{\bbN} / \calA \cong  \Gamma^{\bbN}$. Let $ \bfP $ be the Bernoulli measure on $ \Omega = \Gamma^{\bbN} $ with marginal $ (\beta(\omega_{1}))_{\omega_{1}\in\Gamma}$. Specifically, for $ \omega_{1} \cdots \omega_{n} \in \Gamma^{n}, n \geq 1 $,
\begin{equation}\label{eq:def-bfP}
	\bfP([\omega_{1}\cdots\omega_{n}]) = \prod_{k=1}^{n}\beta(\omega_{k}) = \beta \left\{ x \in \Lambda^{\bbN} \colon \calA(x) \in [\omega_{1}\cdots\omega_{n}]\right\}.
\end{equation}
This shows that $ \bfP = \beta \circ \calA^{-1} $, that is, $ \bfP $ is the pushforward of $ \beta $ under $ \calA $. Here, we slightly abuse the notation by using $ \calA(x)$ to denote both a set in $ \Lambda^{\bbN} $ and a sequence in $ \Omega = \Gamma^{\bbN}$.  

For $\omega_{1} \in \Gamma$, define a measure $ p^{\omega_{1}} $ on $ \Lambda^{\bbN}$ by $ p^{\omega_{1}} := \beta_{\omega_{1}} $ if $ \beta(\omega_{1}) > 0 $, and let $ p^{\omega_{1}} $ be the zero measure if $ \beta(\omega_{1}) = 0 $. For $ \omega = (\omega_{n})_{n=1}^{\infty} \in \Omega $, define a product measure $ \beta^{\omega} $ on $ \Lambda^{\bbN} $ via the identification $ \Lambda^{\bbN} = (\Lambda^{N})^{\bbN}$ as
\begin{equation}\label{eq:def-beta-omega}
	\beta^{\omega}([I]) = \prod_{k=1}^{n} p^{\omega_{k}}([I_{k}]) \mFor I = I_{1} \cdots I_{n} \in (\Lambda^{N})^{n}, n \geq 1.
\end{equation}
Then $\beta^{\omega}$ is supported on $ \calA(x) $ whenever $ \omega = \calA(x)$ for some $ x \in \Lambda^{\bbN}$. On the other hand, let $ \{ \beta^{\calA}_{x}\}_{x\in\Lambda^{\bbN}} $ be the disintegration of $ \beta$ with respect to $ \calA $. It follows from \autoref{thm:rohlin}, $ \wh{\calA}  = (\vee_{i=0}^{n-1}T^{-i}\wh{\Gamma}) \vee \wh{\calA} $ and \autoref{lem:list-identities}\ref{itm:condindp-exp} that for $ \beta \aev x $ and $  I = I_{1} \cdots I_{n} \in (\Lambda^{N})^{n}, n \geq 1 $,
\begin{align*}
	\beta^{\calA}_{x}([I]) & = \Eof{\beta, \indicator{[I]} \mid \wh{\calA} }(x) = \Eof{ \beta, \indicator{[I]} \mid \vee_{i=0}^{n-1} T^{-i}\Gamma }(x) \\ 
	& = \sum_{A \in \vee_{i=0}^{n-1} T^{-i}\Gamma} \indicator{A}(x) \frac{\beta([I] \intxn A)}{\beta(A)} \\
	& = \sum_{(\omega_{k})_{k=1}^{n} \in \Gamma^{n} } \indicator{[\omega_{1}\cdots\omega_{n}]}(\calA(x)) \prod_{k=1}^{n} p^{\omega_{k}}([I_{k}]) \\
	& = \beta^{\calA(x)}([I]),
\end{align*}
where the last equality is by \eqref{eq:def-beta-omega}.
Hence $ \beta^{\calA}_{x} = \beta^{\calA(x)}$ for $ \beta \aev x $. Combining this, \autoref{thm:rohlin} and $ \bfP = \beta \circ \calA^{-1} $, we obtain
\begin{equation}\label{eq:Sym-Disint}
	\beta = \int_{\Lambda^{\bbN}} \beta^{\calA}_{x} \, \diff \beta(x) = \int_{\Lambda^{\bbN}} \beta^{\calA(x)} \, \diff \beta(x) = \int_{\Omega} \beta^{\omega} \, \diff \bfP(\omega).
\end{equation}

Recall the coding map $ \Pi $ from \eqref{eq:def-coding-map}. For $ \omega \in \Omega $, define $ \mu^{\omega} := \Pi \beta^{\omega} $. Applying $ \Pi $ to \eqref{eq:Sym-Disint} yields a disintegration of $ \mu $ as
\begin{equation}
	\mu = \int_{\Omega} \mu^{\omega} \, \diff \bfP(\omega).
\end{equation}

For $ \omega \in \Omega $, the random measure $ \mu^{\omega}$ satisfies the \textit{dynamical self-affinity}. By abuse of notation, let $ T $ be the shift map on $ \Omega $, defined by $ T((\omega_{n})_{n=1}^{\infty}) = (\omega_{n+1})_{n=1}^{\infty}$. Using \eqref{eq:def-beta-omega}, we have, for $ \omega \in \Omega $,
\begin{equation}\label{eq:Tbeta-omega}
	T\beta^{\omega} = \beta^{T\omega},
\end{equation}
and so for $ u \in \Lambda^{N} $,
\begin{equation}\label{eq:push-restrict}
	T(\beta^{\omega}|_{[u]}) = \beta^{\omega}([u]) \beta^{T\omega}.
\end{equation}
From \eqref{eq:def-coding-map} it follows that for $ u \in \Lambda^{\ast} $,
\begin{equation}\label{eq:coding-property}
	\varphi_{u} \circ \Pi \circ \sigma^{\abs{u}} = \Pi \quad \text{ on } [u].
\end{equation}
Thus, $ \mu^{\omega}$ satisfies the dynamical self-affinity:
\begin{equation}\label{eq:dyn-selfaff-general}
	\begin{aligned}
	\mu^{\omega} &  = \Pi\beta^{\omega} = \sum_{u \in \Lambda^{N}} \Pi \beta^{\omega}|_{[u]} \\ 
	& = \sum_{u \in \Lambda^{N}} (\varphi_{u} \Pi T) \beta^{\omega}|_{[u]} & \by{\eqref{eq:coding-property}} \\
	& = \sum_{u \in \Lambda^{N}} (\varphi_{u} \Pi )\left( \beta^{\omega}([u]) \beta^{T\omega} \right) & \by{\eqref{eq:push-restrict}} \\
	& = \sum_{u \in \Lambda^{N}} \beta^{\omega}([u]) \cdot \varphi_{u} \mu^{T\omega}. & \by{$ \mu^{T\omega} = \Pi \beta^{T\omega} $}
\end{aligned}
\end{equation}

\section{Exact dimensionality for disintegrations}
\label{sec:L-Y}

In this section, we establish the exact dimensionality of certain random measures and show that their dimension satisfies a Ledrappier-Young type formula. To prove these results, we adapt the approach from deterministic case of Feng~\cite{Feng2023a}.

For $ J \subset [d] $, define the partition  $ \xi_{J} $ of $ \Lambda^{\bbN} $ as
\begin{equation}\label{eq:def-xi-J}
	 \xi_{J}(x) = \xi_{J}(y) \quad\text{ if and only if }\quad \pi_{J}\Pi(x) = \pi_{J} \Pi(y) \mFor x, y \in \Lambda^{\bbN}.
\end{equation}
Note that $ \wh{\xi_{J}} = \Pi^{-1} \pi_{J}^{-1} \calB(\euclid) \pmod 0 $.

\begin{theorem}\label{thm:exact-dim-general}
	 Let $ N \in \bbN $. Let $ \calC $ be a partition of $\Lambda^{\bbN}$ such that for $ x, y \in \Lambda^{\bbN}$, $ \calC(x) = \calC(y)$ implies $ \varphi_{x|N} = \varphi_{y|N}$. Let $ \Gamma $ be a partition of $ \Lambda^{\bbN} $ such that for $ x, y \in \Lambda^{\bbN}$, $ x|N = y|N $ implies $ \Gamma(x) = \Gamma(y) $. Set $ T = \sigma^{N}$ and $ \calA = \vee_{i=0}^{\infty} T^{-i} \Gamma $. Let $ 1\leq j_{1} < \cdots < j_{s} \leq d $ and write $ J  = \{j_{1}, \ldots, j_{s}\}$. For $ 0 \leq b \leq s $, set $ J_{b} = \{ j_{1}, \ldots, j_{b}\}$. Then for $ \beta \aev y $, $ \beta^{\calA}_{y} \aev x $ and $ 0 \leq k \leq l \leq s $, the measure $ \pi_{J_{l}} \Pi \beta^{\calA,\xi_{J_{k}}}_{y,x} := \pi_{J_{l}} \Pi (\beta^{\calA}_{y})^{\xi_{J_{k}}}_{x} $ is exact dimensional with
	\begin{equation}\label{eq:LY-general}
		\dim \pi_{J_{l}} \Pi \beta^{\calA,\xi_{J_{k}}}_{y,x}  = \sum_{b = k + 1}^{l}  \frac{ \HcalA{J_{b-1}} - \HcalA{J_{b}} }{ \chi_{j_{b}}},
	\end{equation}
	where for $ I \subset [d] $,
	\begin{equation}\label{eq:def-HcalA}
		\HcalA{I} = \frac{1}{N} \Hof{ \beta, \calC \mid \wh{\calA} \vee \wh{\xi_{I}} }.
	\end{equation}
\end{theorem}

\autoref{thm:exact-dim-general} has following consequence which is a general and detailed version of \autoref{thm:ExactDim-Intro}.

\begin{theorem}\label{thm:L-Y-formula}
For $ n \in \bbN $, let $ \calC_{n}$ be the partition of $ \Lambda^{\bbN} $ defined by $ \calC_{n}(x) =  \calC_{n}(y)$ if and only if $ \varphi_{x|n} = \varphi_{y|n}$ for $ x, y\in \Lambda^{\bbN}$. Let $ N \in \bbN $. Let $ \Gamma $ be a partition of $ \Lambda^{\bbN} $ such that for $ x, y \in \Lambda^{\bbN}$, $ x | N = y | N $ implies $ \Gamma(x) = \Gamma(y)$. Set $ \calA = \vee_{i=0}^{\infty} \sigma^{-iN} \Gamma $. Let $ 1\leq j_{1} < \cdots < j_{s} \leq d $ and write $ J  = \{j_{1}, \ldots, j_{s}\}$. For $ 0 \leq b \leq s $, set $ J_{b} = \{ j_{1}, \ldots, j_{b}\}$. Then for $ \beta \aev y $, the measure $ \pi_{J}\Pi\beta^{\calA}_{y}  $ is exact dimensional with dimension given by
\begin{equation}\label{eq:LY-formula}
	\dim \pi_{J} \calA = \sum_{b = 1}^{s}  \frac{ \hcalCA{J_{b-1}} - \hcalCA{J_{b}} }{ \chi_{j_{b}}},
\end{equation}
where for $ I \subset [d] $,
\begin{equation}\label{eq:def-hCalCA}
	\hcalCA{I} = \lim_{n\to\infty} \frac{1}{nN} \Hof{ \beta, \calC_{nN} \mid \wh{\calA} \vee \wh{\xi_{I}} } = \inf_{n} \frac{1}{nN} \Hof{ \beta, \calC_{nN} \mid \wh{\calA} \vee \wh{\xi_{I}} },
\end{equation}
and $ \hcalCA{J_{b-1}} - \hcalCA{J_{b}} \leq \chi_{j_{b}}$ for $ 1 \leq b \leq s $.
\end{theorem}

We write $ \dim \calA := \dim \pi_{[d]} \calA $ by convention.

\begin{proof}[Proof of \autoref{thm:L-Y-formula} assuming \autoref{thm:exact-dim-general}]
	For $ n \in \bbN$ write $ \Gamma_{n} = \vee_{i=0}^{n-1} \sigma^{-i N} \Gamma $. Note that $ \calA = \vee_{i=0}^{\infty} \sigma^{-i(nN)}\Gamma_{n} $ for all $ n \in \bbN $. Applying \autoref{thm:exact-dim-general} with $ nN, \calC_{nN}, \Gamma_{n}$ in place of $ N, \calC, \Gamma$, and taking $ k = 0,\, l = s,\, J = J_{s} $, it follows that for $ \beta \aev y $, the measure $ \pi_{J}\Pi\beta^{\calA}_{y}  $ is exact dimensional with 
	\begin{equation*}
		\dim  \pi_{J}\Pi\beta^{\calA}_{y}  = \sum_{b = 1}^{s}  \frac{ \HcalCAn{J_{b-1}} - \HcalCAn{J_{b}} }{ \chi_{j_{b}}} \quad\text{ for all } n \in \bbN,
	\end{equation*}
	where for $ I \subset [d]$,
	\begin{equation*}
		\HcalCAn{I} = \frac{1}{nN} \Hof{ \beta, \calC_{nN} \mid \wh{\calA} \vee \wh{\xi_{I}}}.
	\end{equation*}
	
	For $ 1\leq b \leq s$, applying \autoref{thm:exact-dim-general} with $ k=b-1, l = b$, we have
	\begin{equation*}
		\HcalCAn{J_{b-1}} - \HcalCAn{J_{b}} \leq \chi_{j_{b}},
	\end{equation*}
	 since $ \pi_{J_{b}} \Pi \beta^{\calA,\xi_{J_{b-1}}}_{y,x} $ is supported on $ \Pi(x) +  \pi_{j_{b}} \euclid $ for $ \beta \aev y$ and $ \beta^{\calA}_{y}\aev x $.
	
	 For $ m, n \in \bbN $, it follows from $ \calC_{(m+n)N} \prec \calC_{mN} \vee T^{-m}\calC_{nN} $, $ \wh{\calA} = \left(\vee_{i=0}^{m-1}T^{-i} \wh{\Gamma}\right) \vee T^{-m} \wh{\calA} $, Lemmas \ref{lem:list-identities}, \ref{lem:list-ests} and \ref{lem:set-relations}\ref{itm:set-shift} that,
\begin{equation}\label{eq:subadd}
		\begin{aligned}
		& \hspace{-2em}\Hof{ \beta, \calC_{(m+n)N} \mid \wh{\calA} \vee \wh{\xi_{I}}} \\
		& \leq \Hof{ \beta, \calC_{mN} \vee T^{-m}\calC_{nN} \mid \wh{\calA} \vee \wh{\xi_{I}}}  \\ 
		& = \Hof{ \beta, \calC_{mN} \mid \wh{\calA} \vee \wh{\xi_{I}}} + \Hof{ \beta, T^{-m}\calC_{nN} \mid \wh{\calA} \vee \wh{\xi_{I}} \vee \wh{\calC_{mN}}} \\
		& = \Hof{ \beta, \calC_{mN} \mid \wh{\calA} \vee \wh{\xi_{I}}} + \Hof{ \beta, T^{-m}\calC_{nN} \mid \left(\vee_{i=0}^{m-1}T^{-i} \wh{\Gamma}\right) \vee T^{-m} \wh{\calA} \vee T^{-m} \wh{\xi_{I}} \vee \wh{\calC_{mN}}} \\
		& \leq \Hof{ \beta, \calC_{mN} \mid \wh{\calA} \vee \wh{\xi_{I}}} + \Hof{ \beta, T^{-m}\calC_{nN} \mid T^{-m}\left( \wh{\calA} \vee \wh{\xi_{I}} \right)} \\
		& = \Hof{ \beta, \calC_{mN} \mid \wh{\calA} \vee \wh{\xi_{I}}} + \Hof{ \beta,\calC_{nN} \mid  \wh{\calA} \vee \wh{\xi_{I}} }.
	\end{aligned}
\end{equation}
	This shows the subadditivity and justifies the limit in \eqref{eq:def-hCalCA}. The proof is finished by letting $ n \to \infty $ in the above equations.
\end{proof}


The rest of this section is devoted to the proof of \autoref{thm:exact-dim-general}. For the remainder of this section, we fix $ N, \calC, \Gamma, T, \calA $ as in \autoref{thm:exact-dim-general}. Without loss of generality, we assume $ J = [d] $, since the general case can be reduced to this one by considering the IFS $ \Phi_{J} $ as defined in \eqref{eq:def-Phi-J}.

\subsection{The Peyri\`ere measure}
We begin by introducing a useful measure on $ \Omega \times \Lambda^{\bbN} $. Recall the definitions of $ \Omega, \bfP, \beta^{\omega}, \mu^{\omega} $ from \autoref{subsec:some-disint}. Define a Borel probability measure $ \bfQ $ on $ \Omega \times \Lambda^{\bbN} $ by
\begin{equation}\label{eq:def-bfQ}
	\int_{\Omega\times\Lambda^{\bbN}} f(\omega,x) \, \diff \bfQ = \int_{\Omega} \int_{\Lambda^{\bbN}} f(\omega, x) \, \diff \beta^{\omega}(x) \diff \bfP(\omega), 
\end{equation}
for every bounded Borel measurable function $ f $ on $ \Omega\times\Lambda^{\bbN}$. Under this definition, the phrase ``for $ \bfQ\aev (\omega,  x)$'' is equivalent to ``for $ \bfP \aev \omega$ and $ \beta^{\omega} \aev x$''. The measure $ \bfQ $ serves a role analogous to the Peyri\`ere measure used in~\cite{FalconerJin2014}. Next, define a transformation on $ \Omega \times \Lambda^{\bbN} $ by
\begin{equation*}
	T(\omega, x) := (T\omega, Tx),
\end{equation*}
for $ (\omega, x) \in \Omega \times \Lambda^{\bbN} $.

\begin{lemma}
	The system $ (\Omega\times\Lambda^{\bbN}, \bfQ, T)  $ is measure-preserving and mixing.
\end{lemma}

\begin{proof}
	For $ A \in \calB(\Omega \times \Lambda^{\bbN} )$,
	\begin{align*}
		\bfQ(T^{-1}A) & = \int_{\Omega} \int_{\Lambda^{\bbN}} \indicator{A}(T\omega,Tx) \, \diff \beta^{\omega}(x)\diff \bfP(\omega) & \by{\eqref{eq:def-bfQ}}\\ & = \int_{\Omega} \int_{\Lambda^{\bbN}} \indicator{A}(T\omega,x) \, \diff \beta^{T\omega}(x)\diff \bfP(\omega) & \by{\eqref{eq:Tbeta-omega}} \\
		& = \int_{\Omega} \int_{\Lambda^{\bbN}} \indicator{A}(\omega,x) \, \diff \beta^{\omega}(x)\diff \bfP(\omega) & \by{$T\bfP = \bfP $} \\
		& = \bfQ(A). & \by{\eqref{eq:def-bfQ}}
	\end{align*}
	Thus $ \bfQ $ is $T$-invariant.
	
	For $ U \times I \in \Gamma^{m_{1}}\times (\Lambda^{N})^{m_{1}} $, $ V \times J \in \Gamma^{m_{2}} \times (\Lambda^{N})^{m_{2}}, m_{1}, m_{2} \geq 1  $ and $ n \geq 2Nm_{1}$, we have
	\begin{align*}
		& \hspace{-2em}\bfQ\left(( [U] \times [I]) \intxn T^{-n}([V]\times[J])\right) \\ & = \bfQ\left( ([U] \intxn T^{-n}[V]) \times ([I] \intxn T^{-n}[J]) \right) \\
		& = \int_{[U] \intxn T^{-n}[V]} \beta^{\omega}([I] \intxn T^{-n}[J]) \, \diff \bfP(\omega)  & \by{\eqref{eq:def-bfQ}}\\
		& = \int_{[U] \intxn T^{-n}[V]} \beta^{\omega}([I]) \beta^{T^{n}\omega}([J]) \, \diff \bfP(\omega) & \by{\eqref{eq:def-beta-omega} and \eqref{eq:Tbeta-omega}} \\
		& = \int_{[U]} \beta^{\omega}([I]) \, \diff \bfP(\omega) \int_{[V]} \beta^{\omega}([J]) \, \diff \bfP(\omega) & \by{\eqref{eq:def-bfP}} \\
		& = \bfQ\left([U] \times [I] \right) \bfQ\left([V]\times[J]\right). & \by{\eqref{eq:def-bfQ}}  
	\end{align*}
	This implies that $ T $ is mixing with respect to $ \bfQ $.
\end{proof}

Below is a direct consequence of Birkhoff's ergodic theorem applied to $ (\Omega\times \Lambda^{\bbN}, \bfQ, T)$.

\begin{lemma}\label{lem:lambda-speed}
	For $ \bfQ \aev (\omega, x)$ and $ 1 \leq j \leq d $, $ \lim_{n\to\infty} - (1/n)\log \lambda^{x|nN}_{j} = N \chi_{j} $.
\end{lemma}


\subsection{Some measurable partitions}

In this subsection we explore the properties of $ \xi_{[j]}, \calA $ and their associated conditional measures.

For $ 0 \leq j \leq d $, we denote $ \xi_{j} = \xi_{[j]}$, $ \Pi_{j} = \pi_{[j]}\Pi $, and for $ x \in \Lambda^{\bbN}$, $ r > 0 $, define
\begin{equation*}
	\PiBall[j]{x}{r} = \left\{ y \in \Lambda^{\bbN} \colon \Abs{\Pi_{j}(x) - \Pi_{j}(y)} \leq r \right\} = \Pi_{j}^{-1} B\left( \Pi_{j} x, r \right).
\end{equation*}
For $ n \in \bbN $, let $ \calC^{n-1}_{0} := \vee^{n-1}_{i=0} T^{-i}\calC $.

We begin with a lemma connecting $ \xi_{j}$, $ \calC $ and $ \PiBall[j]{x}{r}$.

\begin{lemma}\label{lem:set-relations}
	For $ \bfQ \aev (\omega, x)$ and $ 1 \leq i \leq j \leq d $, the following holds.
	\begin{enumerate}[(i)]
			\item \label{itm:set-shift} $ \xi_{j}(x) \intxn \calC(x) = T^{-1}\xi_{j}(Tx) \intxn \calC(x) $, and so $ \xi_{j} \vee \calC = T^{-1} \xi_{j} \vee \calC $.
			
			\item  \label{itm:set-PiBall-shift} $ \xi_{j-1}(x) \intxn \PiBall[j]{x}{\lambda^{x|nN}_{j}} \intxn \calC(x) = T^{-1}\left(\xi_{j-1}(Tx)\intxn \PiBall[j]{Tx}{\lambda^{Tx|(n-1)N}_{j}}\right) \intxn \calC(x) $.
			
			\item \label{itm:SliceElipse} For $ \varepsilon \in (0, 1)$ and $ n \in \bbN $ with $ \varepsilon^{-1} \ll n $, $ \xi_{i-1}(x) \intxn \calC^{n-1}_{0}(x) \subset \PiBall[j]{x}{ \exp(-n(N\chi_{i}- \varepsilon)) } $.
	
	\end{enumerate}
\end{lemma}

\begin{proof}
	By \eqref{eq:def-coding-map},
	\begin{equation}\label{eq:coding-shift}
		\varphi_{x|nN}\left( \Pi(T^{n}x) \right) = \Pi(x) \mFor x\in \Lambda^{\bbN}, n \in \bbN.
	\end{equation}
	For $ x \in \Lambda^{\bbN} $, $ n \in \bbN $, $ a , b \in \euclid $ and $ J \subset [d]$, since $ A_{\varphi_{x|n}}$ is a diagonal matrix, we have
	\begin{equation}\label{eq:pi-phi-switch}
		\begin{aligned}
		\pi_{J} ( \varphi_{x|n}(a) - \varphi_{x|n}(b)) = \varphi_{x|n}(\pi_{J}a)- \varphi_{x|n}(\pi_{J}b).
 	\end{aligned}
	\end{equation}
	
	Then for $ y \in \calC(x)$, we have $ \varphi_{x|N} = \varphi_{y|N}$, and so
	\begin{align*}
		y \in \xi_{j}(x) & \iff \pi_{[j]}\Pi(x) = \pi_{[j]}\Pi(y) & \by{\eqref{eq:def-xi-J}} \\
		& \iff \pi_{[j]}\varphi_{x|N}\left( \Pi(Tx) \right) = \pi_{[j]}\varphi_{y|N}\left( \Pi(Ty) \right) & \by{\eqref{eq:coding-shift}}\\
		& \iff \pi_{[j]}\varphi_{x|N}\left( \Pi(Tx) \right) = \pi_{[j]}\varphi_{x|N}\left( \Pi(Ty) \right) & \by{$ \varphi_{x|N} = \varphi_{y|N}$} \\
		& \iff \varphi_{x|N}\left( \pi_{[j]} \Pi(Tx) \right) = \varphi_{x|N}\left( \pi_{[j]} \Pi(Ty) \right) & \by{\eqref{eq:pi-phi-switch}} \\
		& \iff  \pi_{[j]} \Pi(Tx)  =  \pi_{[j]} \Pi(Ty) & \by{$\varphi_{x|N}$ being invertible} \\
		& \iff y \in T^{-1} \xi_{j}(Tx).  & \by{\eqref{eq:def-xi-J}}
	\end{align*}
	This proves \ref{itm:set-shift}.
	
	For $ y \in \calC(x)$, we have $ \varphi_{x|N} = \varphi_{y|N}$, and so 
	\begin{align*}
		& y \in \xi_{j-1}(x) \intxn \PiBall[j]{x}{\lambda^{x|nN}_{j}}  \\
		& \iff \Abs{\pi_{[j]}\Pi(x)-\pi_{[j]}\Pi(y)} \leq \lambda^{x|nN}_{j} ,\, \pi_{[j-1]}\Pi(x) = \pi_{[j-1]}\Pi(y) \\
		& \iff \Abs{\pi_{j}\Pi(x)-\pi_{j}\Pi(y)} \leq \lambda^{x|nN}_{j} ,\, \pi_{[j-1]}\Pi(x) = \pi_{[j-1]}\Pi(y) \quad \by{$ \pi_{[j]} = \pi_{[j-1]} + \pi_{j}$}\\
		& \iff \Abs{\pi_{j}\varphi_{x|N}(\Pi(Tx))-\pi_{j}\varphi_{x|N}(\Pi(Ty))} \leq \lambda^{x|nN}_{j}, \hspace{3em} \by{\eqref{eq:coding-shift} and $ \varphi_{x|N} = \varphi_{y|N}$} \\ 
		& \hspace{10em} \pi_{[j-1]}\Pi(Tx) = \pi_{[j-1]}\Pi(Ty)  \hspace{3.6em}\by{\ref{itm:set-shift}} \\
		& \iff \lambda_{j}^{x|N}\Abs{\pi_{j}\Pi(Tx)-\pi_{j}\Pi(Ty)} \leq \lambda^{x|nN}_{j} ,\, \pi_{[j-1]}\Pi(Tx) = \pi_{[j-1]}\Pi(Ty) \quad\by{\eqref{eq:pi-phi-switch}} \\
		& \iff \Abs{\pi_{j}\Pi(Tx)-\pi_{j}\Pi(Ty)} \leq \lambda^{Tx|(n-1)N}_{j} , \, \pi_{[j-1]}\Pi(Tx) = \pi_{[j-1]}\Pi(Ty)\\
		& \iff \Abs{\pi_{[j]}\Pi(Tx)-\pi_{[j]}\Pi(Ty)} \leq \lambda^{Tx|(n-1)N}_{j} , \, \pi_{[j-1]}\Pi(Tx) = \pi_{[j-1]}\Pi(Ty)\\
		& \iff y \in T^{-1} \PiBall[j]{Tx}{\lambda^{Tx|(n-1)N}_{j}} \intxn T^{-1}\xi_{j-1}(Tx).
	\end{align*}
	This gives \ref{itm:set-PiBall-shift}.
	
	Finally, we prove \ref{itm:SliceElipse}. By \autoref{lem:lambda-speed} and $ \chi_{\ell} \geq \chi_{i}$, we have for $ \bfQ \aev (\omega, x) $ and $ i \leq \ell \leq j $,
\begin{equation}\label{eq:chi-ell-speed}
	\lambda^{x|nN}_{\ell} \leq \exp\left(-n(N\chi_{\ell}-\varepsilon/4)\right) \leq \exp\left(-n(N\chi_{i}-\varepsilon/2)\right).
\end{equation} 
Let $ y \in \calC_{0}^{n-1} \intxn \xi_{i-1}(x)$. Then $ \varphi_{y|nN} = \varphi_{x|nN}$ and $ \pi_{[i-1]}\Pi(x) = \pi_{[i-1]}\Pi(y)$. Hence
\begin{align*}
	& \hspace{-2em}\Abs{ \pi_{[j]} \Pi(x) - \pi_{[j]}\Pi(y)} \\ & = \Abs{ \sum_{\ell=i}^{j} \pi_{\ell} \Pi(x) - \pi_{\ell} \Pi(y) } & \by{$ \pi_{[i-1]}\Pi(x) = \pi_{[i-1]}\Pi(y)$} \\
	& = \Abs{ \sum_{\ell=i}^{j} \pi_{\ell} \left( \varphi_{x|nN}  \Pi(T^{n}x) -  \varphi_{x|nN} \Pi(T^{n}y) \right) } & \by{\eqref{eq:coding-shift} and $ \varphi_{y|nN} = \varphi_{x|nN}$} \\
	&  \leq  \sum_{\ell=i}^{j} \lambda^{x|nN}_{\ell}  & \by{$\diam(K_{\Phi}) \leq 1 $} \\
	& \leq \exp\left(-n(N\chi_{i}-\varepsilon)\right). & \by{\eqref{eq:chi-ell-speed}}
\end{align*}
This shows that $ y \in \PiBall[j]{x}{ \exp(-n(N\chi_{i}- \varepsilon))} $.
\end{proof}

Next, we establish the relation between the conditional measures $ \beta^{\omega,\xi_{j}}_{x} := (\beta^{\omega})^{\xi_{j}}_{x}$ and $ \beta^{T\omega,\xi_{j}}_{Tx} $.

\begin{lemma}\label{lem:measure-relations}
	For $ \bfQ \aev (\omega, x)$, $ 1 \leq j \leq d $ and $ A \subset \calB(\Lambda^{\bbN})$,
	\begin{equation*}
		\beta^{T\omega,\xi_{j}}_{Tx}(A) =  \frac{\beta_{x}^{\omega,\xi_{j}}(T^{-1}A \intxn \calC(x) )  }{ \beta^{\omega, \xi_{j}}_{x}( \calC(x))}. 
	\end{equation*}
\end{lemma}

\begin{proof}
	First we show that
	\begin{equation}\label{eq:mass-rel-first}
		\begin{aligned}
		\beta^{\omega,T^{-1}\xi_{j} \vee \calC }_{x}(T^{-1}A) & = \Eof{ \beta^{\omega}, \indicator{T^{-1}A} \mid T^{-1}\wh{\xi_{j}} \vee \calC }(x) & \by{\autoref{thm:rohlin}} \\ 
		 & = \Eof{ \beta^{\omega}, \indicator{T^{-1}A} \mid T^{-1}\wh{\xi_{j}} }(x) & \by{\autoref{lem:list-identities}\ref{itm:condindp-exp}} \\
		 & = \Eof{\beta^{T\omega}, \indicator{A} \mid \wh{\xi_{j}} }(Tx)  & \by{\autoref{lem:list-identities}\ref{itm:list-pushfoward-E} and \eqref{eq:Tbeta-omega}} \\
		 & =  \beta^{T\omega,\xi_{j}}_{Tx}(A). & \by{\autoref{thm:rohlin}}
	\end{aligned}
	\end{equation}
	
	By \autoref{thm:rohlin}, for $ \beta \aev x $ we define
	\begin{equation*}
		\nu_{x}(T^{-1}A) = \frac{\beta_{x}^{\omega,\xi_{j}}(T^{-1}A \intxn \calC(x) )  }{ \beta^{\omega, \xi_{j}}_{x}( \calC(x))} = \sum_{B\in \calC} \indicator{B}(x) \cdot h_{B}(x),
	\end{equation*}
	where $ h_{B} := \Eof{\beta^{\omega}, \indicator{T^{-1}A\intxn B} \mid \wh{\xi_{j}}} / \Eof{\beta^{\omega}, \indicator{B} \mid \wh{\xi_{j}}}$. Since $ h_{B}$ is $ \wh{\xi_{j}}$ measurable, the function $ x\mapsto \nu_{x}(T^{-1}A)$ is $ \wh{\xi_{j}} \vee \wh{\calC}$-measurable. Moreover,
	\begin{equation}\label{eq:cond-exp-compute}
		\begin{aligned}
		\int \nu_{x}(T^{-1}A) \, \diff \beta^{\omega} & = \sum_{B\in\calC} \int \indicator{B}  h_{B} \, \diff \beta^{\omega} \\
		 & = \sum_{B \in\calC} \int \Eof{\beta^{\omega}, \indicator{B} h_{B} \mid \wh{\xi_{j}} } \, \diff \beta^{\omega} \\
		 & = \sum_{B \in\calC} \int \Eof{\beta^{\omega}, \indicator{B}  \mid \wh{\xi_{j}} } h_{B} \, \diff \beta^{\omega} & \by{$h_{B}$ being $\wh{\xi_{j}}$-measurable} \\
		 & = \sum_{B \in\calC} \int \Eof{\beta^{\omega}, \indicator{T^{-1}A\intxn B}  \mid \wh{\xi_{j}} }  \, \diff \beta^{\omega} & \by{the definition of $ h_{B}$} \\
		 & = \sum_{B\in\calC} \beta^{\omega}(T^{-1}A\intxn B) = \beta^{\omega}(T^{-1}A). 
	\end{aligned}
	\end{equation}
	Hence, the uniqueness of conditional expectation implies that
	\begin{align*}
		\nu_{x}(T^{-1}A) & = \Eof{ \beta^{\omega}, \indicator{T^{-1}A} \mid \wh{\xi_{j}} \vee \wh{\calC} } \\
		& = \Eof{ \beta^{\omega}, \indicator{T^{-1}A} \mid T^{-1}\wh{\xi_{j}} \vee \wh{\calC} } & \by{\autoref{lem:set-relations}\ref{itm:set-shift}}\\
		 & = \beta^{\omega, T^{-1}\xi_{j} \vee \calC}_{x}(T^{-1}A). & \by{\autoref{thm:rohlin}}
	\end{align*}
	 This together with \eqref{eq:mass-rel-first} finishes the proof.
\end{proof}

Then we compute some useful integrals related to the conditional information and entropy.

\begin{lemma}\label{lem:integral-entropies}
	Let $ \calE $ be a finite partition of $ \Lambda^{\bbN} $, and let $ \calF $ be a sub-$\sigma$-algebra of $ \calB(\Lambda^{\bbN})$. Then
	\begin{equation}\label{eq:Int-Info}
		\int_{\Omega\times\Lambda^{\bbN}} \Iof{\beta^{\omega}, \calE \mid \calF }(x) \, \diff \bfQ(\omega, x) = \Hof{ \beta, \calE \mid \wh{\calA} \vee \calF},
	\end{equation}
	and
	\begin{equation}\label{eq:int-omega-e}
		\int_{\Omega} \Hof{ \beta^{\omega}, \calE \mid \calF} \, \diff \bfP(\omega) = \Hof{ \beta, \calE \mid \wh{\calA} \vee \calF }.
	\end{equation}
\end{lemma}

\begin{proof}
	Since $ (\Lambda^{\bbN}, \calB(\Lambda^{\bbN}), \beta )$ is a separable probability space, there exists a sequence of countable partitions $ (\calF_{n})_{n=1}^{\infty} $ of $ \Lambda^{\bbN} $ so that $ \wh{\calF_{n}} \uparrow \calF $. Note that for any sub-$\sigma$-algebra $ \calG $ of $ \calB(\Lambda^{\bbN})$,
	\begin{align}
	\nonumber	\int_{\Omega\times\Lambda^{\bbN}} & \Iof{\beta^{\omega}, \calE \mid \calG }(x) \, \diff \bfQ(\omega, x) \\
	\label{eq:omega-e-int}	& = \int_{\Omega} \Hof{\beta^{\omega}, \calE \mid \calG} \, \diff \bfP(\omega) & \by{\eqref{eq:def-bfQ}}\\
	\label{eq:entropy-int}	& = \int_{\Lambda^{\bbN}} \Hof{\beta^{\calA}_{y}, \calE \mid \calG} \, \diff\beta(y) & \by{\eqref{eq:Sym-Disint}} \\
	\nonumber & = \int_{\Lambda^{\bbN}}\int_{\Lambda^{\bbN}} \Iof{ \beta^{\calA}_{y}, \calE \mid \calG  }(x) \,\diff \beta^{\calA}_{y}(x) \diff \beta(y) & \by{\eqref{eq:def-Hof}}\\
	\nonumber	& = \int_{\Lambda^{\bbN}}\int_{\Lambda^{\bbN}} \Iof{ \beta^{\calA}_{x}, \calE \mid \calG  }(x) \,\diff \beta^{\calA}_{y}(x) \diff \beta(y)  & \by{$\beta^{\calA}_{x} = \beta^{\calA}_{y}$ if $ x\in\calA(y)$} \\
	\label{eq:info-int}	& = \int_{\Lambda^{\bbN}} \Iof{\beta^{\calA}_{x} , \calE \mid \calG }(x) \, \diff \beta(x). & \by{\eqref{eq:Sym-Disint}}
	\end{align}
	Since \eqref{eq:int-omega-e} follows from \eqref{eq:Int-Info} and \eqref{eq:omega-e-int}, it suffices to prove \eqref{eq:Int-Info}.
	
	For each $ E \in \calE $, $ n \in \bbN $, $ \beta \aev x $, by \autoref{thm:rohlin} we have
	\begin{align*}
		\Eof{\beta^{\calA}_{x}, \indicator{E} \mid \wh{\calF_{n}} }(x) = \frac{\beta^{\calA}_{x}(E\intxn \calF_{n}(x))}{ \beta^{\calA}_{x}(\calF_{n}(x)) } = \sum_{F\in\calF_{n}} \indicator{F}(x) h_{F}(x),
	\end{align*}
	where $ h_{F}(x) = \Eof{ \beta, \indicator{E\intxn F} \mid \wh{\calA}} / \Eof{ \beta, \indicator{F} \mid \wh{\calA}} $. Then $ x \mapsto \Eof{\beta^{\calA}_{x}, \indicator{E} \mid \wh{\calF_{n}} }(x) $ is $ \wh{\calA} \vee \wh{\calF_{n}}$ measurable. This together with the computation in \eqref{eq:cond-exp-compute} shows that
	\begin{equation}\label{eq:cond-join}
		\Eof{\beta^{\calA}_{x}, \indicator{E} \mid \wh{\calF_{n}} }(x) = \Eof{\beta, \indicator{E} \mid \wh{\calA} \vee \wh{\calF_{n}}}(x).
	\end{equation}
	
	Hence 
	\begin{align*}
		\int_{\Omega\times\Lambda^{\bbN}} & \Iof{\beta^{\omega}, \calE \mid \calF }(x) \, \diff \bfQ(\omega, x) \\
		& = \int_{\Lambda^{\bbN}} \Hof{\beta^{\calA}_{y}, \calE \mid \calF} \, \diff\beta(y) & \by{\eqref{eq:entropy-int}} \\
		& = \int_{\Lambda^{\bbN}} \lim_{n\to\infty}  \Hof{\beta^{\calA}_{y}, \calE \mid \wh{\calF_{n}}} \, \diff \beta(y) & \by{\autoref{lem:list-identities}\ref{itm:increasing} and $\#\calE < \infty $} \\
		& = \lim_{n\to\infty}  \int_{\Lambda^{\bbN}}  \Hof{\beta^{\calA}_{y}, \calE \mid \wh{\calF_{n}}} \, \diff \beta(y) & \by{$ \# \calE < \infty $} \\
		& = \lim_{n\to\infty}  \int_{\Lambda^{\bbN}} \Iof{\beta^{\calA}_{x} , \calE \mid \wh{\calF_{n}} }(x) \, \diff \beta(x) & \by{\eqref{eq:info-int}} \\
		& = \lim_{n\to\infty} \Hof{\beta, \calE \mid \wh{\calA} \vee \wh{\calF_{n}}} & \by{\eqref{eq:cond-join}} \\
		& = \Hof{\beta, \calE \mid \wh{\calA} \vee \calF }, & \by{\autoref{lem:list-identities}\ref{itm:increasing} and $\#\calE < \infty $}
  	\end{align*}
	which finishes the proof.
\end{proof}

%

We finish this subsection with the a version of Shannon-McMillan-Breiman theorem.

\begin{lemma}\label{lem:Cn-S-M-B} 
	For $ \bfQ \aev (\omega, x)$ and $ 0 \leq j \leq d $, $ \lim_{n\to\infty} - (1/n) \log \beta^{\omega, \xi_{j}}_{x}(\calC_{0}^{n-1}(x)) = N \HcalA{[j]} $.
\end{lemma}

\begin{proof}
	For $ n \in \bbN $, we have
	\begin{align*}
		& \hspace{-2em} \Iof{ \beta^{\omega}, \vee_{i=0}^{n-1}T^{-i} \calC \mid \wh{\xi_{j}} }(x) \\ 
		& = \Iof{\beta^{\omega}, \calC \mid \wh{\xi_{j}}}(x) + \Iof{\beta^{\omega}, \vee_{i=1}^{n-1} T^{-i}\calC \mid \wh{\xi_{j}} \vee \wh{\calC} }(x)  \hspace{3.5em} \by{\autoref{lem:list-identities}\ref{itm:ChainRule-I}} \\
		& = \Iof{\beta^{\omega}, \calC \mid \wh{\xi_{j}}}(x) + \Iof{\beta^{\omega},  \vee_{i=1}^{n-1} T^{-i}\calC  \mid T^{-1}\wh{\xi_{j}} \vee \wh{\calC} }(x) \hspace{2em} \by{\autoref{lem:set-relations}\ref{itm:set-shift}} \\
		& = \Iof{\beta^{\omega}, \calC \mid \wh{\xi_{j}}}(x) + \Iof{\beta^{\omega}, \vee_{i=1}^{n-1} T^{-i}\calC \mid T^{-1}\wh{\xi_{j}} }(x) \hspace{2em} \by{\autoref{lem:list-identities}\ref{itm:condindp-info}} \\
		& = \Iof{\beta^{\omega}, \calC \mid \wh{\xi_{j}}}(x) + \Iof{\beta^{T\omega}, \vee_{i=0}^{n-2} T^{-i}\calC \mid \wh{\xi_{j}} }(Tx). \hspace{1.5em} \by{\autoref{lem:list-identities}\ref{itm:list-pushfoward-I} and \eqref{eq:Tbeta-omega}} 
	\end{align*}
	Then an induction shows that
	\begin{equation}\label{eq:I-C-n-Birk-Sum}
		\Iof{ \beta^{\omega}, \vee_{i=0}^{n-1}T^{-i} \calC \mid \wh{\xi_{j}} }(x) = \sum_{k=0}^{n-1} \Iof{\beta^{T^{k}\omega}, \calC \mid \wh{\xi_{j}}}(T^{k}x).
	\end{equation}
	
	On the other hand, it follows from \autoref{thm:rohlin} and \eqref{eq:def-Iof} that for $ \bfQ \aev (\omega, x)$,
	\begin{equation}\label{eq:cond-Cn}
		- \log \beta^{\omega,\xi_{j}}_{x}(\calC_{0}^{n-1}(x)) = \Iof{ \beta^{\omega}, \vee_{i=0}^{n-1}T^{-i} \calC \mid \wh{\xi_{j}} }(x).
	\end{equation}
	 By \eqref{eq:I-C-n-Birk-Sum}, \eqref{eq:cond-Cn} and \eqref{eq:Int-Info}, applying Birkhoff's ergodic theorem finishes the proof.
\end{proof} 


\subsection{Transverse dimensions}

The aim of this subsection is to prove \autoref{prop:LB-transDim}, which intuitively provides the local dimension of $ \mu^{\omega} $ along each coordinate.

\begin{proposition}\label{prop:LB-transDim}
	For $ \bfQ $-a.e.\ $ (\omega, x) $ and $ 1 \leq j \leq d $,
	\begin{equation*}
		\lim_{r\to 0} \frac{ \log \beta^{\omega,\xi_{j-1}}_{x}( \PiBall[j]{x}{r} ) }{\log r} = \frac{\HcalA{[j-1]} - \HcalA{[j]} }{ \chi_{j}},
	\end{equation*}
	where $ \HcalA{I} $ is defined in \eqref{eq:def-HcalA}.
\end{proposition}

The proof of \autoref{prop:LB-transDim} is inspired by \cite[Proposition 5.1]{Feng2023a}. The key idea is to reformulate the measures of small balls in terms of certain variants of Birkhoff sums. The proof is then completed by applying Birkhoff's and the following Maker's ergodic theorems~\cite{Maker1940}.


\begin{lemma}[{Maker~\cite{Maker1940}}]\label{lem:Makers}
	Let $ T $ be a measure-preserving transformation on a probability space $ (X, \calB, \theta)$. Let $ (g_{n})_{n=1}^{\infty} $ be a sequence of measurable functions converging $ \theta $-a.e.\@ to $ g $. Suppose $ \sup_{n} \abs{g_{n}} \leq f $ for some $ f \in L^{1}(X,\calB,\theta)$. Then both $\theta$-a.e.\@ and in $ L^{1}$,
	\begin{equation*}
		\lim_{n\to\infty} \frac{1}{n} \sum_{k=0}^{n-1} g_{n-k}(T^{k}x) = \Eof{\theta, g \mid \calI }(x),
	\end{equation*}
	 where $ \calI = \{ B \in \calB, T^{-1}B= B\} $.
\end{lemma}

The following lemma is a preparation for applying \autoref{lem:Makers}.

\begin{lemma}\label{lem:Cond-Density}
	For  $ \bfQ $-a.e.\ $ (\omega, x)$ and $ 1 \leq j \leq d $,
	\begin{equation}\label{eq:cond-density}
		\lim_{r \to 0 } - \log \frac{ \betaOXj{\omega}{x}{j-1} \left( \PiBall[j]{x}{r} \intxn \calC(x)  \right)  }{\betaOXj{\omega}{x}{j-1} \left( \PiBall[j]{x}{r} \right) } = \Iof{ \beta^{\omega}, \calC \mid \wh{\xi_{j}} }(x).
	\end{equation}
	Furthermore, set
	\begin{equation*}
		g(\omega, x) = - \inf_{r>0} \log \frac{ \beta^{\omega,\xi_{j-1}}_{x} \left( \PiBall[j]{x}{r} \intxn \calC(x)  \right)  }{\beta^{\omega,\xi_{j-1}}_{x} \left( \PiBall[j]{x}{r} \right) }.
	\end{equation*}
	Then $ g \geq 0 $ and $ g \in L^{1}(\Omega\times \Lambda^{\bbN}, \bfQ)$.
\end{lemma}
\begin{proof} 
	Applying \cite[Lemma 2.5(2)]{Feng2023a} with $ \Lambda^{\bbN}, \pi_{[j]}\euclid, \pi_{[j]}, \beta^{\omega}, \calC, \xi_{j-1} $ in place of $ X, Y, \pi, m, \alpha, \eta $ gives
	\begin{equation*}
		\lim_{r \to 0 } - \log \frac{ \beta^{\omega,\xi_{j-1}}_{x} \left( \PiBall[j]{x}{r} \intxn \calC(x)  \right)  }{\beta^{\omega,\xi_{j-1}}_{x} \left( \PiBall[j]{x}{r} \right) } = \Iof{ \beta^{\omega}, \calC \mid \wh{\xi_{j}} \vee \wh{\xi_{j-1}} }(x).
	\end{equation*} 
	This implies \eqref{eq:cond-density} since $ \xi_{j-1} \prec \xi_{j}  $. The last statement follows from the second part of \cite[Lemma 2.5(2)]{Feng2023a} and  $ \Hof{\beta^{\omega}, \calC} \leq N \log \abs{\Lambda}  $ for all $ \omega \in \Omega $.
\end{proof}

We are now ready to prove \autoref{prop:LB-transDim}.

\begin{proof}[Proof of \autoref{prop:LB-transDim}]
	The proof is adapted from \cite[Proposition 5.1]{Feng2023a}. For clarity and to account for the dependence on 
	$ \omega $, we provide the details in full.
	
	For $ n \in \bbN $, define
	\begin{equation}\label{eq:def-H}
		H_{n}(\omega, x)  = \log \frac{  \betaOXj{\omega}{x}{j-1} \left( \PiBall[j]{x}{\lambda^{x|nN}_{j}}  \right)   }{ \betaOXj{T\omega}{Tx}{j-1} \left( \PiBall[j]{Tx}{\lambda^{Tx|(n-1)N}_{j}} \right) }.
	\end{equation}
	Then by telescoping and $ \diam(\supp \mu) \leq 1 $,
\begin{equation}\label{eq:H-telescoped}
		\begin{aligned}
		\sum_{k=0}^{n-1} H_{n-k}( T^{k}(\omega, x)) = \log \beta^{\omega,\xi_{j-1}}_{x} \left( \PiBall[j]{x}{\lambda_{j}^{x|nN}} \right).
		\end{aligned}
\end{equation}

	
	For $ n \in \bbN $, define
	\begin{equation}\label{eq:def-G}
		G_{n}(\omega, x)  = \log \frac{  \betaOXj{\omega}{x}{j-1} \left( \PiBall[j]{x}{\lambda_{j}^{x|nN}} \intxn \calC(x) \right)   }{ \betaOXj{\omega}{x}{j-1} \left( \PiBall[j]{x}{\lambda_{j}^{x|nN}} \right)  }.
	\end{equation}
	For $ 1 \leq j \leq d $, write
	\begin{equation}\label{eq:def-Qj}
		Q_{j}(\omega,x) = \Iof{ \beta^{\omega}, \calC \mid \wh{\xi_{j}} }(x).
	\end{equation}
	Then \autoref{lem:Cond-Density} implies that $ \sup_{n} \abs{G_{n}} \in L^{1}(\bfQ)$ and for $ \bfQ \aev (\omega, x) $,
	\begin{equation*}
		\lim_{n\to \infty} G_{n} = - Q_{j}.
	\end{equation*}  
	Thus for $ \bfQ \aev (\omega, x) $, combining \autoref{lem:Makers} and \autoref{lem:integral-entropies} shows that
	\begin{equation}\label{eq:limit-G}
		\lim_{n \to \infty } \frac{1}{n} \sum_{k=0}^{n-1} G_{n-k}(T^{k}(\omega,x)) = - \int Q_{j} \, \diff \bfQ = - N \HcalA{[j]},
	\end{equation}
	and by Birkhoff's ergodic theorem,
	\begin{equation}\label{eq:lim-Qj}
		\lim_{n\to\infty} \frac{1}{n} \sum_{k=0}^{n-1} Q_{j-1}(T^{k}(\omega, x)) = N \HcalA{[j-1]}.
	\end{equation}

	
	Next, we show that for $ n \in \bbN $,
	\begin{equation}\label{eq:H+G=-Q_j-1}
		H_{n} = - Q_{j-1} - G_{n}.
	\end{equation}
	This is justified as follows,
	\begin{align*}
		H_{n}&(\omega, x)  + G_{n}(\omega, x) \\
		& = \log \frac{  \beta^{\omega,\xi_{j-1}}_{x} \left( B^{\Pi_{j}}( x, \lambda_{j}^{x|nN} ) \intxn \calC(x) \right)    }{ \beta^{T\omega,\xi_{j-1}}_{Tx} \left( B^{\Pi_{j}}( Tx, \lambda_{j}^{Tx | (n-1)N} ) \right) } \hspace{3em} \by{\eqref{eq:def-H} and \eqref{eq:def-G}} \\
		& = \log \frac{  \beta^{\omega,\xi_{j-1}}_{x} \left( \xi_{j-1}(x) \intxn B^{\Pi_{j}}( x, \lambda_{j}^{x|nN} ) \intxn \calC(x) \right)    }{ \beta^{T\omega,\xi_{j-1}}_{Tx} \left( B^{\Pi_{j}}( Tx, \lambda_{j}^{Tx | (n-1)N} ) \right) } \hspace{3em} \by{$ \betaOXj{\omega}{x}{j-1}(\xi_{j-1}(x)) =1 $} \\
		& = \log \frac{  \beta^{\omega,\xi_{j-1}}_{x} \left( T^{-1}\left(\xi_{j-1}(Tx)\intxn B^{\Pi_{[j]}}(Tx, \lambda_{j}^{Tx|(n-1)N})\right) \intxn \calC(x) \right)    }{ \beta^{T\omega,\xi_{j-1}}_{Tx} \left( B^{\Pi_{j}}( Tx, \lambda_{j}^{Tx | (n-1)N} ) \right) } \hspace{2em} \by{\autoref{lem:set-relations}\ref{itm:set-PiBall-shift}}\\
		& = \log \frac{  \beta^{\omega,\xi_{j-1}}_{x} \left( T^{-1} B^{\Pi_{j}}(Tx, \lambda_{j}^{Tx|(n-1)N})  \intxn T^{-1}\xi_{j-1}(Tx) \intxn \calC(x) \right)    }{ \beta^{T\omega,\xi_{j-1}}_{Tx} \left( B^{\Pi_{j}}( Tx, \lambda_{j}^{Tx | (n-1)N} ) \right) } \hspace{2em} \by{rearranging} \\
		& = \log \frac{  \beta^{\omega,\xi_{j-1}}_{x} \left( T^{-1} B^{\Pi_{j}}(Tx, \lambda_{j}^{Tx|(n-1)N})  \intxn \xi_{j-1}(x) \intxn \calC(x) \right)    }{ \beta^{T\omega,\xi_{j-1}}_{Tx} \left( B^{\Pi_{j}}( Tx, \lambda_{j}^{Tx | (n-1)N} ) \right) } \hspace{3em} \by{\autoref{lem:set-relations}\ref{itm:set-shift}} \\
		& = \log \frac{  \beta^{\omega,\xi_{j-1}}_{x} \left( T^{-1} B^{\Pi_{j}}(Tx, \lambda_{j}^{Tx|(n-1)N}) \intxn \calC(x) \right)    }{ \beta^{T\omega,\xi_{j-1}}_{Tx} \left( B^{\Pi_{j}}( Tx, \lambda_{j}^{Tx | (n-1)N} ) \right) } \hspace{3em} \by{$ \betaOXj{\omega}{x}{j-1}(\xi_{j-1}(x)) =1 $} \\
		& = \log \beta^{\omega,\xi_{j-1}}_{x}\left( \calC(x)\right) \hspace{7.5em} \by{\autoref{lem:measure-relations}}\\
		& = - \Iof{ \beta^{\omega}, \calC \mid \wh{\xi_{j-1}}}(x) \hspace{6em} \by{\autoref{thm:rohlin} and \eqref{eq:def-Iof}}\\
		&  = - Q_{j-1}(\omega, x). \hspace{9em} \by{\eqref{eq:def-Qj}}
	\end{align*}
	
	Finally, for $ \bfQ \aev (\omega, x)$, we have 
	\begin{align*}
		\lim_{r \to 0} & \frac{\log \beta^{\omega,\xi_{j-1}}_{x} \left( B^{\Pi_{j}}( x, r ) \right) }{ \log r} \\ & = \lim_{n\to\infty}  \frac{\log \beta^{\omega,\xi_{j-1}}_{x} \left( B^{\Pi_{j}}( x, \lambda_{j}^{x|nN}) \right) }{ \log \lambda_{j}^{x|nN}} \hspace{10em} \by{\autoref{lem:lambda-speed}}\\
		& = \lim_{n\to\infty}  \frac{  \sum_{k=0}^{n-1} H_{n-k}(T^{k}(\omega,x)) }{ \log \lambda_{j}^{x|nN}} \hspace{13em} \by{\eqref{eq:H-telescoped}}\\
		& = \lim_{n\to\infty}  \frac{  \sum_{k=0}^{n-1} Q_{j-1}(T^{k}(\omega, x))+ \sum_{k=0}^{n}G_{n-k}(T^{k}(\omega,x)) }{  - \log \lambda_{j}^{x|nN}} \hspace{3em} \by{\eqref{eq:H+G=-Q_j-1}} \\
		& = \lim_{n\to\infty}  \frac{ H^{\calA}_{[j-1]} - H^{\calA}_{[j]} }{  \chi_{j} }. \hspace{8em} \by{\eqref{eq:limit-G}, \eqref{eq:lim-Qj} and \autoref{lem:lambda-speed}}
	\end{align*}
	This finishes the proof.
\end{proof}

\subsection{Proof of \autoref{thm:exact-dim-general}}

In this subsection, we prove \autoref{thm:exact-dim-general} by adapting the arguments in \cite[Section 6]{Feng2023a}, which is itself inspired by ideas from Ledrappier and Young~\cite{LedrappierYoung1985}.

For $ 1 \leq i \leq d $, denote
\begin{equation} \label{eq:def-theta-i}
	\vartheta_{i} := \frac{ \HcalA{[i-1]} - \HcalA{[i]} }{ \chi_{i} }.
\end{equation}
Using \autoref{prop:LB-transDim}, it follows that for $ \bfQ \aev (\omega, x)$,
\begin{equation}\label{eq:TransDim}
	\vartheta_{i} = \lim_{r\to 0} \frac{ \log \beta^{\omega,\xi_{i-1}}_{x}( \PiBall[i]{x}{r} ) }{\log r}.
\end{equation}

For $ \bfQ $-a.e. $(\omega, x)$ and $ 0 \leq i \leq j \leq d $, define
\begin{equation}\label{eq:gamma-def}
	\ol{\gamma}_{i,j}^{\omega}(x) = \limsup_{r \to 0} \frac{ \log \betaOXj{\omega}{x}{i} \left(  \PiBall[j]{x}{r} \right)  }{ \log r } \mAnd   \ul{\gamma}_{i,j}^{\omega} = \liminf_{r\to 0} \frac{\log  \betaOXj{\omega}{x}{i} \left(  \PiBall[j]{x}{r} \right) }{ \log r}. 
\end{equation}

We claim that the following three statements hold for $ \bfQ $-a.e.\ $ (\omega, x)$:
	\begin{align} 
		& \tag*{(D1)}\label{itm:D1} \hspace{0em}  \ugamma{j}{j} = \lgamma{j}{j}  = 0 \\
		& \tag*{(D2)}\label{itm:D2} \hspace{0em} \chi_{i} \left( \ugamma{i-1}{j} - \ugamma{i}{j} \right) \leq \HcalA{[i-1]} - \HcalA{[i]} \mFor  1 \leq i \leq j . \\
		& \tag*{(D3)}\label{itm:D3} \hspace{0em} \lgamma{i}{j} + \vartheta_{i} \leq \lgamma{i-1}{j}  \mFor  1 \leq i \leq j.
	\end{align}

\begin{proof}[Proof of \autoref{thm:exact-dim-general} assuming \ref{itm:D1}--\ref{itm:D3}]
	Combining \eqref{eq:TransDim}, \ref{itm:D2} and \ref{itm:D3} shows that if $ \ugamma{i}{j} = \lgamma{i}{j} = \cgamma{i}{j}$ for some $ \cgamma{i}{j} \in \bbR $, then
	\begin{equation}\label{eq:InductGamas}
	\lgamma{i-1}{j} \leq \ugamma{i-1}{j} \leq \ugamma{i}{j} + \theta_{i} = \lgamma{i}{j} + \theta_{i} \leq \lgamma{i-1}{j}.
	\end{equation}
	Thus $\lgamma{i-1}{j} = \ugamma{i-1}{j} =  \cgamma{i-1}{j} $ for some $ \cgamma{i-1}{j} \in \bbR $, and so
	\begin{equation} \label{eq:EqGamas}
		\cgamma{i-1}{j} = \cgamma{i}{j} + \vartheta_{i}.
	\end{equation} 
	By \ref{itm:D1}, an induction from $ i = j$ shows that \eqref{eq:InductGamas} and \eqref{eq:EqGamas} hold for all $ 1 \leq i \leq j $. Hence
	\begin{equation}\label{eq:gamma-as-sum-theta}
		\cgamma{i}{j} = \sum_{\ell=i+1}^{j} \vartheta_{\ell} = \sum_{\ell=i+1}^{j} \frac{\HcalA{[\ell-1]} - \HcalA{[\ell]}}{\chi_{\ell}} \mFor 0 \leq i \leq j.
	\end{equation}
	
	Note that for $ \bfQ \aev (\omega, x)$ and $ r > 0 $,
	\begin{equation*}\label{eq:tube-as-ball}
		\betaOXj{\omega}{x}{i} \left(  \PiBall[j]{x}{r} \right) =  (\pi_{[j]}\Pi \betaOXj{\omega}{x}{i})\left( B\left (\pi_{[j]} \Pi(x), r \right)\right). 
	\end{equation*}
	This together with \eqref{eq:gamma-def} and \eqref{eq:gamma-as-sum-theta} shows that for $ \bfQ \aev (\omega, x)$ and $ 0 \leq i \leq j$, the measure $ \pi_{[j]} \Pi \betaOXj{\omega}{x}{i} $ is exact dimensional with
	\begin{equation}\label{eq:LY-[d]}
		\dim \pi_{[j]} \Pi \betaOXj{\omega}{x}{i} =  \sum_{\ell=i+1}^{j} \frac{\HcalA{[\ell-1]} - \HcalA{[\ell]}}{\chi_{\ell}}.
	\end{equation}
	This proves \autoref{thm:exact-dim-general} when $ J = [d]$. For general $ J \subset [d]$, the proof is finished by considering $ \Phi_{J}$ instead.
\end{proof}

It remains to prove \ref{itm:D1}--\ref{itm:D3}.

\begin{proof}[Proof of \ref{itm:D1}]
	Since $ \xi_{j}(x) = \Pi_{j}^{-1}(\Pi_{j}(x)) \subset \PiBall[j]{x}{r} $ for every $ x \in \Lambda^{\bbN} $ and $ r > 0 $, we have
	\begin{equation*}
		1 \geq \betaOXj{\omega}{x}{j}\left( \PiBall[j]{x}{r} \right) \geq \betaOXj{\omega}{x}{j}(\xi_{j}(x)) = 1.  
	\end{equation*}
	Thus $ \ugamma{j}{j} = \lgamma{j}{j} = 0 $ for $\bfQ$-a.e.\ $ (\omega, x)$.
\end{proof}

The proof of \ref{itm:D2} and \ref{itm:D3} relies on the next lemma showing that a set with positive measure has positive density with respect to conditional measures almost surely.

\begin{lemma}\label{lem:pos-CondExp}
	Let $ \omega \in \Omega $ and $ A \in \calB(\Lambda^{\bbN}) $ be with $ \beta^{\omega}(A) > 0 $. Then for $ 0 \leq i \leq j \leq d $ and $ \beta^{\omega } \aev x \in A $,
	\begin{equation*}
		\lim_{r\to 0} \frac{ \beta^{\omega,\xi_{i}}_{x}(A\intxn\PiBall[j]{x}{r}) }{ \beta^{\omega,\xi_{i}}_{x}(\PiBall[j]{x}{r}) } > 0.
	\end{equation*}
\end{lemma}

\begin{proof}
	Applying \cite[Lemma 2.5(1)]{Feng2023a} with $ \Lambda^{\bbN}, \pi_{[j]}\euclid, \pi_{[j]}, \beta^{\omega}, \calC, \xi_{i} $ in place of $ X, Y, \pi, m, \alpha, \eta $ shows that for $ \beta^{\omega} \aev x $,
	\begin{equation*}
		\lim_{r\to 0} \frac{ \beta^{\omega,\xi_{i}}_{x}(A\intxn\PiBall[j]{x}{r}) }{ \beta^{\omega,\xi_{i}}_{x}(\PiBall[j]{x}{r}) } = \Eof{ \beta^{\omega}, \indicator{A} \mid \wh{\xi_{i}} \vee \wh{\xi_{j}} }(x).
	\end{equation*}
	The proof is completed by an almost trivial property of conditional expectation that, for a probability space  $(X, \calB, \theta )$ and a sub-$\sigma$-algebra $\calF$ of $ \calB $, letting $ A \in \calB $ be with $ \theta(A) > 0 $, we have
	\begin{equation*}
		\Eof{\theta, \indicator{A} \mid \calF }(x) > 0 \mFor \theta \aev x \in A.
	\end{equation*}
	(See e.g.\@ \cite[Lemma 3.10]{FengHu2009} for a proof.)
\end{proof}

Now we are ready to prove \ref{itm:D2} and \ref{itm:D3}.

\begin{proof}[Proof of \ref{itm:D2}]
	
	For $ 0 \leq i \leq j $, write $ h_{i} := \HcalA{[i]}$ for short. Suppose on the contrary that \ref{itm:D2} is not true.  There exists $ 1 \leq i \leq j $ and $ U \subset \Omega \times \Lambda^{\bbN} $ with $ \bfQ(U) > 0 $ such that for $ (\omega, x) \in U $,
	\begin{equation} \label{eq:D2-false}
		\frac{ h_{i-1} - h_{i} }{\chi_{i}} < \ugamma{i-1}{j} - \ugamma{i}{j}.
	\end{equation}
	
	It follows from \eqref{eq:D2-false} and \eqref{eq:gamma-def} that $U$ is a subset of the following set, 
	\begin{align*}
		\Union_{\alpha \in \bbQ \intxn (0,\infty)}  \Union_{\ol{\gamma}_{i-1}, \ol{\gamma}_{i} \in \bbQ } \Intxn_{\varepsilon > 0 } \bigg\{(\omega,x) \colon & \frac{ h_{i-1} - h_{i} }{\chi_{i}} < \ol{\gamma}_{i-1} - \ol{\gamma}_{i} - \alpha , \\ &   \ugamma{i-1}{j} > \ol{\gamma}_{i-1} - \varepsilon/2, \,  \ugamma{i}{j} < \ol{\gamma}_{i} + \varepsilon/2   \bigg \}.
	\end{align*}
	Then there exist $ \alpha > 0 $ and real numbers $ \ol{\gamma}_{i-1}, \ol{\gamma_{i}} $ such that
	\begin{equation}\label{eq:D2-const-rel}
		\frac{ h_{i-1} - h_{i} }{\chi_{i}} < \ol{\gamma}_{i-1} - \ol{\gamma}_{i} - \alpha,
	\end{equation}
	and for $  \varepsilon > 0 $ there exists $ U_{\varepsilon} \subset U $ with $ \bfQ(U_{\varepsilon}) > 0 $ so that for $ x \in U_{\varepsilon}$, 
	\begin{equation}
		\ugamma{i-1}{j} > \ol{\gamma}_{i-1} - \varepsilon/2, \qquad \ugamma{i}{j} < \ol{\gamma}_{i} + \varepsilon/2.
	\end{equation}
	Fix $ \varepsilon \in (0, \chi_{i}/3)$. There exists $ n_{0} \colon U_{\varepsilon} \to \bbN $ such that for $ \bfQ $-a.e.\ $ (\omega, x) \in U_{\varepsilon}$ and $ n > n_{0}(x)$,
	\begin{enumerate}[(1)]
		\item\label{itm:D2-(1)} $ \betaOXj{\omega}{x}{i}\left(\PiBall[j]{x}{ \exp(-n(N\chi_{i}-2\varepsilon)) }\right)  > \exp\left( -n(N\chi_{i}-2\varepsilon)(\ol{\gamma}_{i} + \varepsilon)\right)   $; \quad \by{\eqref{eq:gamma-def}}
		
		\item \label{itm:D2-(2)} $ \betaOXj{\omega}{x}{i}\left( \calC_{0}^{n-1}(x) \right)  < \exp\left( - n( Nh_{i} - \varepsilon)\right)  $; \quad \by{\autoref{lem:Cn-S-M-B}}
		
		\item \label{itm:D2-(3)} $ \betaOXj{\omega}{x}{i-1}\left( \calC^{n-1}_{0}(x) \right)  > \exp\left(-n( Nh_{i-1} + \varepsilon)\right) $; \quad \by{\autoref{lem:Cn-S-M-B}}
		
		\item \label{itm:D2-(4)} $ \xi_{i-1}(x) \intxn \calC^{n-1}_{0}(x) \subset \PiBall[j]{x}{ \exp(-n(N\chi_{i}-2\varepsilon)) }  $. \quad \by{\autoref{lem:set-relations}\ref{itm:SliceElipse}}.
	\end{enumerate}
	
	Take $ N_{0} $ such that
	\begin{equation*}
		\Delta := \{ x\in U_{\varepsilon} \colon n_{0}(x) \leq N_{0} \}
	\end{equation*}
	satisfies $ \bfQ(\Delta) > 0 $. By \eqref{eq:def-bfQ} there exists $ \wt{\Omega} \in \Omega $ with $ \bfP(\wt{\Omega}) > 0 $ such that for each $ \omega \in \wt{\Omega}$ there exists $ X^{\omega} \subset \Lambda^{\bbN} $ satisfying $ \{\omega\} \times X^{\omega} \subset \Delta $ and $  \beta^{\omega}(X^{\omega}) > 0 $. \autoref{lem:pos-CondExp} implies that for some $ c > 0 $ and each $ \omega \in \wt{\Omega}$, there exists $ Y^{\omega} \subset X^{\omega} $ with $ \beta^{\omega} (Y^{\omega}) > 0 $ such that for $  x \in Y^{\omega} $ there exists $ n = n(\omega,x) \geq N_{0} $ satisfying,
	\begin{enumerate}[(1)]
		\setcounter{enumi}{4}
		\item \label{itm:D2-(5)} $ \betaOXj{\omega}{x}{i}\left( L \intxn X^{\omega} \right) > c \betaOXj{\omega}{x}{i}\left( L \right)  $, where $ L = \PiBall[j]{x}{\exp( - n(N\chi_{i} -  2\varepsilon) )}$;
		
		\item \label{itm:D2-(6)} $   \betaOXj{\omega}{x}{i-1}\left(\PiBall[j]{x}{ 2\exp(-n(N\chi_{i}-2\varepsilon)) }\right) < \exp\left(-n(N\chi_{i}-2\varepsilon)(\ol{\gamma}_{i-1} - \varepsilon)\right)    $; (by \eqref{eq:gamma-def})
		
		\item \label{itm:D2-(7)} $ \log(1/c) <  n\varepsilon$.
	\end{enumerate}
	
	Take $ \omega \in \wt{\Omega} $ and $ x \in Y^{\omega }$ such that \ref{itm:D2-(1)}--\ref{itm:D2-(7)} are satisfied with $ n = n(\omega, x)$.  By \ref{itm:D2-(5)} and \ref{itm:D2-(1)},
	\begin{equation*}
		\betaOXj{\omega}{x}{i} \left( L \intxn X^{\omega} \right) \geq c \betaOXj{\omega}{x}{i}(L)\geq c  \exp( -  n(N\chi_{i}\ol{\gamma}_{i} + O(\varepsilon) ) ).
	\end{equation*} 
	For each $ I \in \calC^{n-1}_{0}$ with $I \intxn \xi_{i}(x) \intxn L \intxn X^{\omega} \neq \emptyset$, there is $ y \in X^{\omega} $ such that $ I = \calC^{n-1}_{0}(y)$ and $ \xi_{i}(y) = \xi_{i}(x) $. Thus, \ref{itm:D2-(2)} implies
	\begin{equation*}
		\betaOXj{\omega}{x}{i}\left( I \right) = \betaOXj{\omega}{y}{i}\left( \calC^{n-1}_{0}(y) \right)  < \exp(-n(Nh_{i} - \varepsilon )).
	\end{equation*}
	Hence, by $ \xi_{i}(x) \subset \xi_{i-1}(x) $, combining the previous two equations gives
	\begin{align*}
		\# \bigg\{I & \in \calC_{0}^{n-1}  \colon  I \intxn \xi_{i-1}(x) \intxn L \intxn X^{\omega} \neq \emptyset \bigg\} \\ & \geq \# \bigg\{I \in \calC_{0}^{n-1} \colon I \intxn \xi_{i}(x) \intxn L \intxn X^{\omega} \neq \emptyset \bigg\} \\
		& \geq c \exp\left(n( N(h_{i}-\chi_{i}\ol{\gamma}_{i})-O(\varepsilon))\right).
	\end{align*}
	
	On the other hand, for each $ I \in \calC^{n-1}_{0} $ with $ I \intxn \xi_{i-1}(x) \intxn L \intxn X^{\omega} \neq \emptyset $, there exists $ z \in I \intxn \xi_{i-1}(x)  \intxn L \intxn X^{\omega} $. Thus,
	\begin{align*}
		\xi_{i-1}(x) \intxn I & = \xi_{i-1}(z) \intxn \calC^{n-1}_{0}(z) \\ & \subset \PiBall[j]{z}{\exp\left(-n(N\chi_{i}-2\varepsilon)\right) }  & \text{(by \ref{itm:D2-(4)})} \\
		& \subset \PiBall[j]{x}{2 \exp\left(-n(N\chi_{i}-2\varepsilon)\right) }. & \text{(by $ z \in L $)}
	\end{align*}
	It follows from \ref{itm:D2-(3)} that
	\begin{equation*}
		\betaOXj{\omega}{x}{i-1}(I) = \betaOXj{\omega}{z}{i-1}\left(\calC^{n-1}_{0}(z)\right) \geq \exp( -n(Nh_{i-1}+\varepsilon)).
	\end{equation*}
	Hence
	\begin{align*}
		\betaOXj{\omega}{x}{i-1} & \left( \PiBall[j]{x}{2\exp(-n(N\chi_{i}-2\varepsilon))} \right) \\ & \geq  \# \bigg\{I \in \calC_{0}^{n-1}  \colon  I \intxn \xi_{i-1}(x) \intxn L \intxn X^{\omega} \neq \emptyset \bigg\} \exp(-n(Nh_{i-1}+\varepsilon)) \\
		& \geq \exp\left(\log c + n\left (N(h_{i}-h_{i-1}-\chi_{i}\ol{\gamma_{i}}) - O(\varepsilon)\right )  \right).
	\end{align*}
	From this, \ref{itm:D2-(6)} and \ref{itm:D2-(7)} it follows that
	\begin{equation*}
		- N\chi_{i}\ol{\gamma}_{i-1} + O(\varepsilon) \geq N(h_{i}-h_{i-1}-\chi_{i}\ol{\gamma}_{i}) - O(\varepsilon).
	\end{equation*}
	Letting $ \varepsilon \to 0 $ and dividing by $ N $ give $ h_{i-1} - h_{i} \geq \chi_{i} (\ol{\gamma}_{i-1} - \ol{\gamma}_{i})$, a contradiction to \eqref{eq:D2-const-rel}.
\end{proof}

\begin{proof}[Proof of \ref{itm:D3}]
	
	Suppose on the contrary that \ref{itm:D3} is not true. Then there exists $ 1 \leq i \leq j $ and $ U \subset \Omega \times \Lambda^{\bbN} $ with $ \bfQ(U) > 0 $ such that for $ (\omega, x) \in U $,
	\begin{equation*}\label{eq:D3-false} 
		\lgamma{i}{j} + \vartheta_{i} > \lgamma{i-1}{j}.
	\end{equation*}
	Then there exist $ \alpha > 0 $ and real numbers $ \ul{\gamma}_{i-1}, \ul{\gamma}_{i-1} $ such that
	\begin{equation}\label{eq:D3-false-consts}
		\ul{\gamma}_{i} + \vartheta_{i} > \ul{\gamma}_{i-1} + \alpha,
	\end{equation} and for every $ \varepsilon > 0 $, there exists $ U_{\varepsilon} \subset U $ with $ \bfQ(U_{\varepsilon}) > 0 $ so that for $ (\omega, x) \in U_{\varepsilon}$,
	\begin{equation}\label{eq:D3-lgammas}
		\Abs{ \lgamma{i-1}{j} - \ul{\gamma}_{i-1} } < \varepsilon / 2 \mAnd \Abs{ \lgamma{i}{j} - \ul{\gamma}_{i} } < \varepsilon / 2.
	\end{equation}
	
	Let $ 0<  \varepsilon < \alpha / 4 $. By Egorov's theorem, there exists $ \Delta \subset U_{\varepsilon} $ with $ \bfQ(\Delta) > 0 $  and $ N_{0} \in \bbN$ such that for $ (\omega,x) \in \Delta $ and $ n > N_{0}$,
	\begin{equation}\label{eq:D3-i-UB}
		\betaOXj{\omega}{x}{i}\left( \PiBall[j]{x}{2\exp(-n)} \right) \leq \exp\left(-n\left (\ul{\gamma}_{i} - \varepsilon \right )\right).
	\end{equation}
	By \eqref{eq:def-bfQ}, there exists $ \wt{\Omega} \in \Omega $ with $ \bfP(\wt{\Omega}) > 0 $ so that for each $ \omega \in \wt{\Omega}$ there exists $ X^{\omega} \subset \Lambda^{\bbN} $ satisfying $ \{\omega\} \times X^{\omega} \subset \Delta $ and $  \beta^{\omega}(X^{\omega}) > 0 $. \autoref{lem:pos-CondExp} implies that for some $ c > 0 $ and each $ \omega \in \wt{\Omega}$, there exists $ Y^{\omega} \subset X^{\omega} $ with $ \beta^{\omega} (Y^{\omega}) > 0 $ such that for $  x \in Y^{\omega} $ there exists $ N_{1} \geq N_{0} $ so that for $ \omega \in \wt{\Omega} $, $ x \in Y^{\omega} $ and $ n \geq N_{1}$,
	\begin{equation}
		\betaOXj{\omega}{x}{i-1}\left( X^{\omega} \intxn \PiBall[j]{x}{\exp(-n)} \right) > c \betaOXj{\omega}{x}{i-1}\left( \PiBall[j]{x}{\exp(-n)} \right).  
	\end{equation}
	Then
	\begin{equation}\label{eq:D3-tower}
		\begin{aligned}
			& \hspace{-2em}\betaOXj{\omega}{x}{i-1} \left( \PiBall[j]{x}{\exp(-n)}\right)  \\ & \leq c^{-1}  \betaOXj{\omega}{x}{i-1}\left( X^{\omega} \intxn \PiBall[j]{x}{\exp(-n)} \right) \\
			& \leq c^{-1} \int_{\Lambda^{\bbN}} \betaOXj{\omega}{y}{i}\left( X^{\omega} \intxn \PiBall[j]{x}{\exp(-n)} \right) \, \diff \betaOXj{\omega}{x}{i-1}(y) & \text{(by $\xi_{i-1} \prec \xi_{i} $)}	\\
			& \leq c^{-1} \int_{\PiBall[i]{x}{\exp(-n)}} \betaOXj{\omega}{y}{i}\left( X^{\omega} \intxn \PiBall[j]{x}{\exp(-n)} \right) \, \diff \betaOXj{\omega}{x}{i-1}(y),
		\end{aligned}
	\end{equation}
	where the last inequality holds since combining $ y \in \xi_{i-1}(x)$ and $ \xi_{i}(y) \intxn X^{\omega} \intxn \PiBall[j]{x}{\exp(-n)} \neq \emptyset $ implies $ y \in \PiBall[i]{x}{\exp(-n)} $. To see that, take $ z \in \xi_{i}(y) \intxn X^{\omega} \intxn \PiBall[j]{x}{\exp(-n)}$. Since $ \Pi_{i}(z) = \Pi_{i}(y) $ and $ \pi_{[i]}\Pi_{j} = \Pi_{i} $ by $ i \leq j $, we have
	\begin{equation*}
		\norm{ \Pi_{i}(y) - \Pi_{i}(x) } = \norm{ \Pi_{i}(z)) - \Pi_{i}(x) } \leq \norm{\Pi_{j}(z) - \Pi_{j}(x)} \leq \exp(-n),
	\end{equation*}
	which implies $ y \in \PiBall[i]{x}{\exp(-n)}$. Moreover, it follows from $ z \in \PiBall[j]{x}{\exp(-n)}$ that
	\begin{equation*}
		\PiBall[j]{x}{\exp(-n)} \subset \PiBall[j]{z}{2\exp(-n)}.
	\end{equation*}
	 Hence,
	\begin{align*}
		\betaOXj{\omega}{y}{i}\left( X^{\omega} \intxn \PiBall[j]{x}{\exp(-n)} \right) & =  \betaOXj{\omega}{z}{i}\left( X^{\omega} \intxn \PiBall[j]{x}{\exp(-n)} \right) & \text{(by $ \xi_{i}(z) = \xi_{i}(y)$)} \\ & \leq \betaOXj{\omega}{z}{i}\left( \PiBall[j]{z}{2\exp(-n)} \right) \\
		& \leq \exp\left(-n(\ul{\gamma}_{i}-\varepsilon)\right). & \text{(by \eqref{eq:D3-i-UB} and $ z \in X^{\omega}$)}
	\end{align*}
	Combining this with \eqref{eq:D3-tower} shows that for $ \omega \in \wt{\Omega}$ and $ x \in Y^{\omega}$,
	\begin{equation*}
		\betaOXj{\omega}{x}{i-1}\left(\PiBall[j]{x}{\exp(-n)}\right)\leq \exp\left(-\log c - n(\ul{\gamma}_{i} -\varepsilon ) \right) \betaOXj{\omega}{x}{i-1}\left(\PiBall[i]{x}{\exp(-n)}\right).
	\end{equation*}
	By taking logarithm, dividing by $ n $ and letting $ n \to\infty$, we have $ \lgamma{i-1}{j} \geq \ul{\gamma}_{i} - \varepsilon + \vartheta_{i} $. Then applying \eqref{eq:D3-lgammas} shows
	\begin{equation*}
		\ul{\gamma}_{i-1} \geq \ul{\gamma}_{i} + \vartheta_{i} - 2 \varepsilon.
	\end{equation*}
	Letting $ \varepsilon \to 0 $ gives $ \ul{\gamma}_{i-1} \geq \ul{\gamma}_{i} + \vartheta_{i} $, a contradiction to \eqref{eq:D3-false-consts}.
\end{proof}

\section{The disintegrations with respect to linear parts}

\label{sec:dis-linear}

In this and all the subsequent sections, we fix $ N \in \bbN $ and let $ \Gamma $ be a partition of $ \Lambda^{\bbN} $ so that for $ x , y \in \Lambda^{\bbN}$, $ x|N = y|N $ implies $ \Gamma(x)= \Gamma(y)$, which in turn implies $ A_{\varphi_{x|N}} = A_{\varphi_{y|N}}$. Specifically,
\begin{equation}\label{eq:part-by-linear}
	L \prec \Gamma \prec \{ [I] \colon I \in \Lambda^{N}\},
\end{equation}
where $ L $ is the partition of $ \Lambda^{\bbN} $ defined by $ L(x) = L(y)$ if and only if $ A_{\varphi_{x|N}} = A_{\varphi_{y|N}}$ for $ x,y \in \Lambda^{\bbN}$. We set $ T = \sigma^{N} $ and $ \calA = \vee_{i=0}^{\infty} T^{-i} \Gamma$. Recall the definitions of $ \Omega, \bfP, \beta^{\omega}, \mu^{\omega} $ from \autoref{subsec:some-disint}. In this section we introduce some properties of $ \calA $ and the associated random measures.

We begin with some notations. For $ \omega \in \Omega $, where $ \omega = \calA(x)$ with $ x \in \Lambda^{\bbN} $, and $ n \geq 0 $, define 
\begin{equation*}
 	A^{\omega|n} := A_{\varphi_{x|nN}} \mAnd  A^{-\omega|n} := (A^{\omega|n})^{-1}.
\end{equation*}
This is well defined since, by \eqref{eq:part-by-linear}, it is independent of the choice of $ x $.
For $ 1 \leq j \leq d $, let the $j$-th entry on the diagonal of $ A^{\omega|n}$ be denoted by $ A^{\omega|n}_{j}$. Define
\begin{equation*}
	\lambda_{j}^{\omega|n} := \Abs{A^{\omega|n}_{j} } \mAnd \chi_{j}^{\omega|n} := - \log \lambda_{j}^{\omega|n}.
\end{equation*}
Let $ r_{\min} := \min \{ \abs{r_{i,j}} \colon 1 \leq i, j\leq d \} $ and $ r_{\max} := \max \{ \abs{r_{i,j}} \colon 1 \leq i, j \leq d \} $. Then
\begin{equation}\label{eq:MinMaxSize}
	r_{\min}^{Nn} \leq \lambda^{\omega|n}_{j} \leq r_{\max}^{Nn} \mFor 1 \leq j \leq d.
\end{equation}
The following lemma is a direct consequence of \autoref{lem:lambda-speed}, \eqref{eq:def-bfQ} and Egorov's theorem.
\begin{lemma}\label{lem:Aj-rate}
	For $ \eta \in (0, 1)$ there exists $ \ol{\Omega} \subset \Omega $ with $ \bfP(\ol{\Omega}) > 1 - \eta $ so that for $ \omega \in \ol{\Omega}$ and $ n \in \bbN $ with $ \eta^{-1} \ll n $, we have $ \Abs{ \chi^{\omega|n}_{j} - n N\chi_{j} } < n\eta $ for $ 1 \leq j \leq d $. 
\end{lemma}

The random measure $ \mu^{\omega}$ exhibits a convolution structure. For $ \omega \in \Omega $ and $ n \geq 0 $, define
\begin{equation*}
	\nu^{\omega}_{n} =  \sum_{u\in \Lambda^{nN}} \beta^{\omega}([u]) \delta_{\varphi_{u}(0)}.
\end{equation*}
Since $ A_{\varphi_{u}} = A^{\omega|n} $ for $ u \in \Lambda^{nN}$ with $ \beta^{\omega}([u]) \neq  0 $, it follows from \eqref{eq:dyn-selfaff-general} that
\begin{equation}\label{eq:dyn-selfaff}
	\mu^{\omega} = \nu^{\omega}_{n} * A^{\omega|n} \mu^{T^{n}\omega},
\end{equation}
where $ A \theta $ denotes the pushforward of a measure $ \theta $ by a matrix $ A $.

\subsection{Nonconformal partition} Fix $ \omega \in \Omega $. Following \cite{Rapaport2023}, we define the nonconformal partitions used to analyze the entropy growth of $ \mu^{\omega} $. For $ n \in \bbZ $, let $ \calD_{n}^{d}$ be the $ n $-th level dyadic partition of $\euclid$, that is,
\begin{equation*}
	\calD_{n}^{d} = \left\{ \frac{k}{2^{n}} + \left[0, \frac{1}{2^{n}}\right)^{d}  \colon k \in \bbZ^{d} \right\}.
\end{equation*}
For $ t \in \bbR $, define $ \calD_{t}^{d} = \calD_{\flr{t}}^{d}$. We omit the superscript $ d $ when the ambient space is clear from the context. For $ \omega \in \Omega $ and $ n \geq 0 $, define
\begin{equation} \label{eq:def-E-omega}
	\calE^{\omega}_{n} := A^{\omega|n} \calD_{0}^{d} = \{ A^{\omega|n} D \colon D \in \calD_{0}^{d}\} = \bigtimes_{j=1}^{d} \lambda^{\omega|n}_{j}\calD_{0}^{1}.
\end{equation}
It follows that
\begin{equation}\label{eq:Aomega-shift-part}
	A^{\omega|b}\pi_{J}^{-1} \calE^{T^{b}\omega}_{n} = \pi_{J}^{-1} \calE^{\omega}_{n+b} \mFor b \geq 0 \text{ and } J \subset [d],
\end{equation}
and
\begin{equation}\label{eq:E-omega-D-chi}
	\calE^{\omega}_{n} \text{ and } \bigtimes_{j=1}^{d} \calD_{\chi^{\omega|n}_{j}}^{1} \text{ are $\bigO{1}$-commensurable.}
\end{equation}

For $ y \in \euclid $, we define the translation map $ T_{y}(x) = x + y, \, x\in \euclid$. It is readily checked that
\begin{equation}\label{eq:almost-transinv}
	\pi_{J}\calE^{\omega}_{n} \text{ and } T_{y}^{-1} \pi_{J}^{-1} \calE^{\omega}_{n} \text{ are $ \bigO{1}$-commensurable for } J \subset [d] \text{ and } y \in \euclid.
\end{equation}

Next, suppose $ f, g $ are two maps from a set $ X $ to $ \euclid $ such that for some $ C > 1 $,
\begin{equation*}
	\Abs{ \pi_{j}(f(x)-g(x)) } \leq C \lambda^{\omega|n}_{j} \mFor 1 \leq j \leq d \text{ and } x \in X.
\end{equation*}
Then
\begin{equation}\label{eq:fg-comm}
	f^{-1} \pi_{J}^{-1} \calE^{\omega}_{n} \text{ and } g^{-1} \pi_{J}^{-1}\calE^{\omega}_{n} \text{ are $ \bigO{C^{d}}$-commensurable for } J \subset [d].
\end{equation}

Combining \eqref{eq:dyn-selfaff}, \autoref{lem:list-ests}\ref{itm:concav-aconvex} and \eqref{eq:almost-transinv}, we obtain the following inequality for $ m, n \geq 0 $, 
\begin{equation} \label{eq:En-LB}
	\Hof{\mu^{\omega}, \calE^{\omega}_{m+n} \mid \calE^{\omega}_{n}} \geq \Hof{\mu^{T^{n}\omega}, \calE^{T^{n}\omega}_{m} } - \bigO{1}.
\end{equation}
This estimate is the major advantage of considering $ \mu^{\omega} $ and $ \calE^{\omega}_{n} $.



\subsection{Component measure}

Fix $ \omega \in \Omega $. We introduce the component measures along $ \calE^{\omega}_{n}$. Given $ \theta \in \calM(\euclid)$ and $ n \geq 0 $, let $ \theta^{\omega}_{x,n}$ be a measure-valued random element such that $ \theta^{\omega}_{x,n} = \theta_{\calE^{\omega}_{n}(x)}$ with probability $ \theta(\calE^{\medspace}_{n}(x))$ for $ x \in \euclid $. Thus, for a event $ \calU \subset \calM(\euclid)$,
\begin{equation*}
	\bbPof{ \theta^{\omega}_{x,n} \in \calU } =  \theta \left\{x \in \euclid \colon \theta_{\calE^{\omega}_{n}(x)} \in \calU \right\}.
\end{equation*}
We call $ \theta^{\omega}_{x,n}$ an \textit{$ n $-th level component} of $ \theta $ given $ \omega \in \Omega $ and $ x \in \euclid[d]$. For $ x \in \euclid $ with $ \theta(\calE^{\omega}_{n}(x)) > 0 $, we write $ \theta^{\omega}_{x,n}$ in place of $ \theta_{\calE^{\omega}_{n}(x)} $ even when no randomness is involved. Thus, for $ n \geq  0 $,
\begin{equation}\label{eq:theta-E-decomp}
	\theta = \int \theta^{\omega}_{x,n} \, \diff \theta(x).
\end{equation}
We can also choose a random scale $ n $ uniformly from a range. For example, for a finite set  $ I \subset \bbN $, define
\begin{equation*}
	\bbPof[i\in I]{ \theta^{\omega}_{x,i} \in \calU } := \frac{1}{\abs{I}} \sum_{i\in I} \bbPof{ \theta^{\omega}_{x,i} \in \calU }.
\end{equation*}
Let $ \bbE $ and $ \bbE_{i\in I}$ denote the corresponding expectation with respect to $ \bbP $ and $ \bbP_{i \in I} $. Thus, for each bounded measurable function $ f \colon \calM(\euclid) \to \bbR $ and $ n \geq 0 $,
\begin{equation*}
	\bbEof[i=n]{f(\theta^{\omega}_{x,i})} = \int f(\theta_{\calE^{\omega}_{n}(x)}) \, \diff \theta(x).
\end{equation*} 
In particular, for $ k, n \geq 0 $,
\begin{equation}\label{eq:entropy-int-comp}
	\Hof{ \theta, \calE^{\omega}_{n+k} \mid \calE^{\omega}_{n} } = \bbEof[i = n]{ \Hof{ \theta^{\omega}_{x,i}, \calE^{\omega}_{n+k} } }.
\end{equation}

We finish this section with the a useful lemma relating the entropies of a measure and its components. The proof is almost identical to \cite[Lemma 3.4]{Hochman2014} and is therefore omitted.

\begin{lemma} \label{lem:e-telescope}
	Let $ \theta \in \calM_{c}(\euclid[d])$ with $ \diam (\supp \theta) \leq R $ for some $ R \geq 1 $. Then for all $ \omega \in \Omega $ and every $ 1 \leq m \leq n$, 
	\begin{align*}
		\frac{1}{n}  \Hof{ \theta, \calE_{n}^{\omega} } &  = \bbEof[1 \leq q \leq n]{ \frac{1}{m} \Hof{ \theta_{x,q}^{\omega}, \calE_{q+m}^{\omega} } } + \bigO{ \frac{m + \log R}{n} } \\ & = \bbEof[1 \leq q \leq n]{ \frac{1}{m} \Hof{\theta, \calE^{\omega}_{q+m} \mid \calE^{\omega}_{q} } } + \bigO{ \frac{m + \log R}{n} } .
	\end{align*}
\end{lemma}

\section{Entropy of repeated self-convolutions} \label{sec:self-conv}

This section is devoted to proving the following proposition, which is analogous to \cite[Proposition 1.15]{Rapaport2023} for the random measures. It plays a crucial role in establishing the entropy increase result. The proof is adapted from \cite{Rapaport2023}. To account for the dependence on $ \omega $ and other additional parameters, based on the dynamics on $ (\Omega, \bfP)$ we adapt the arguments to prove the modified version of the statements. For clarity, we provide the necessary details.

\begin{proposition}\label{prop:conv-e}
	For $ \varepsilon \in (0,1 )$, there is $ \delta > 0 $ so that the following holds. Let $ \eta \in (0, 1)$ and $ m_{1} , \ldots, m_{d}, k_{1}, \ldots, k_{d} \in \bbN $ be with $ \varepsilon^{-1} \ll \eta^{-1} \ll m_{d} \ll k_{d} \ll m_{d-1} \ll \cdots \ll k_{2} \ll m_{1} \ll k_{1} $.  There exists $ \ol{\Omega} \subset \Omega $ with $ \bfP(\ol{\Omega}) \geq 1 - \eta $ so that for $ n\in \bbN $ with $ k_{1} \ll n $ and $ \omega \in \ol{\Omega}$ the following holds. Let $\theta \in \calM_{c}(\euclid[d])$ with $ \diam (\supp \theta) \leq \varepsilon^{-1}$ and $ \frac{1}{n} H(\theta, \calE_{n}^{\omega}) > \varepsilon $. Then there exist $ j\in [d] $ and $ Q^{\omega} \subset [n]$ with $ \#_{n}(Q^{\omega}) \geq \delta $ so that
	\begin{equation}\label{eq:conv-e}
		\frac{1}{m_{j}} H\left(\theta^{* k_{j}} , \calE_{q + m_{j}}^{\omega}  \mid \calE_{q}^{\omega} \vee \pi_{[d]\setminus \{j\}}^{-1} \calE_{q + m_{j}}^{\omega} \right) > N \chi_{j} - \varepsilon \mFor q \in Q^{\omega}.
	\end{equation}
\end{proposition}

\subsection{Entropy of self-convolutions under a condition on variance}

The purpose of this subsection is to prove the following lemma, which is analogous to \cite[Lemma 3.2]{Rapaport2023}.

\begin{lemma}\label{lem:j-1Dim}
	Let $ \eta \in (0,1) $ and $ m, \ell, k \in \bbN $ be with $ \eta^{-1} \ll m \ll \ell \ll  k $. There exists $ \ol{\Omega} \subset \Omega $ with $ \bfP(\ol{\Omega} ) > 1 - \eta $ so that, for  $ n \in \bbN $ with $ k \ll n $ and $ \omega \in \ol{\Omega}$, there is $ B^{\omega}\subset [n]$ with $ \#_{n}(B^{\omega}) > 1 - \eta $ so that the following holds. Let $ \theta_{1}, \ldots, \theta_{k} \in \calM_{c}(\euclid[d]) $ be with $ \diam(\supp \theta_{i}) \leq \eta^{-1}$ for $ 1 \leq i \leq k $. Set $ \rho := \theta_{1} * \cdots * \theta_{k} $. Suppose that there exists $ 1 \leq j \leq d $ so that $ \Var(\pi_{j}\rho) \geq \eta k $ and $ \Var(\pi_{j'}\rho) \leq \eta^{-1} $ for $ 1 \leq j' < j $. Then setting $ a : = \flr{\log k/(2N\chi_{j})} $, we have for $ \omega \in \ol{\Omega}$,
	\begin{equation*}
		\frac{1}{m} \Hof{ \rho, \calE_{\ell-a+m}^{T^{b}\omega} \mid \calE_{\ell-a}^{T^{b}\omega} \vee \pi_{[d]\setminus\{j\}}^{-1} \calE_{\ell-a+m}^{T^{b}\omega} } > N\chi_{j} - \eta \mFor b \in B^{\omega}.
	\end{equation*}
\end{lemma}

%
%


\begin{proof}
	The proof is adapted from \cite[Lemma 3.2]{Rapaport2023}.  To account for the dependence on additional parameters, we include the details for clarity. For $ 1 \leq j \leq d $,  the coordinate map from $\euclid $ to $ \bbR $ is denoted as $ \wt{\pi}_{j}(x) = \innp{x}{e_{j}} $ for $ x \in \euclid[d] $. After a translation of $ \rho $, by \autoref{lem:list-identities}\ref{itm:commensure} we can assume that the mean $ \langle \wt{\pi_{j}} \rho \rangle = 0 $ for $ 1 \leq j' \leq d $ and $ \supp \theta_{i} \subset [-\eta, \eta]^{d}$ for $ 1 \leq i \leq k $.
	
	Let $ \varepsilon \in (0,1) $ be with $ \eta^{-1} \ll \varepsilon^{-1} \ll m $. By \autoref{lem:Aj-rate} and the $ T $-invariance of $ \bfP $, there exists $ \ol{\Omega} \subset \Omega $ with $ \bfP(\ol{\Omega}) > 1 - \varepsilon/2 $ so that for $ \omega \in \ol{\Omega} $ and $ 1 \leq j' \leq d $,
	\begin{equation} \label{eq:chi-a-j'}
	\Abs{ \frac{\chi^{\omega|a}_{j'}}{ \log k } - \frac{\chi_{j'}}{2\chi_{j}} } < \varepsilon ,\qquad  \Abs{\chi_{j'}^{\omega|(\ell+m)} - (\ell+m)N\chi_{j'}} < \varepsilon,
	\end{equation}
	and
	\begin{equation} \label{eq:chi-m-j}
 		\Abs{ \chi^{T^{\ell}\omega|m}_{j} - mN\chi_{j}} < m\varepsilon.
	\end{equation}
	In what follows we take $ \omega \in \ol{\Omega} $.

	We first show that 
	\begin{equation}\label{eq:j-proj-large}
		\frac{1}{m} \Hof{ \pi_{j} A^{\omega|a} \rho , \calE^{\omega}_{\ell+m} \mid \calC^{\omega} } \geq N\chi_{j} - \frac{\eta}{4},
	\end{equation}
	where $ \calC^{\omega} := \calE^{\omega}_{\ell} \vee \pi_{[d]\setminus\{j\}}^{-1} \calE^{\omega}_{\ell+m} $. The proof of \eqref{eq:j-proj-large} is based on the Berry-Essen theorem. Next, we estimate the moments of corresponding measures. For $ 1\leq i \leq k$ and $ s = 2, 3$, it follows from \eqref{eq:chi-a-j'} that
		\begin{equation}\label{eq:moments}
			\int \abs{t}^{s} \, \diff \wt{\pi}_{j} A^{\omega|a} \theta_{i}(t) = \exp\left( - s \chi^{\omega|a}_{j} \right) \int \abs{t}^{s} \, \diff \wt{\pi}_{j} \theta_{i}(t) = \bigO{ \eta^{-s} k^{-s/2 + s\varepsilon} }.
		\end{equation}
		Thus, the variance satisfies
		\begin{equation*}
			\Var(\wt{\pi}_{j}A^{\omega|a} \rho ) = \sum_{i=1}^{k} \Var( \wt{\pi}_{j} A^{\omega|a} \theta_{i} ) =  \bigO{ \eta^{-2} k^{2\varepsilon}} .
		\end{equation*}
		Moreover,
		\begin{equation*}
			\Var(\wt{\pi}_{j} A^{\omega|a} \rho ) = \exp\left( -2 \chi^{\omega|a}_{j} \right) \Var(\wt{\pi}_{j}\rho) \geq \eta k^{-2\varepsilon}.
		\end{equation*}
		Hence
		\begin{equation*}
			\frac{ \sum_{i=1}^{k}\int \abs{t}^{3} \, \diff \wt{\pi}_{j} A^{\omega|a} \theta_{i}(t) }{ \Var(\wt{\pi}_{j}A^{\omega|a} \rho )^{3/2} } = \bigO{ \eta^{-9/2} k^{-1/2 + 6\varepsilon} }.
		\end{equation*}
		Combining all above with $ \varepsilon^{-1} \ll m \ll \ell $ and \cite[Theorem 3.1 and Lemma 3.3]{Rapaport2023}, we conclude from \autoref{lem:list-ests}\ref{itm:commensure} and \eqref{eq:chi-m-j} that
		\begin{align*}
			\frac{1}{ mN\chi_{j} } & \Hof{ \wt{\pi}_{j} A^{\omega|a} \rho, \calD^{1}_{\chi_{j}^{\omega|\ell} + \chi^{T^{\ell}\omega|m}_{j} } \mid \calD^{1}_{\chi^{\omega|\ell}_{j}}} \\
			& \geq \frac{1}{ mN\chi_{j} } \Hof{ \wt{\pi}_{j} A^{\omega|a} \rho, \calD^{1}_{\chi_{j}^{\omega|\ell} + mN\chi_{j} } \mid \calD^{1}_{\chi^{\omega|\ell}_{j}}} - \bigO{\varepsilon} \\
			& \geq 1 - \varepsilon - \bigO{\varepsilon}  \geq 1 - \bigO{\varepsilon}.
		\end{align*}
		By \eqref{eq:E-omega-D-chi} and $ \eta^{-1} \ll \varepsilon^{-1} $, this proves \eqref{eq:j-proj-large}.
		
		We proceed to estimate the error caused by $ \pi_{j}$ in \eqref{eq:j-proj-large}. For $ j' \in [d] \setminus \{j\}$, set
		\begin{equation*}
			S_{j'} := \left\{ x \in \euclid \colon \Abs{\pi_{j'}A^{\omega|a}x} \leq  \exp\left(-2N \chi_{d}(\ell+m)\right) \right\},
		\end{equation*}
		and define $ S := \intxn_{j'\in [d] \setminus \{j\}} S_{j'} $. For $ x \in S$,
		\begin{equation*}
			\Abs{ A^{\omega|a} x - \pi_{j} A^{\omega|a} x } = \bigO{ \exp\left(-2 N \chi_{d}(\ell+m)\right) }.
		\end{equation*}
	Hence by \eqref{eq:chi-a-j'} and \eqref{eq:fg-comm},
		\begin{equation}\label{eq:S-proj-j}		\Hof{ A^{\omega|a} \rho_{S}, \calE^{\omega}_{\ell+m} \mid \calC^{\omega} }  = \Hof{ \pi_{j} A^{\omega|a} \rho_{S}, \calE^{\omega}_{\ell+m} \mid \calC^{\omega} } + O(1).
		\end{equation}
		For $ j < j' \leq d$, it follows from \eqref{eq:chi-a-j'} that
		\begin{equation*}
			\Var( \wt{\pi}_{j'} A^{\omega|a} \rho)  = \exp\left(-2\chi^{\omega|a}_{j'}\right)  \sum_{i=1}^{k} \Var(\wt{\pi}_{j'}\theta_{i}) = \bigO{ \eta^{-2}k^{1-\chi_{j'}/\chi_{j}+2\varepsilon} } .
		\end{equation*}
		For $ 1 \leq j' < j $, it follows from $ \Var(\pi_{j'}\rho) \leq \eta^{-1} $ and \eqref{eq:chi-a-j'} that
		\begin{equation*}
			\Var(\wt{\pi}_{j'}A^{\omega|a} \rho ) \leq \eta^{-1}\exp\left(-2\chi^{\omega|a}_{j'}\right) = O\left( \eta^{-1} k^{-\chi_{j'}/\chi_{j} + 2 \varepsilon }\right).
		\end{equation*}
		Recall that $ \chi_{1} < \cdots < \chi_{d}$. By $ \eta^{-1} \ll \varepsilon^{-1}$, there is $ \delta > 0 $ only depending on $ \chi_{1}, \ldots, \chi_{d}$ so that 
		\begin{equation*}
			\Var( \wt{\pi}_{j'} A^{\omega|a} \rho ) = O(\eta^{-2}k^{-\delta}) \mFor j' \in [d]\setminus \{j\}.
		\end{equation*}
		From this, since the mean $ \langle \wt{\pi}_{j'} \rho  \rangle = 0 $ for $ j' \in [d]$, and by Chebyshev's inequality,
	\begin{equation} \label{eq:rho-Sc}
			\begin{aligned}
			\rho(S^{c}) \leq \sum_{j'\in[d]\setminus\{j\}} \rho(S^{c}_{j'}) & \leq \sum_{j'\in [d]\setminus\{j\}} \exp\left( 4N\chi_{d} (\ell+m) \right)  \Var(\wt{\pi}_{j'} A^{\omega|a} \rho) \\
			& = \bigO{ \exp\left( 4N\chi_{d} (\ell+m) \right) \eta^{-2} k^{-\delta} }. 
		\end{aligned}
	\end{equation}
		By $ \supp \pi_{j}A^{\omega|a} \rho  \subset [-k\eta^{-1}, k\eta^{-1}]^{d}$ and \eqref{eq:chi-a-j'},
		\begin{equation*}
			\Hof{\pi_{j}A^{\omega|a}\rho_{S^{c}}, \calE^{\omega}_{\ell+m} } = O(\ell + m + \log (k\eta^{-1})).
		\end{equation*}
		From the above two equations, it follows from $ \eta^{-1} \ll m \ll \ell \ll k $ that
		\begin{equation}\label{eq:Sc-proj}
			\frac{\rho(S^{c})}{m} \Hof{\pi_{j}A^{\omega|a}\rho_{S^{c}}, \calE^{\omega}_{\ell+m} \mid \calC^{\omega} } \leq \frac{\eta}{4}.
		\end{equation}
		Hence
		\begin{align*}
			\frac{1}{m}  & \Hof{ A^{\omega|a}\rho, \calE^{\omega}_{\ell+m} \mid \calC^{\omega} }
			\\
			& \geq \frac{\rho(S)}{m} \Hof{A^{\omega|a}\rho_{S} , \calE^{\omega}_{\ell+m} \mid \calC^{\omega} } & \by{\autoref{lem:list-ests}\ref{itm:concav-aconvex}}
			\\
			& \geq \frac{\rho(S)}{m} \Hof{ \pi_{j}A^{\omega|a}\rho_{S} , \calE^{\omega}_{\ell+m} \mid \calC^{\omega} } - \bigO{\frac{1}{m}} & \by{\eqref{eq:S-proj-j}} \\
			& \geq \frac{\rho(S)}{m} \Hof{ \pi_{j} A^{\omega|a}\rho_{S} , \calE^{\omega}_{\ell+m} \mid \calC^{\omega} }  - \frac{\eta}{4} & \by{$ \eta^{-1} \ll m $}
			\\ & \hspace{2em} +  \frac{\rho(S^{c})}{m} \Hof{\pi_{j}A^{\omega|a}\rho_{S^{c}}, \calE^{\omega}_{\ell+m} \mid \calC^{\omega} } - \frac{\eta}{4} & \by{\eqref{eq:Sc-proj}} \\
			& \geq \frac{1}{m} \Hof{ \pi_{j}A^{\omega|a} \rho, \calE^{\omega}_{\ell+m} \mid \calC^{\omega}} - \frac{3}{4}\eta & \by{\autoref{lem:list-ests}\ref{itm:concav-aconvex} and \eqref{eq:rho-Sc}} \\
			& \geq N\chi_{j} - \eta. & \by{\eqref{eq:j-proj-large}}
		\end{align*}
		By \eqref{eq:def-E-omega} and \eqref{eq:Aomega-shift-part}, this implies that
			\begin{equation*}
				\frac{1}{m}  \Hof{ \rho, \calE^{T^{a}\omega}_{\ell-a+m} \mid \calE^{T^{a}\omega}_{\ell-a} \vee \pi_{[d]\setminus\{j\}}^{-1}\calE^{T^{a}\omega}_{\ell-a+m} } \geq N\chi_{j} - \eta.
			\end{equation*}
		Since $ \bfP(\ol{\Omega}) > 1 - \varepsilon / 2 >  1 - \eta/2 $ and $ a = \bigO{\log k} \ll n $, the proof is finished by applying Birkhoff's ergodic theorem and Egorov's theorem to $ \indicator{\ol{\Omega}} $.
	\end{proof}

	\subsection{Positive entropy implies nonnegligible variance}
	
	Based on Chebyshev's inequality and \eqref{eq:MinMaxSize}, the proof of the next lemma  is almost identical to \cite[Lemma 4.4]{Hochman2014} and so omitted.
	
	\begin{lemma} \label{lem:Var2E}
		Let $ \varepsilon, \delta \in (0, 1)$ and $ m \in \bbN $ be with $ \varepsilon^{-1} \ll m \ll \delta^{-1} $. Let $ \theta \in \calM_{c}(\euclid)$ such that $ \diam (\supp \theta)  \leq \varepsilon^{-1} $ and $ \Var(\pi_{j} \theta )  \leq \delta $ for each $ 1 \leq j \leq d $. Then $ \frac{1}{m}\Hof{\theta, \calE^{\omega}_{m} } < \varepsilon $  for $ \omega \in \Omega $. 
	\end{lemma}

	The following lemma is analogous to \cite[Lemma 3.5]{Rapaport2023}, providing a nonnegligible proportion of components with positive variance based on the assumption of positive entropy. The proof is nearly identical to that of \cite[Lemma 3.5]{Rapaport2023}, based on \autoref{lem:Var2E}, and is therefore omitted.
	
	\begin{lemma}\label{lem:e2var-large}
		For $ \varepsilon \in (0,1 )$ there exists $ \delta > 0 $ so that the following holds. Let $ n \in \bbN $ be with $ \varepsilon^{-1} \ll n $. Let $ \omega \in \Omega $ and $ \theta \in \calM_{c}(\euclid)$ be with $ \diam (\supp \theta ) \leq \varepsilon^{-1} $ and $ \frac{1}{n} \Hof{\theta, \calE_{n}^{\omega} } > \varepsilon $. Then there exists $ B^{\omega} \subset [n]$ with $ \#_{n}(B^{\omega}) \geq \delta $ so that
		\begin{equation*}
			\bbPof[i=b]{ \Var(\pi_{j} A^{-\omega|i} \theta^{\omega}_{x,i}) > \delta \text{ for some } 1\leq j \leq d } \geq \delta \mFor b \in B^{\omega}.
		\end{equation*}
	\end{lemma}
%


	\subsection{Proof of \autoref{prop:conv-e}} Now we are ready to prove \autoref{prop:conv-e}.
	
	\begin{proof}[Proof of \autoref{prop:conv-e}]
		The proof is adapted from \cite[Proposition 1.15]{Rapaport2023}, with Lemmas \ref{lem:j-1Dim} and \ref{lem:e2var-large} in roles of \cite[Lemmas 3.2 and 3.5]{Rapaport2023}, respectively. To account for the dependence on additional parameters and for clarity, we include the necessary details.
		
		 Let $ \delta \in (0, 1)$ and  $ \ell_{1}, \ldots , \ell_{d} \in \bbN $ be with
		 \begin{equation}\label{eq:conv-depends}
		 	\varepsilon^{-1} \ll \delta^{-1} \ll \eta^{-1} \ll m_{d} \ll \ell_{d} \ll k_{d} \ll m_{d-1} \ll \cdots \ll k_{2} \ll m_{1} \ll \ell_{1} \ll k_{1} \ll n.
		 	\end{equation}
		 Define $ \wt{k_{j}} = \Flr{\delta k_{j} /  (2d) } $ for $ 1 \leq j \leq d $. By $ \ell_{j} \ll k_{j}$ and $ \delta^{-1} \ll k_{j}$, we have $ \ell_{j} \ll \wt{k_{j}}$. Let $ \eta_{d} := \eta $ and $ \eta_{j} := k_{j+1}^{-1} $ for $ 1 \leq j < d $. Then $ \eta_{j} \leq \eta $ and $  \eta_{j}^{-1} \ll m_{j} \ll \ell_{j} \ll \wt{k_{j}} \ll n $ for $ 1 \leq j \leq d$.

		Let $ \ol{\Omega} $ be the intersection of the $ \ol{\Omega}$'s obtained by applying \autoref{lem:j-1Dim} repeatedly with $ \eta_{j}, m_{j}, \ell_{j}, k_{j}$ in place of $ \eta, m, \ell, k $ for $ 1 \leq j \leq d $. Note that $ k_{j} \ll n$ for $ 1 \leq j \leq d $. For $ \omega \in \ol{\Omega}$, let $ B^{\omega} $ be the intersection of corresponding $ B^{\omega}$'s obtained by applying \autoref{lem:j-1Dim} with $ n $ in place of $ n $. Then $ \bfP( \ol{\Omega}) > 1 - d\eta$ and  for $ \omega \in \ol{\Omega} $,  $ \#_{n}(B^{\omega}) > 1 - d\eta$. In what follows we take $ \omega \in \ol{\Omega} $, and let $ B^{\omega} \subset [n]$ accordingly. By \autoref{lem:e2var-large} and $ \varepsilon^{-1} \ll \delta^{-1} \ll \eta^{-1} $, there exists $ \ol{B}^{\omega} \subset B^{\omega}$ with $ \#_{n}(\ol{B}^{\omega}) > \delta - d \eta > \delta/2 $ so that for $ b \in \ol{B}^{\omega} $,
		\begin{equation*}
			\bbPof[i=b]{  \Var(\pi_{j} A^{-	\omega|b} \theta^{\omega}_{x,i} ) > \delta \text{ for some } 1 \leq j \leq d } > \delta.
		\end{equation*}
		 For $ 1 \leq j \leq d $, let $B_{j}^{\omega}$ be the set of all $ b \in \ol{B}^{\omega} $ so that
		\begin{equation*}
			\bbPof[i=b]{ \Var(\pi_{j} A^{-\omega|i}\theta^{\omega}_{x,i}) > \eta_{j}  \text{ and } \Var(\pi_{j'}A^{-\omega|i}\theta^{\omega}_{x,i}) \leq \eta_{j'} \text{ for } 1 \leq j' < j } > \delta / d.
		\end{equation*} 
		It is clear that $ \ol{B}^{\omega} \subset \union_{j=1}^{d}B_{j}^{\omega}$. Since $ \#_{n}(\ol{B}^{\omega}) > \delta/2  $, it follows that $ \#_{n}(B_{j}^{\omega}) >  \delta /(2d) $ for some $ 1 \leq j \leq d $. Fix such $ j $ until the end of the proof.
		
		Note that 
		\begin{equation*}
			\varepsilon^{-1} \ll \delta^{-1} \ll \eta_{j}^{-1} \ll m_{j} \ll \ell_{j} \ll k_{j} \ll n \mAnd \eta_{j'} \leq k_{j}^{-1} \text{ for } 1 \leq j' < j.
		\end{equation*}
		Let $ b \in B^{\omega}_{j} $ be given, and define
		\begin{equation*}
			Y := \{ x\in \euclid \colon \Var(\pi_{j}A^{-\omega|b}\theta^{\omega}_{x,b}) > \eta_{j} \text{ and } \Var(\pi_{j'}A^{-\omega|b}\theta^{\omega}_{x,b})  \leq \eta_{j'} \text{ for } 1 \leq j' < j \}.
		\end{equation*}
		Recall $ \wt{k_{j}} = \Flr{\delta k_{j}/(2d)} $, and write $ k = \wt{k_{j}} $ for short. Set
		\begin{equation*}
			Z := \left\{ (x_{1}, \ldots, x_{k_{j}}) \in (\euclid)^{k_{j}} \colon \# \{ 1 \leq s \leq k_{j} \colon x_{s} \in Y \} \geq k \right\}.
		\end{equation*}
		Since $ \theta(Y) > \delta /d$ and $ \delta^{-1} \ll k_{j} $, the weak law of large numbers implies $ \theta^{\times k _{j}} (Z) >  1 - \delta $.
		
		Let $ (x_{1}, \ldots, x_{k_{j}}) \in Z $ be given. Then there exist integers $ 1 \leq s_{1} < \cdots < s_{k} \leq k_{j} $ so that $ x_{s_{i}} \in Y$ for $ 1 \leq i \leq k $. Note that
		\begin{equation*}
			\diam \left( \supp A^{-\omega|b} \theta^{\omega}_{x_{s_{i},b}} \right) = O(1)  \mFor 1 \leq i \leq k.
		\end{equation*}
		Set
		\begin{equation*}
			\rho := A^{-\omega|b} \theta^{\omega}_{x_{s_{1}},b} * \cdots * A^{-\omega|b} \theta^{\omega}_{x_{s_{k}},b}.
		\end{equation*}
		We have
		\begin{equation*}
			\Var(\pi_{j} \rho) = \sum_{i=1}^{k} \Var(\pi_{j} A^{-\omega|b} \theta^{\omega}_{x_{s_{i}},b}) \geq k \eta_{j},
		\end{equation*}
		and for each $ 1 \leq j' < j $, recalling $ k = \wt{k_{j}} = \Flr{\delta k_{j} / (2d) } $,
		\begin{equation*}
			\Var(\pi_{j'}\rho) = \sum_{i=1}^{k}\Var(\pi_{j'}A^{-\omega|b} \theta^{\omega}_{x_{s_{i}},b}) \leq k \eta_{j'}  = \bigO{ \delta k_{j} \eta_{j'} } \leq 1.
		\end{equation*}
		Recall the definition of $ B^{\omega}$ and $ B_{j}^{\omega} $. Set $ a := \Flr{\log k / (2N\chi_{j})} $. It follows from \autoref{lem:j-1Dim} that
		\begin{equation}\label{eq:j-dim-large}
			\frac{1}{m_{j}} \Hof{ \rho, \calE_{\ell_{j}-a+m_{j}}^{T^{b}\omega} \mid \calE_{\ell_{j}-a}^{T^{b}\omega} \vee \pi_{[d]\setminus\{j\}}^{-1} \calE_{\ell_{j}-a+m_{j}}^{T^{b}\omega} } > N\chi_{j} - \delta \mFor b \in B^{\omega}_{j}.
		\end{equation}
		
		For $ s\in \bbZ $ and $ b \geq 0 $, write $ \calC_{s}^{T^{b}\omega} := \calE_{s+\ell_{j}-a}^{T^{b}\omega} \vee \pi_{[d]\setminus\{j\}}^{-1} \calE_{s+\ell_{j}-a+m_{j}}^{T^{b}\omega} $ for short. Since \eqref{eq:j-dim-large}, $ k \leq k_{j}$ and $ \delta^{-1} \ll m_{j} $, we conclude from \eqref{eq:almost-transinv} and the concavity of entropy that for $ b \in B^{\omega}_{j}$,
		\begin{equation*}
			\frac{1}{m_{j}} \Hof{ *_{s=1}^{k_{j}} A^{-\omega|b} \theta^{\omega}_{x_{s},b}, \calE_{\ell_{j}-a+m_{j}}^{T^{b}\omega} \mid \calC_{0}^{T^{b}\omega} } > N\chi_{j} - 2 \delta.
		\end{equation*}
		Then by \eqref{eq:Aomega-shift-part},
		\begin{equation}\label{eq:conv-large}
			\frac{1}{m_{j}} \Hof{ *_{s=1}^{k_{j}} \theta^{\omega}_{x_{s},b}, \calE_{b+\ell_{j}-a+m_{j}}^{\omega} \mid \calC_{b}^{\omega} } > N\chi_{j} - 3 \delta.
		\end{equation}
		Note that by \eqref{eq:theta-E-decomp},
		\begin{equation*}
			\theta^{*k_{j}} = \int *_{s=1}^{k_{j}} \theta^{\omega}_{x_{s},b} \, \diff \theta^{\times k_{j}}(x_{1}, \ldots, x_{k_{j}}).
		\end{equation*}
		From this, concavity of entropy, \eqref{eq:conv-large} and $ \theta^{\times k_{j}}(Z) > 1 - \delta $, it  follows that for $ b \in B_{j}^{\omega} $,
		\begin{equation} \label{eq:shifted}
			\begin{aligned}
				\frac{1}{m_{j}} & \Hof{  \theta^{*k_{j}}, \calE_{b+\ell_{j}-a+m_{j}}^{\omega} \mid \calE_{b+\ell_{j}-a}^{\omega} \vee \pi_{[d]\setminus\{j\}}^{-1} \calE_{b+\ell_{j}-a+m_{j}}^{\omega} } \geq \\
				& \int_{Z} \frac{1}{m_{j}} \Hof{ *_{s=1}^{k_{j}} \theta^{\omega}_{x_{s}, b}, \calE_{b+\ell_{j}-a+m_{j}}^{\omega} \mid \calC_{b}^{\omega} } \, \diff \theta^{\times k_{j}}(x_{1}, \ldots, x_{k_{j}}) \geq N\chi_{j} - O(\delta).
			\end{aligned}
		\end{equation}
		
		Finally, define $ Q^{\omega} := \left\{  q \in [n] \colon q - \ell_{j} + a \in B_{j}^{\omega} \right\} $. From $ \ell_{j}, a, \delta^{-1} \ll n $ and $ \#_{n}(B_{j}^{\omega}) > \delta/(2d) $, it follows that $ \#_{n}(Q^{\omega}) > \delta /(3d)$. The proof is finished by $ \varepsilon^{-1} \ll \delta^{-1} $ and \eqref{eq:shifted}.
	\end{proof}
	

\section{Entropy of component measures}	
\label{sec:components}

In this section, we prove three lemmas about the entropy of $ \mu^{\omega} $ across different scales, which will be applied in Sections \ref{sec:EntropyIncrease} and \ref{sec:pf-main-A}. \autoref{lem:rand-asym-e} is an analog of \cite[Lemma 4.1]{Rapaport2023}, while Lemmas \ref{lem:comp-e} and \ref{lem:comp-proj-e} replace \cite[Lemmas 1.13 and 1.14]{Rapaport2023} with analogous estimates for random measures at a large proportion of scales.


We begin with some notations. By \autoref{thm:L-Y-formula}, for $ \bfP \aev \omega $ and $ J \subset [d]$, $ \pi_{J}\mu^{\omega}$ is exact dimensional with dimension given by $ \dim \pi_{J}\calA $ as in \eqref{eq:LY-formula}. Inspired by \cite{Rapaport2023}, we define
\begin{equation}\label{eq:def-kappa-A}
	\kappa_{\calA} := \sum_{j=1}^{d-1} \chi_{j} + \chi_{d}( \dim \calA - (d-1)).
\end{equation}
Now, we are ready to state the three lemmas to be proved in this section.

\begin{lemma}\label{lem:rand-asym-e} Suppose $ \dim \pi_{[d-1]} \calA = d - 1$. For $ \eta \in (0,1)$ there exists $ \ol{\Omega} \subset \Omega $ with $ \bfP(\ol{\Omega}) > 1 - \eta $ such that for $ n \in \bbN $ with $ \eta^{-1} \ll n $, 
	\begin{equation*}\label{eq:rand-asym-e}
		\Abs{  \frac{1}{n}\Hof{ \mu^{\omega}, \calE^{\omega}_{n} } - N \kappa_{\calA}}  <  \eta \mFor  \omega \in \ol{\Omega}.
	\end{equation*}
\end{lemma}

\begin{lemma} \label{lem:comp-e} 
	Suppose $ \dim \pi_{[d-1]} \calA = d - 1 $. For $\eta \in (0,1)$ there exists $ \ol{\Omega} \subset \Omega $ with $ \bfP(\ol{\Omega}) \geq 1 - \eta $ so that the following holds. Let $ m, n \in \bbN $ be with $ \eta^{-1} \ll m \ll n $. Then for $ \omega \in \ol{\Omega} $ there is $ Q^{\omega} \subset [n]$ with $ \#_{n}(Q^{\omega}) \geq 1 -\eta $ so that
	\begin{equation*}\label{eq:comp-e}
		\frac{1}{m} \Hof{ \mu^{\omega}, \calE_{q+m}^{\omega} \mid \calE_{q}^{\omega} } > N \kappa_{\calA} - \eta \mFor q \in Q^{\omega}.
	\end{equation*}
\end{lemma}

\begin{lemma} \label{lem:comp-proj-e}  
	Suppose $ \dim \pi_{[J]} \calA = \abs{J} $ for some $ J \subset [d]$. For $\eta \in (0,1)$ there exists $ \ol{\Omega} \subset \Omega $ with $ \bfP(\ol{\Omega}) \geq 1 - \eta $ so that the following holds. Let $ m, n \in \bbN $ be with $ \eta^{-1} \ll m \ll n $. Then for $ \omega \in \ol{\Omega} $ there is $ Q^{\omega} \subset [n]$ with $ \#_{n}(Q^{\omega}) \geq 1 -\eta $ so that
	\begin{equation*}\label{eq:comp-proj-e}
		\frac{1}{m} \Hof{ \mu^{\omega}, \pi_{J}^{-1}\calE_{q+m}^{\omega} \mid \calE_{q}^{\omega} } > N \sum_{j\in J} \chi_{j} - \eta \mFor q \in Q^{\omega}.
	\end{equation*}
\end{lemma}

\subsection{Entropy growth along dyadic partitions} In this subsection, we explore the entropy growth of the random measures along dyadic partitions.

\begin{lemma}\label{lem:Dn-rate}
	For $ \eta \in (0, 1) $ there exists $ \ol{\Omega} \subset \Omega $ with $ \bfP(\ol{\Omega}) > 1 - \eta $ so that for $ n \in \bbN $ with $ \eta^{-1} \ll n $ and $ J \subset [d]$,
	\begin{equation*}
		\Abs{ \frac{1}{n} \Hof{ \pi_{J}\mu^{\omega}, \calD_{\chi^{\omega|n}_{d}} } - N \chi_{d}\dim  \pi_{J}\calA} < \eta \mFor \omega \in \ol{\Omega}.
	\end{equation*}
\end{lemma}

\begin{proof}
	By Egorov's theorem, there is $ \Omega_{1} \subset \Omega $ with $ \bfP(\Omega_{1}) > 1 - \eta / 2 $ so that for $ \omega \in \Omega_{1} $
	\begin{equation*}
		\Abs{ \frac{1}{n} \Hof{ \pi_{J}\mu^{\omega}, \calD_{ nN\chi_{d} } } - N \chi_{d}\dim  \pi_{J}\calA} < \eta/2.
	\end{equation*}
	On the other hand, by \autoref{lem:Aj-rate} there exists $ \ol{\Omega} \subset \Omega_{1} $ with $ \bfP(\ol{\Omega}) > 1 - \eta $ so that for $ \omega \in \ol{\Omega}$
	\begin{equation*}
		\Abs{ \chi^{\omega|n}_{d} - n N \chi_{d}} < n\eta/2.
	\end{equation*}
	The proof is finished by combining the above two equations with \autoref{lem:list-ests}\ref{itm:commensure}.
\end{proof}

\begin{lemma}\label{lem:Tn-Dn-rate}
	Suppose $ \dim \pi_{J} \calA = \abs{J}$ for some $ J \subset [d]$. For $ \eta \in (0, 1) $ there exists $ \ol{\Omega} \subset \Omega $ with $ \bfP(\ol{\Omega}) > 1 - \eta $ so that for $ n \in \bbN $ with $ \eta^{-1} \ll n $,
	\begin{equation}
		\Abs{ \frac{1}{n} \Hof{\pi_{J}\mu^{T^{n}\omega}, \calD_{\chi_{d}^{\omega|n}-\chi^{\omega|n}_{1}}} - \abs{J}N(\chi_{d}-\chi_{1})} < \eta \mFor \omega \in \ol{\Omega}.
	\end{equation}
\end{lemma}
\begin{proof}
	Let $ \varepsilon \in (0,1)$ be with $ \eta^{-1} \ll \varepsilon^{-1} $. By Egorov's theorem, there is $ \Omega_{1} \subset \Omega  $ with $ \bfP(\Omega_{1}) > 1 - \varepsilon $ and $ n_{0} \in \bbN $ so that for $ \omega \in \Omega_{1} $ and $ n \geq n_{0} $,
	\begin{equation*}
		\frac{1}{n} \Hof{\pi_{J} \mu^{\omega}, \calD_{n}} \geq \abs{J} - \varepsilon.
	\end{equation*}
	Then by $ \bfP $ being $ T $-invariant, we have
	\begin{align*}
		\int \inf_{n\geq n_{0}}\frac{1}{n} \Hof{ \pi_{[d-1]} \mu^{T^{n}\omega},  \calD_{n} } \, \diff \bfP( \omega ) & = \int \inf_{n\geq n_{0}} \frac{1}{n} \Hof{ \pi_{[d-1]} \mu^{\omega},  \calD_{n} } \, \diff \bfP( \omega ) \\
		& \geq \int_{\Omega_{1}} \inf_{n\geq n_{0}} \frac{1}{n} \Hof{ \pi_{[d-1]} \mu^{\omega},  \calD_{n} } \, \diff \bfP( \omega ) \\
		 & \geq \abs{J} - \bigO{\varepsilon}.
	\end{align*}
	
	On the other hand, we have $ (1/n) \Hof{ \pi_{J} \mu^{T^{n}\omega},\calD_{n}} \leq \abs{J} $ for $ n  \in \bbN $. From this and above, it follows that there exists $ \Omega_{2} \subset \Omega $ with $ \bfP(\Omega_{2}) > 1 - \bigO{\varepsilon^{1/3}}$ so that for $ \omega \in \Omega_{2} $ and $ n \geq n_{0} $,
	\begin{equation}\label{eq:Tn-omega-J}
		\Abs{ \frac{1}{n} \Hof{ \pi_{J} \mu^{T^{n}\omega}, \calD_{n}} - \abs{J} } \leq \bigO{\varepsilon^{1/3}}.
	\end{equation}
	By \autoref{lem:Aj-rate} there is $ \ol{\Omega } \subset \Omega_{2}$ with $ \bfP(\ol{\Omega}) > 1 - \bigO{\varepsilon^{1/3}} $ so that for $ \omega \in \ol{\Omega}$ and $ \varepsilon^{-1} \ll n $,
	\begin{equation}\label{eq:chi-d-1-speed}
		\Abs{\chi^{\omega|n}_{d}-\chi^{\omega|n}_{1} - nN(\chi_{d}-\chi_{1})} \leq n\varepsilon. 
	\end{equation}
	Combining \eqref{eq:Tn-omega-J} and \eqref{eq:chi-d-1-speed}, we conclude from \autoref{lem:list-ests}\ref{itm:commensure} that for $ \omega \in \ol{\Omega} $ and $ \varepsilon^{-1} \ll n $,
	\begin{equation*}
		\Abs{ \frac{1}{n} \Hof{\pi_{J}\mu^{T^{n}\omega}, \calD_{\chi_{d}^{\omega|n}-\chi^{\omega|n}_{1}}} - \abs{J}N(\chi_{d}-\chi_{1})} < \bigO{\varepsilon^{1/3}}.
	\end{equation*}
	This finishes the proof since $\eta^{-1} \ll \varepsilon^{-1}$.
\end{proof}

\subsection{Proof of Lemmas \ref{lem:rand-asym-e}--\ref{lem:comp-proj-e}} In this subsection, we prove the lemmas in the beginning of this section. First we prove \autoref{lem:rand-asym-e}.

\begin{proof}[Proof of \autoref{lem:rand-asym-e}]
	Let $ \varepsilon \in (0,1)$ be with $ \eta^{-1} \ll \varepsilon^{-1} \ll n $.
	Let $ \ol{\Omega}$ be the intersection of the $ \ol{\Omega}$'s obtained by applying Lemmas \ref{lem:Aj-rate}, \ref{lem:Dn-rate} and \ref{lem:Tn-Dn-rate} with $ \varepsilon, n $ in place of $ \eta,n $. Then $ \bfP(\ol{\Omega}) >  1 - 3 \varepsilon $. In what follows we take $ \omega \in \ol{\Omega}$. Note that by \eqref{eq:E-omega-D-chi},
	\begin{equation*}
		\Hof{ \mu^{\omega}, \calE^{\omega}_{n} } = \Hof{ \mu^{\omega}, \calD_{\chi_{d}^{\omega|n}}} - \Hof{ \mu^{\omega}, \calD_{\chi_{d}^{\omega|n}} \mid \calE^{\omega}_{n}} + \bigO{1}.
	\end{equation*}
	From this, \autoref{lem:Dn-rate}, \eqref{eq:def-kappa-A} and $ \eta^{-1} \ll \varepsilon^{-1} $, it suffices to show
	\begin{equation}\label{eq:asym-e-goal}
		\frac{1}{n}\Hof{ \mu^{\omega}, \calD_{\chi_{d}^{\omega|n}} \mid \calE^{\omega}_{n}} = N\sum_{j=1}^{d-1} (\chi_{d} - \chi_{j}) + \bigO{\varepsilon}.
	\end{equation}
	
	First we show the upper bound. It follows from \autoref{lem:Aj-rate} that for each $ E \in \calE^{\omega}_{n}$,
	\begin{align*}
		\log \# \left\{  D \in \calD_{\chi_{d}^{\omega|n}} \colon D \intxn E \neq \emptyset \right\} & \leq \sum_{j=1}^{d}( \chi^{\omega|n}_{d} - \chi^{\omega|n}_{j} ) + \bigO{1} \leq nN\sum_{j=1}^{d} (\chi_{d} -\chi_{j}) + \bigO{n\varepsilon}.
	\end{align*}
	Thus,
	\begin{equation}\label{eq:asym-e-UB}
		\frac{1}{n}\Hof{ \mu^{\omega}, \calD_{\chi_{d}^{\omega|n}} \mid \calE^{\omega}_{n}} \leq N\sum_{j=1}^{d-1} (\chi_{d} - \chi_{j}) + \bigO{\varepsilon}.
	\end{equation}
	
	Next, we prove the lower bound in \eqref{eq:asym-e-goal}. Since $ A^{-\omega|n} \calD_{\chi^{\omega|n}_{d}} $ and $ \pi_{[d-1]}^{-1}\left(\bigtimes_{j=1}^{d-1} \calD_{\chi^{\omega|n}_{d}-\chi^{\omega|n}_{j}}\right) $ are $ O(1)$-commensurable, it follows from \eqref{eq:dyn-selfaff}, \autoref{lem:list-ests}\ref{itm:concav-aconvex} and \eqref{eq:almost-transinv} that
	\begin{equation}\label{eq:mu-pi-d-1}
		\begin{aligned}
	\Hof{ \mu^{\omega}, \calD_{\chi_{d}^{\omega|n}} \mid \calE^{\omega}_{n}} & = \Hof{ \nu^{\omega}_{n} * A^{\omega|n} \mu^{T^{n}\omega}, \calD_{\chi_{d}^{\omega|n}} \mid \calE^{\omega}_{n}} \\
		& \geq  \Hof{ \mu^{T^{n}\omega}, (A^{\omega|n})^{-1} \calD_{\chi^{\omega|n}_{d}} }  - \bigO{1} \\ & \geq \Hof{ \mu^{T^{n}\omega}, \pi_{[d-1]}^{-1}\left(\bigtimes_{j=1}^{d-1} \calD_{\chi^{\omega|n}_{d}-\chi^{\omega|n}_{j}}\right) }  - \bigO{1} \\
		& = \Hof{ \pi_{[d-1]} \mu^{T^{n}\omega}, \bigtimes_{j=1}^{d-1} \calD_{\chi^{\omega|n}_{d}-\chi^{\omega|n}_{j}} }  - \bigO{1}.
	\end{aligned}
	\end{equation}
	For each $ E \in \bigtimes_{j=1}^{d-1} \calD_{\chi^{\omega|n}_{d}-\chi^{\omega|n}_{j}}$, by \autoref{lem:Aj-rate} we have
	\begin{align*}
		\log \# \left\{  F \in \calD^{d-1}_{\chi_{d}^{\omega|n}-\chi_{1}^{\omega|n}}  \colon F \intxn E \neq \emptyset \right\} & \leq \sum_{j=1}^{d-1}( \chi^{\omega|n}_{j} - \chi^{\omega|n}_{1} ) + \bigO{1} \leq nN\sum_{j=1}^{d-1} (\chi_{j} - \chi_{1}) + \bigO{n\varepsilon}.
	\end{align*} 
	Thus,
	\begin{equation}\label{eq:pi-d-1-UB}
		\frac{1}{n}\Hof{ \pi_{[d-1]} \mu^{T^{n}\omega}, \calD^{d-1}_{\chi_{d}^{\omega|n}-\chi_{1}^{\omega|n}}  \mid \bigtimes_{j=1}^{d-1} \calD_{\chi^{\omega|n}_{d}-\chi^{\omega|n}_{j}} } \leq N \sum_{j=1}^{d-1} (\chi_{j} - \chi_{1}) + \bigO{\varepsilon}.
	\end{equation}
	Applying \autoref{lem:list-identities}\ref{itm:ChainRule-H}, we conclude from \eqref{eq:mu-pi-d-1}, \eqref{eq:pi-d-1-UB} and \autoref{lem:Tn-Dn-rate} that
	\begin{equation*}\label{eq:asym-e-LB}
		\begin{aligned}
			\frac{1}{n}\Hof{ \mu^{\omega}, \calD_{\chi_{d}^{\omega|n}} \mid \calE^{\omega}_{n}} & \geq (d-1)N(\chi_{d}-\chi_{1}) - N\sum_{j=1}^{d-1} (\chi_{j}- \chi_{1}) - \bigO{\varepsilon} \\
			& = N \sum_{j=1}^{d-1}(\chi_{d}-\chi_{j}) - \bigO{\varepsilon}.
		\end{aligned}
	\end{equation*}
	This, together with \eqref{eq:asym-e-UB}, finishes the proof of \eqref{eq:asym-e-goal}.
\end{proof}

Next, we prove \autoref{lem:comp-e}.

\begin{proof}[Proof of \autoref{lem:comp-e}]
	
	By applying \autoref{lem:rand-asym-e} with $ \eta/2, m $ in place of $ \eta, n $, there exists $ \Omega_{1} \subset \Omega $ with $ \bfP(\Omega_{1}) > 1 - \eta / 2 $ so that for $ \omega \in \Omega_{1}$,
	\begin{equation*}
		\frac{1}{m} \Hof{\mu^{\omega},\calE^{\omega}_{m}} > N\kappa_{\calA} - \frac{\eta}{2}.
	\end{equation*}
	By applying Birkhoff's ergodic theorem and Egorov's theorem to $ \indicator{\Omega_{1}} $, we find $ \ol{\Omega} \subset \Omega $ with $ \bfP(\ol{\Omega}) > 1 - \eta $ such that for $ \omega \in \ol{\Omega}$ there is $ Q^{\omega} \subset [n] $ with $ \#_{n}(Q^{\omega}) > 1 - \eta $ and $ T^{q}\omega \in \Omega_{1}$ for $ q \in Q^{\omega} $. From the above inequality, $ T^{q}\omega \in \Omega_{1}$, $ \eta^{-1} \ll m $ and \eqref{eq:En-LB}, it follows that
	\begin{equation*}
		\frac{1}{m}\Hof{ \mu^{\omega}, \calE_{q+m}^{\omega} \mid \calE^{\omega}_{q} } \geq \frac{1}{m} \Hof{ \mu^{T^{q}\omega}, \calE_{m}^{T^{q}\omega} } - \bigO{\frac{1}{m}} > N\kappa_{\calA} - \eta.
	\end{equation*}
	This finishes the proof.
\end{proof}

Finally, we prove \autoref{lem:comp-proj-e}.

\begin{proof}[Proof of \autoref{lem:comp-proj-e}]
	Let $ \varepsilon \in (0,1)$ be with $ \eta^{-1} \ll \varepsilon^{-1} \ll m $. Let $ \Omega_{1}$ be the intersection of the $ \ol{\Omega}$'s obtained from applying  \autoref{lem:Aj-rate} and \autoref{lem:Dn-rate} with $ \varepsilon, m $ in place of  $ \eta, n $.  Then $ \bfP(\Omega_{1}) > 1 - 2 \varepsilon $. By applying Birkhoff's ergodic theorem and Egorov's theorem to $ \indicator{\Omega_{1}} $, we find $ \ol{\Omega} \subset \Omega $ with $ \bfP(\ol{\Omega}) > 1 - \eta $ so that for $ \omega \in \ol{\Omega}$ there is $ Q^{\omega} \subset [n] $ with $ \#_{n}(Q^{\omega}) > 1 - \eta $ and $ T^{q}\omega \in \Omega_{1}$ for $ q \in Q^{\omega} $. In what follows we take $ \omega \in \Omega $ and let $ q \in Q^{\omega}$. Then $ T^{q} \omega \in \Omega_{1}$.
	
	For $ E \in \calE_{m}^{T^{q}\omega} $ with $ E \intxn \pi_{J}\euclid \neq \emptyset $, by \autoref{lem:Aj-rate} we have
	\begin{equation*}
		\log \#\left \{ D \in \calD_{\chi_{d}^{T^{q}\omega|m}} \colon D \intxn E \intxn \pi_{J} \euclid \neq \emptyset \right\} \leq m N\sum_{j\in J}( \chi_{d} - \chi_{j}) + \bigO{m\varepsilon}.
	\end{equation*}
	Thus,
	\begin{equation}\label{eq:piJ-cond-UB}
		\frac{1}{m}\Hof{\pi_{J}\mu^{T^{q}\omega}, \calD_{\chi_{d}^{T^{q}\omega|m}} \mid \calE_{m}^{T^{q}\omega} } \leq N\sum_{j\in J}( \chi_{d} - \chi_{j}) + \bigO{\varepsilon}.
	\end{equation}
	
	Next we estimate that
\begin{equation*}\label{eq:concavity-piJ-cond-q}
	\begin{aligned}
		& \hspace{-2em} \Hof{ \mu^{\omega}, \pi_{J}^{-1}\calE_{q+m}^{\omega} \mid \calE_{q}^{\omega} } \\ & \geq \Hof{ A^{\omega|q}\mu^{T^{q}\omega}, \pi_{J}^{-1}\calE_{q+m}^{\omega} \mid \calE_{q}^{\omega} } \hspace{4em} \by{\eqref{eq:dyn-selfaff} and concavity of entropy} \\
		& \geq \Hof{ \mu^{T^{q}\omega}, \pi_{J}^{-1}\calE_{m}^{T^{q}\omega} } - \bigO{1} \hspace{5em} \by{\eqref{eq:Aomega-shift-part}} \\ 
		 & =  \Hof{ \pi_{J}\mu^{T^{q}\omega}, \calE_{m}^{T^{q}\omega} } - \bigO{1} \hspace{5.5em} \by{\autoref{lem:list-identities}\ref{itm:list-pushfoward-H}}\\
		 & = \Hof{\pi_{J}\mu^{T^{q}\omega}, \calD_{\chi_{d}^{T^{q}\omega|m}}  } -  \Hof{\pi_{J}\mu^{T^{q}\omega}, \calD_{\chi_{d}^{T^{q}\omega|m}} \mid \calE_{m}^{T^{q}\omega} } - \bigO{1},
	\end{aligned}
\end{equation*}
where the last equality is by \autoref{lem:list-identities}\ref{itm:ChainRule-H}. 
Since $ T^{q} \omega \in \Omega_{1}$ and $ \eta^{-1} \ll m $, combining the above with \autoref{lem:Dn-rate} and \eqref{eq:piJ-cond-UB} yields that
\begin{equation*}
	\frac{1}{m}\Hof{ \mu^{\omega}, \pi_{J}^{-1}\calE_{q+m}^{\omega} \mid \calE_{q}^{\omega} } \geq N \sum_{j\in J} \chi_{j} - \bigO{\frac{1}{m}} - \bigO{\varepsilon}.
\end{equation*}
This finishes the proof since $ \eta^{-1} \ll \varepsilon^{-1} \ll m $.
\end{proof}

\section{Proof of the entropy increase result} \label{sec:EntropyIncrease}

In this section, we prove the following entropy increase result for random measures, which serves as an analog to \cite[Theorem 1.12]{Rapaport2023}. This result is a crucial ingredient in the proof of \autoref{thm:main-A}.

\begin{theorem}\label{thm:omega-e-increase}
	Suppose $ \dim \calA < d $ and $ \dim \pi_{J} \calA = \abs{J} $ for each $ J \subsetneq [d]$. For $ \varepsilon \in (0,1)$ there exists $ \delta = \delta(\varepsilon) > 0 $ so that the following holds. Let $ \eta \in (0,1)$ be with $ \varepsilon^{-1} \ll \eta^{-1} $. There exists $ \ol{\Omega} \subset \Omega $ with $ \bfP(\ol{\Omega}) > 1 - \eta $ so that for $ n \in \bbN $ with $ \eta^{-1} \ll n $ and $ \omega \in \ol{\Omega}$ the following holds. Let $ \theta \in \calM_{c}(\euclid[d]) $ with $ \diam( \supp \theta) \leq 1/\varepsilon $ and $ \frac{1}{n} H(\theta, \calE_{n}^{\omega}) > \varepsilon $. Then $ \frac{1}{n} \Hof{ \theta * \mu^{\omega}, \calE_{n}^{\omega} } \geq N\kappa_{\calA} + \delta $.
\end{theorem}

To prove \autoref{thm:omega-e-increase}, we need the following version of the Kaimanovich-Vershik lemma. Its proof follows a similar approach of \cite[Corollary 5.2]{Rapaport2023} and is therefore omitted.

\begin{lemma}\label{lem:KV}
	Let $ \omega \in \Omega$, $ \theta, \rho \in \calM_{c}(\euclid)$ and $ n \in \bbN $ be given. Then for $ k \in \bbN $,
	\begin{equation*}
		\Hof{ \theta^{*k} * \rho, \calE^{\omega}_{n} } - \Hof{ \rho, \calE^{\omega}_{n} } \leq k \left( \Hof{ \theta * \rho, \calE^{\omega}_{n} } - \Hof{ \rho, \calE^{\omega}_{n}} \right) + \bigO{k}. 
	\end{equation*}
\end{lemma}

Now we are ready to prove \autoref{thm:omega-e-increase}.

\begin{proof}[Proof of \autoref{thm:omega-e-increase}]
	The proof is adapted from \cite[Theorem 1.12]{Rapaport2023}, with \autoref{prop:conv-e}, Lemmas \ref{lem:comp-e} and \ref{lem:comp-proj-e} respectively in place of \cite[Proposition 1.15, Lemmas 1.13 and 1.14]{Rapaport2023}. To account for the dependence on additional parameters and for clarity, we provide the necessary details.
	
	Let $ \delta_{0}, \varepsilon_{1} \in (0,1)$ and $ m_{1}, \ldots, m_{d}, k_{1}, \ldots, k_{d} \in \bbN $ be with
	\begin{equation}\label{eq:depends}
		\varepsilon^{-1} \ll \delta_{0}^{-1} \ll \eta^{-1} \ll m_{d} \ll k_{d} \ll \cdots \ll m_{1} \ll k_{1} \ll \varepsilon_{1}^{-1} \ll n.
	\end{equation} 
	Let $ \ol{\Omega}$ be the intersection of the $ \ol{\Omega}$'s obtained by applying \autoref{prop:conv-e} with $ \varepsilon, \delta_{0}, \eta, m_{j}, k_{j} $ in place of $ \varepsilon, \delta, \eta, m_{j}, k_{j} $, Lemmas \ref{lem:comp-e} with $ \eta $ in place of $ \eta $, \autoref{lem:comp-proj-e} repeatedly for $ J \subsetneq [d]$ with $ J , \eta $ in place of $ J, \eta $, and \autoref{lem:rand-asym-e} with $ \varepsilon_{1} $ in place of $ \eta $. Then $ \bfP(\ol{\Omega}) > 1 - \bigO{\eta} $. Note that $ \eta^{-1} \ll m_{j} \ll k_{j} \ll n $ for $ 1 \leq j \leq d $. For $ \omega \in \ol{\Omega} $, let $ Q^{\omega}_{1}, Q^{\omega}_{2}, Q^{\omega}_{3} $ be respectively the $ Q^{\omega}$ obtained from \autoref{prop:conv-e}, Lemmas \ref{lem:comp-e} and \ref{lem:comp-proj-e}. Then $ \#_{n}(Q^{\omega}_{1}) > \delta_{0}$, $ \#_{n}(Q^{\omega}_{2})> 1 - \eta/4$ and $ \#_{n}(Q^{\omega}_{3}) > 1 - \eta / 4 $. Define $ Q^{\omega} := Q^{\omega}_{1} \intxn Q^{\omega}_{2} \intxn Q^{\omega}_{3}$. From $ \delta_{0}^{-1} \ll \eta^{-1} $, it follows that $ \#_{n}(Q^{\omega}) > \delta_{0} - \eta/2 > \delta_{0} / 2 $. Let $ 1 \leq j \leq d $ be the integer obtained along with $ Q_{1}^{\omega}$ in the application of \autoref{prop:conv-e}. In what follows we take $ \omega \in \ol{\Omega}$, and let $ Q^{\omega} \subset [n] $ accordingly.
	 
	
	 Note that $ \diam (\supp \theta^{\ast k_{j}}) \leq k_{j}/\varepsilon $ and  $ \varepsilon^{-1} \ll \eta^{-1} \ll m_{j} \ll k_{j} \ll n $. Using \autoref{lem:e-telescope}, it follows that
	\begin{equation}\label{eq:telescope}
		\begin{aligned}
	\frac{1}{n} & \Hof{ \theta^{*k_{j}} * \mu^{\omega}, \calE^{\omega}_{n} } \\ & \geq \bbEof[1\leq q \leq n]{ \frac{1}{m_{j}} \Hof{ \theta^{*k_{j}} * \mu^{\omega}, \calE^{\omega}_{q + m_{j}} \mid \calE^{\omega}_{q}}} - \bigO{\eta}  \\
		& \geq \#_{n}(Q^{\omega})   \bbEof[q\in Q^{\omega}]{ \frac{1}{m_{j}} \Hof{ \theta^{*k_{j}} * \mu^{\omega}, \calE^{\omega}_{q + m_{j}} \mid \calE^{\omega}_{q}}}  \\
		& \qquad + \#_{n}(Q^{\omega}_{2}\setminus Q^{\omega}) \bbEof[{q\in Q_{2}^{\omega} \setminus Q^{\omega}}]{ \frac{1}{m_{j}} \Hof{ \theta^{*k_{j}}* \mu^{\omega}, \calE^{\omega}_{q + m_{j}} \mid \calE^{\omega}_{q} } } - \bigO{\eta}.
	\end{aligned}
	\end{equation}
	
	By $ \dim\calA < d$, we have $ \Delta := \sum_{j=1}^{d} \chi_{j} - \kappa_{\calA} > 0 $. By \autoref{lem:list-identities}\ref{itm:ChainRule-H}, concavity of entropy, \eqref{eq:almost-transinv} and $ \eta^{-1} \ll m_{j} $, we conclude from \autoref{prop:conv-e} and \autoref{lem:comp-proj-e} that for $ q \in Q^{\omega}$,
	\begin{equation} \label{eq:Q-part}
		\begin{aligned}
			\frac{1}{m_{j}} \Hof{ \theta^{*k_{j}} * \mu^{\omega}, \calE^{\omega}_{q + m_{j}} \mid \calE^{\omega}_{q}} & \geq \frac{1}{m_{j}} \Hof{ \theta^{*k_{j}} , \calE^{\omega}_{q + m_{j}} \mid \calE^{\omega}_{q}\vee \pi_{[d]\setminus\{j\}}^{-1}\calE^{\omega}_{q + m_{j}} } \\
			& \qquad + \frac{1}{m_{j}} \Hof{ \mu^{\omega} , \pi_{[d]\setminus\{j\}}^{-1}\calE^{\omega}_{q + m_{j}} \mid \calE^{\omega}_{q}} - \bigO{\frac{1}{m_{j}}} \\
			& \geq  N\chi_{j} + N \sum_{j' \neq j }\chi_{j'} - \bigO{\eta} \\
			& = N\kappa_{\calA} + N\Delta - \bigO{\eta}.
		\end{aligned} 
	\end{equation}
	
	For $ q \in Q^{\omega}_{2}$, by concavity of entropy and $ \eta^{-1} \ll m_{j} $, it follows from \autoref{lem:comp-e} that
	\begin{equation}\label{eq:Q2-part}
		\frac{1}{m_{j}} \Hof{ \theta^{*k_{j}}* \mu^{\omega}, \calE^{\omega}_{q + m_{j}} \mid \calE^{\omega}_{q} } \geq 	\frac{1}{m_{j}} \Hof{ \mu^{\omega}, \calE^{\omega}_{q + m_{j}} \mid \calE^{\omega}_{q} } - \bigO{\frac{1}{m_{j}}} >  N\kappa_{\calA} - \bigO{\eta}.
	\end{equation}
	
	
	Combining \eqref{eq:telescope}, \eqref{eq:Q-part}, \eqref{eq:Q2-part}, $ \#_{n}(Q^{\omega}) > \delta_{0}/2$ and $ \#_{n}(Q^{\omega}_{2}) > 1 - \eta/4 $ shows that 
	\begin{align*}
	 \frac{1}{n} \Hof{ \theta^{*k_{j}} * \mu^{\omega}, \calE^{\omega}_{n} }
		&  \geq N\kappa_{\calA} + \frac{ \delta_{0} N \Delta}{2} - \bigO{\eta} \\
		& \geq \frac{1}{n} \Hof{\mu^{\omega}, \calE^{\omega}_{n} } + \frac{\delta_{0} N\Delta}{2} - \bigO{\eta} & \by{\autoref{lem:rand-asym-e}}\\
		& \geq \frac{1}{n} \Hof{\mu^{\omega}, \calE^{\omega}_{n} } + \delta_{0}^{2}. & \by{$ \delta_{0}^{-1} \ll \eta^{-1}$}
	\end{align*}
	By a rearrangement,
	\begin{equation*}
		\frac{1}{n} \left( \Hof{ \theta^{* k_{j}} * \mu^{\omega}, \calE^{\omega}_{n} } - \Hof{\mu^{\omega}, \calE^{\omega}_{n} } \right)  \geq \delta_{0}^{2}.
	\end{equation*}
	 By \autoref{lem:KV} and $ \delta_{0}^{-1} \ll k_{j} \ll n $,
	 \begin{equation*}
	 	\frac{1}{n} \left( \Hof{\theta* \mu^{\omega}, \calE^{\omega}_{n} } - \Hof{ \mu^{\omega}, \calE^{\omega}_{n} } \right) \geq \frac{\delta_{0}^{2}}{2k_{j}}.
	 \end{equation*}
	 By \autoref{lem:rand-asym-e} and $ \delta_{0}^{-1} \ll k_{j} \ll \varepsilon_{1}^{-1} $, this completes the proof with $ \delta = \delta_{0}^{2}/4k_{j}$.
\end{proof}

\section{Proof of \autoref{thm:main-A}} \label{sec:pf-main-A}

In this section, we establish the following theorem, which directly implies \autoref{thm:main-A}.

For $ n \in \bbN $, let $ \calC_{n}$ be the partition of $ \Lambda^{\bbN} $ defined by that $ \calC_{n}(x) = \calC_{n}(y) $ if and only if $ \varphi_{x|n} = \varphi_{y|n}$ for $ x, y\in \Lambda^{\bbN}$.

\begin{theorem}\label{thm:main-final}
	Fix $ N \in \bbN $. Let $ \Gamma $ be a partition of $ \Lambda^{\bbN}$ satisfying \eqref{eq:part-by-linear}. Set $ \calA = \vee_{i=0}^{\infty} \sigma^{-iN} \Gamma $. Suppose $ \chi_{1} < \cdots < \chi_{d}$, and $\Phi_{j}$  is Diophantine for $ 1 \leq j \leq d $. Suppose further that for $ x, y \in \Lambda^{\bbN}, n \in \bbN $ and $ 1 \leq j \leq d $, $ \pi_{j} \varphi_{x|n} = \pi_{j} \varphi_{y|n} $ implies $ \varphi_{x|n} = \varphi_{y|n}$. Then
	\begin{equation}\label{eq:main-A-RW}
		\dim \calA  = \min \left\{d, f_{\Phi}( \hRWphiA ) \right\}, 
	\end{equation}
	where $\dim \calA $ is from \autoref{thm:L-Y-formula}, $ f_{\Phi} $ is as in \eqref{eq:f-LyaDim}, and $ \hRWphiA $ is as in \eqref{eq:def-hRW-calA}.
\end{theorem}

\subsection{Super-exponential concentration}

Using \autoref{thm:omega-e-increase}, we derive the following theorem, which demonstrates that any linear acceleration of scales fails to produce positive entropy for $ \nu_{n}^{\omega} $. This indicates a super-exponential concentration of cylinders.


\begin{theorem}\label{thm:super-exp}
	If $ \dim \calA < d $ and $  \dim \pi_{[J]}\calA = \abs{J} $ for each $ J \subsetneq [d] $. Then for $ \varepsilon \in (0, 1)$ and $ n \in \bbN $ with $ \varepsilon^{-1} \ll n $, there exists $ \ol{\Omega} \subset \Omega $ with $  \bfP(\ol{\Omega}) > 1 - \varepsilon $ so that
	\begin{equation}
		  \frac{1}{n} \Hof{ \nu^{\omega}_{n}, \calE^{\omega}_{Mn} \mid \calE^{\omega}_{n} } < \varepsilon \mFor \omega \in \ol{\Omega}.
	\end{equation}
\end{theorem}

\begin{proof}
	Suppose on the contrary that there exist $ M > 1 $, $ \varepsilon \in (0, 1)$, $ n \in \bbN $ with $ \varepsilon^{-1} \ll n $, and $ \Omega_{1} \subset \Omega $ with $ \bfP(\Omega_{1}) \geq \varepsilon $ so that for $ \omega \in \Omega_{1}$,
	\begin{equation} \label{eq:pos-e}
		\frac{1}{n} \Hof{ \nu^{\omega}_{n}, \calE^{\omega}_{Mn} \mid \calE^{\omega}_{n} } \geq \varepsilon.
	\end{equation}
	Let $ \eta \in (0,1) $ be with
	\begin{equation}\label{eq:superdense-depends}
		\varepsilon^{-1}, M \ll \eta^{-1} \ll n.
	\end{equation}
	Let $ \Omega_{2} $ be the intersection of the $ \ol{\Omega} $'s obtained from \autoref{lem:rand-asym-e} and \autoref{thm:omega-e-increase} with $ \varepsilon, \eta $ in place of $ \varepsilon, \eta $. Then $ \bfP(\Omega_{2}) > 1 - 2 \eta $. Define $ \Omega_{3} := \Omega_{1} \intxn \Omega_{2} \intxn T^{-n} \Omega_{2} $. Since $ \bfP $ is $ T $-invariant, $ \bfP(T^{-n}\Omega_{2}) = \bfP(\Omega_{2}) > 1 - 2 \eta $. By $ \varepsilon^{-1} \ll \eta^{-1} $ we have $ \bfP(\Omega_{3}) > \varepsilon - 4\eta > \varepsilon / 2 > 0  $. In what follows we take $ \omega \in \Omega_{3} $.
	
	For $ x \in \euclid $, define $ \theta_{x}^{\omega} := A^{-\omega|n} (\nu^{\omega}_{n})_{\calE^{\omega}_{n}(x)} $. Then $ \diam (\supp \theta^{\omega}_{x}) = \bigO{1} $.
	Combining \eqref{eq:Aomega-shift-part},  \eqref{eq:entropy-int-comp} and \eqref{eq:pos-e} yields that
	\begin{equation*}
	 \int \frac{1}{n}\Hof{ \theta^{\omega}_{x} , \calE^{T^{n}\omega}_{(M-1)n} }  \, \diff \nu^{\omega}_{n}(x) = \int \frac{1}{n}\Hof{ (\nu^{\omega}_{n})_{\calE^{\omega}_{n}(x)}, \calE^{\omega}_{Mn} }  \, \diff \nu^{\omega}_{n}(x) = \frac{1}{n}\Hof{ \nu^{\omega}_{n}, \calE^{\omega}_{Mn} \mid \calE^{\omega}_{n} } \geq \varepsilon.
	\end{equation*}
	Since $ \frac{1}{n}\Hof{ \theta^{\omega}_{x} , \calE^{T^{n}\omega}_{(M-1)n} } \leq C(M-1) $ for some $ C > 0 $, from above there exists $ E \subset \euclid $ with $ \nu^{\omega}_{n}(E) > \varepsilon / (4C(M-1)) $ so that for $ x \in E $,
	\begin{equation*}\label{eq:e-in-E}
	 \frac{1}{n}\Hof{ \theta^{\omega}_{x} , \calE^{T^{n}\omega}_{(M-1)n} } > \frac{\varepsilon}{4}.
	\end{equation*}
	Hence by $ T^{n} \omega \in \Omega_{2}$ and \autoref{thm:omega-e-increase} there exists $ \delta = \delta(\varepsilon, M) > 0 $ so that
	\begin{equation}\label{eq:app-e-increase}
		\frac{1}{n}\Hof{ \theta^{\omega}_{x} * \mu^{T^{n}\omega}, \calE^{T^{n}\omega}_{(M-1)n}} \geq (M-1) N\kappa_{\calA}+(M-1)\delta.
	\end{equation}
	
	By $ \omega, T^{n} \omega \in \Omega_{2} $ and $ M \ll \eta^{-1} \ll n $, it follows from \autoref{lem:rand-asym-e} that
	\begin{equation}\label{eq:LB-mu-Tn-omega}
		\frac{1}{n} \Hof{ \mu^{T^{n}\omega}, \calE^{T^{n}\omega}_{(M-1)n} } > (M-1)N\kappa_{\calA} - \bigO{\eta},
	\end{equation}
	and
	\begin{equation}\label{eq:UB-mu-diff-omega}
		\frac{1}{n}\Hof{\mu, \calE^{\omega}_{Mn} \mid \calE^{\omega}_{n}} < (M-1) N\kappa_{\calA} + \bigO{\eta}.
	\end{equation}

	 Note that $\diam\left(\supp \theta^{\omega}_{x} * \mu^{T^{n}\omega}\right) = \bigO{1}$. From all above we estimate that,
	\begin{align*}
		 & \hspace{-2em} (M-1) N\kappa_{\calA} + \bigO{\eta} \\
		& \geq \frac{1}{n}\Hof{\mu, \calE^{\omega}_{Mn} \mid \calE^{\omega}_{n}} \hspace{21em} \by{\eqref{eq:UB-mu-diff-omega}} \\ &  = \frac{1}{n}\Hof{ \nu^{\omega}_{n} \ast A^{\omega|n} \mu^{T^{n}\omega}, \calE^{\omega}_{Mn} \mid \calE^{\omega}_{n}} \hspace{15em} \by{\eqref{eq:dyn-selfaff}} \\ 
		& \geq \int \frac{1}{n}\Hof{ (\nu^{\omega}_{n})_{\calE^{\omega}_{n}(x)} \ast A^{\omega|n} \mu^{T^{n}\omega}, \calE^{\omega}_{Mn} \mid \calE^{\omega}_{n}} \, \diff \nu^{\omega}_{n}(x) \hspace{1em} \by{concavity of entropy}\\
		& \geq \int \frac{1}{n}\Hof{ \theta^{\omega}_{x} \ast \mu^{T^{n}\omega}, \calE^{T^{n}\omega}_{(M-1)n}} \, \diff \nu^{\omega}_{n}(x) - \bigO{\eta} \hspace{3em} \by{\eqref{eq:Aomega-shift-part}} \\
		& \geq  \int_{\euclid\setminus E} \frac{1}{n}\Hof{  \mu^{T^{n}\omega}, \calE^{T^{n}\omega}_{(M-1)n}} \, \diff \nu^{\omega}_{n}(x) \hspace{2em} \by{concavity of entropy and \eqref{eq:almost-transinv}}  \\
		& \qquad + \int_{E} \frac{1}{n}\Hof{ \theta^{\omega}_{x} * \mu^{T^{n}\omega}, \calE^{T^{n}\omega}_{(M-1)n}} \, \diff \nu^{\omega}_{n}(x)  - \bigO{\eta} \\
		& \geq (1-\nu^{\omega}_{n}(E))((M-1)N\kappa_{\calA} - \bigO{\eta} ) \hspace{13em} \by{\eqref{eq:LB-mu-Tn-omega}}  \\ 
		& \qquad + \nu^{\omega}_{n}(E) ((M-1) N\kappa_{\calA}+(M-1)\delta)   - \bigO{\eta} \hspace{8.3em}  \by{\eqref{eq:app-e-increase}} \\		& = (M-1)N\kappa_{\calA} + \frac{\varepsilon \delta}{4C} - \bigO{\eta}. \hspace{7.5em} \by{$\nu^{\omega}_{n}(E) > \varepsilon / (4C(M-1))$}
	\end{align*}
	Then a rearrangement shows that
	\begin{equation*}
		 \frac{\varepsilon \delta }{C} < \bigO{\eta}.
	\end{equation*}
	This contradicts $ \delta = \delta(\varepsilon, M )$ and $ \varepsilon^{-1}, M \ll \eta^{-1} $. The proof is completed.
\end{proof}

\subsection{Proof of \autoref{thm:main-final}}

We begin with a lemma that relates the entropies of $ \nu^{\omega}_{n}$ and $ \mu^{\omega} $.

\begin{lemma}\label{lem:asym-nu-n}
	Let $ \eta \in (0, 1)$ and $ n \in \bbN$ be with $ \eta^{-1} \ll n $. Then for $ \omega \in \Omega $,
	\begin{equation*}
		\Abs{ \frac{1}{n} \Hof{\nu^{\omega}_{n}, \calE^{\omega}_{n} } - \frac{1}{n} \Hof{\mu^{\omega}, \calE^{\omega}_{n}  } } < \eta.
	\end{equation*}
\end{lemma}

\begin{proof}
	Define $ \Pi^{nN} \colon \Lambda^{\bbN} \to \euclid $ by $ \Pi^{nN}(x) = \varphi_{x|nN}(0)$ for $ x \in \Lambda^{\bbN }$. 
	Since $ \mu^{\omega} = \Pi \beta^{\omega}$, $ \nu^{\omega}_{n} = \Pi^{nN} \beta^{\omega}$, and $ \Abs{ \pi_{j}\left ( \Pi(x) - \Pi^{nN}(x)\right ) } \leq \bigO{ \lambda_{j}^{\omega|n}}$ for $ 1 \leq j \leq d $, the proof is finished by \eqref{eq:fg-comm}.
\end{proof}

Next, we give some properties of the function defined in \eqref{eq:f-LyaDim}. Let $ 1 \leq j_{1} < \cdots < j_{s} \leq d $ and write $ J = \{j_{b}\}_{b=1}^{s}$. Recall the IFS $ \Phi_{J} $ from \eqref{eq:def-Phi-J}. By \eqref{eq:f-LyaDim},
\begin{equation}\label{eq:def-fJ}
		f_{\Phi_{J}}(x) =  \begin{cases}
	\ell + \frac{x - \sum_{b=1}^{\ell}\chi_{j_{b}} }{\chi_{j_{\ell+1}}}  & \text{if } x \in \left[\sum_{b=1}^{\ell}\chi_{j_{b}},\sum_{b=1}^{\ell+1}\chi_{j_{b}}\right) \text{ for some } 0 \leq \ell \leq s - 1; \\[0.5em]
		s \frac{x }{\sum_{b=1}^{s}\chi_{j_{b}}} & \text{if } x \in \left[\sum_{b=1}^{s} \chi_{j_{b}}, \infty\right).
	\end{cases}
\end{equation}
The following two lemmas provide the desired properties of $ f_{\Phi_{J}}$. Their proofs follow directly from the definition and are thus omitted.

\begin{lemma}\label{lem:fJ-maximum}
	For $ x \geq 0 $, write
	\begin{equation*}
		Y(x) := \left\{ (y_{1},\ldots, y_{s}) \in \bbR^{s}\colon 0 \leq y_{b} \leq \chi_{j_{b}} \text{ for } 1 \leq b \leq s \text{ and } \sum_{b=1}^{s} y_{b} \leq x \right\},
	\end{equation*}
	and let $ g \colon Y(x) \to [0, \infty)$ be defined as 
	\begin{equation*}
		g(y) = \sum_{b=1}^{s} \frac{y_{b}}{\chi_{j_{b}}} \mFor y = (y_{1}, \ldots, y_{s}) \in Y(x).
	\end{equation*}
	If $ f_{\Phi_{J}}(x) \leq s $, then $ \max_{y\in Y(x)} g(y) =  f_{\Phi_{J}}(x) $ and the maximal value is uniquely attained at
	\begin{equation*}
		\wt{y} := \left (\chi_{j_{1}}, \ldots, \chi_{j_{m}}, x - \sum_{b=1}^{m}\chi_{j_{b}}, 0, \ldots, 0 \right),
	\end{equation*}
	where $ m = \max \{ 0 \leq k \leq s \colon \sum_{b=1}^{k} \chi_{j_{b}} \leq x \} $.
\end{lemma}

\begin{lemma}\label{lem:fJ-UB}
	For $ x \geq 0 $ and $ 0 \leq m < s$,
	\begin{equation*}
		m + \frac{x - \sum_{b=1}^{m} \chi_{j_{b}}}{\chi_{j_{m+1}}} \geq \min \left\{s, f_{\Phi_{J}}(x) \right\}.
	\end{equation*}
\end{lemma}

Now we are ready to prove \autoref{thm:main-final}.

\begin{proof}[Proof of \autoref{thm:main-final}]
	
The proof is adapted from \cite[Theorem 1.7]{Rapaport2023} and proceeds by induction on $ d $. To address the parameter dependence arising from disintegration and to maintain clarity, we include all necessary details. Assume that the theorem holds whenever the dimension of the ambient space is strictly less than $ d $. For $ d =  1 $, this induction hypothesis is vacuous.

Let $ \emptyset \neq J \subsetneq [d]$. Since $ \pi_{j} \varphi_{x|n} = \pi_{j} \varphi_{y|n} $ implies $ \varphi_{x|n} = \varphi_{y|n}$, the partitions $ (\calC_{n})_{n\in\bbN}$ are the same for $\Phi_{J}$ and $ \Phi $. Thus $ h_{RW}(\Phi_{J}, \calA) = h_{RW}(\Phi, \calA) $ by \eqref{eq:def-hRW-calA}. Since $ A_{\varphi_{x|n}} = A_{\varphi_{y|n}}$ implies $ A_{\pi_{J}\varphi_{x|n}} = A_{\pi_{J} \varphi_{y|n}}$, the partition $ \calA $ also satisfies the assumption in the theorem for $ \Phi_{J}$. Note that $ \dim \pi_{J} \calA $ is the dimension of $ \pi_{J} \Pi \beta^{\omega} = \Pi^{\Phi_{J}} \beta^{\omega} $ for $ \bfP \aev \omega $, where $ \Pi^{\Phi_{J}}$ is the coding map associated with $ \Phi_{J} $. Hence by the induction hypothesis,
\begin{equation}\label{eq:induction}
	  \dim \pi_{J} \calA = \min \left\{ \abs{J}, f_{\Phi_{J}}(\hRWphiA) \right\} \mFor \emptyset \neq J \subsetneq [d].
\end{equation}
Since combining \autoref{thm:L-Y-formula} and \autoref{lem:fJ-maximum} implies that $ f_{\Phi}(\hRWphiA) $ is always an upper bound of $ \dim \calA $, we only need to show that if $ \dim \calA < d$, then
\begin{equation*}
	\dim \calA \geq \min \left\{ d, f_{\Phi}(\hRWphiA) \right\}.
\end{equation*}
In what follows we assume $ \dim \calA < d$.

First, suppose that $ \dim \pi_{[d-1]} \calA < d - 1 $. Then $ \dim \pi_{[d-1]} \calA = f_{\Phi_{[d-1]}}(\hRWphiA) $ by \eqref{eq:induction}. It follows from \eqref{eq:def-fJ} that $ f_{\Phi_{[d-1]}}(\hRWphiA) = f_{\Phi}(\hRWphiA) $. Hence $ \dim \calA \geq \dim \pi_{[d-1]}\calA = f_{\Phi}(\hRWphiA) $.

Next, suppose $ \dim \pi_{[d-1]} \calA = d - 1$ and $ \dim \pi_{J} \calA < \abs{J} $ for some $ \emptyset \neq J \subsetneq [d]$.  Then $\dim \pi_{J} \calA = f_{\Phi_{J}}(\hRWphiA) $ by \eqref{eq:induction}. Write $ J = \{ j_{b}\}_{b=1}^{s}$ with $ j_{1} < \cdots < j_{s}$, and set $ J_{b} = \{j_{1}, \ldots, j_{b}\}$ for $ 0 \leq b \leq s $. It follows from \autoref{thm:L-Y-formula} that
\begin{equation*}
	\sum_{b=1}^{s} \frac{\Delta_{b}}{\chi_{j_{b}}} = f_{\Phi_{J}}(\hRWphiA),
\end{equation*}
where $ \Delta_{b} := \hcalCA{J_{b-1}} - \hcalCA{J_{b}} \leq \chi_{j_{b}}$ for $ 1 \leq b \leq s $. Recall $ \hcalCA{\emptyset} =  \hRWphiA $ by definition. Then \autoref{lem:fJ-maximum} implies that $ \hRWphiA - \hcalCA{J} = \sum_{b=1}^{s} \Delta_{b} = \hRWphiA $, and so $ \hcalCA{J} = 0 $. This shows $ \hcalCA{[d]} = 0 $ by \eqref{eq:def-hCalCA} and $ \xi_{J} \prec \xi_{[d]}$.
From $ \dim \pi_{[d-1]} \calA = d - 1 $ and \autoref{lem:fJ-maximum} it follows that
\begin{equation*}
	\hcalCA{[j-1]} - \hcalCA{[j]} = \chi_{j} \mFor 1 \leq j \leq d - 1.
\end{equation*}
Thus,
\begin{equation*}
	\hcalCA{[d-1]} - \hcalCA{[d]} = \hRWphiA - \sum_{j=1}^{d-1} \chi_{j}.
\end{equation*}
Combining the last two equations with \autoref{thm:L-Y-formula} gives
\begin{equation*}
	\dim \calA = \sum_{j=1}^{d} \frac{\hcalCA{[j-1]}-\hcalCA{[j]}}{\chi_{j}} = d - 1 + \frac{\hRWphiA-\sum_{j=1}^{d-1}\chi_{j}}{\chi_{d}} \geq \min \left\{d, f_{\Phi}(\hRWphiA)  \right\},
\end{equation*}
where the last inequality is by \autoref{lem:fJ-UB}.

Finally, suppose $ \dim \pi_{[d-1]} \calA = d - 1$ and $ \dim \pi_{J} \calA = \abs{J}$ for each $ J \subsetneq [d]$. Recall $ S_{n}(\Phi_{j}) $,  $ 1 \leq j \leq d $ from \eqref{eq:def-Diop-sep}. For $ n \in \bbN $, define $ S_{n}(\Phi) = \max_{1\leq j \leq d} S_{n}(\Phi_{j})$, and for $ \omega \in \Omega $, define
\begin{equation*}
	S_{n}^{\omega}(\Phi) = \min \left \{ \max_{1 \leq j \leq d} d(\varphi_{u, j}, \varphi_{v, j}) \colon u, v\in \Lambda^{nN}, \beta^{\omega}([u]) > 0, \beta^{\omega}([v])>0, \, \psi_{u} \neq \psi_{v} \right\},
\end{equation*}
with convention $ \min \emptyset = 0 $. Thus $ S_{n}^{\omega}(\Phi) > 0 $ implies $ S_{n}^{\omega}(\Phi) \geq S_{nN}(\Phi)$. Since $ \Phi_{j} $ is Diophantine for $ 1 \leq j \leq d $, there exists $ c >  0$ such that $ S_{n}(\Phi) > c^{n} $ for infinitely many $ n \in \bbN $. By pigeonholing, there exists $ 0 \leq l \leq N - 1 $ such that $ S_{nN+l}(\Phi) > c^{nN+l}$ for infinitely many $ n \in \bbN $. Thus,
\begin{equation}\label{eq:S_nN-LB}
	 S_{nN}(\Phi) \geq S_{nN+l}(\Phi) > c^{nN+l} \geq  (c^{2N})^{n}.
\end{equation}
In what follows we let $ \eta \in (0,1)$ and $ n \in \bbN $ be with $ \eta^{-1} \ll n $ such that \eqref{eq:S_nN-LB} holds for $ n $. Take $ M $ large enough so that $ 2 r_{\max}^{MN} < c^{2N}$.

Let $ \omega \in \Omega$. If $ S^{\omega}_{n}(\Phi) = 0 $, then $  \Hof{\nu_{n}^{\omega}, \calE^{\omega}_{Mn}} = \Hof{\beta^{\omega}, \calC_{nN}} = 0$; If $ S^{\omega}_{n}(\Phi) > 0  $, then $ S^{\omega}_{n}(\Phi) \geq S_{nN}(\Phi) > (c^{2N})^{n} $ by \eqref{eq:S_nN-LB}. From this, \eqref{eq:MinMaxSize} and  $ 2 r_{\max}^{MN} < c^{2N}$, it follows that  $ \Hof{ \nu^{\omega}_{n}, \calE^{\omega}_{Mn} } = \Hof{ \beta^{\omega}, \calC_{nN}} $. Hence,
\begin{equation}\label{eq:part->shannon}
	\Hof{ \nu^{\omega}_{n}, \calE^{\omega}_{Mn} } = \Hof{ \beta^{\omega}, \calC_{nN}} \mFor \omega \in \Omega.
\end{equation}

Let $ \ol{\Omega} $ be the intersection of the $\ol{\Omega}$'s obtained from \autoref{lem:rand-asym-e} with $ \eta, n $ in place of $ \eta, n$, and \autoref{thm:super-exp} with $ \eta, n $ in place of $ \varepsilon, n$. Then $ \bfP(\ol{\Omega}) > 1 - \bigO{\eta} $. For $ \omega \in \ol{\Omega}$, we have
\begin{align*}
	N \kappa_{\calA} &  > \frac{1}{n} \Hof{\mu^{\omega}, \calE^{\omega}_{n} } - \eta  & \by{\autoref{lem:rand-asym-e}}  \\ & >  \frac{1}{n} \Hof{\nu^{\omega}_{n}, \calE^{\omega}_{n} } - \bigO{\eta} & \by{\autoref{lem:asym-nu-n}} \\
	& > \frac{1}{n} \Hof{\nu^{\omega}_{n}, \calE^{\omega}_{Mn} } - \bigO{\eta}  & \by{\autoref{thm:super-exp}}\\
	& = \frac{1}{n}\Hof{ \beta^{\omega}, \calC_{nN}} - \bigO{\eta}. & \by{\eqref{eq:part->shannon}}
\end{align*}
Note that $\bfP(\ol{\Omega}) > 1 - \bigO{\eta}$ and $\Hof{ \beta^{\omega}, \calC_{nN}}/(nN) \leq H(p) $. From above, taking integral for $ \omega $ in $ \ol{\Omega} $ with respect to $ \bfP $ gives
\begin{align*}
	\kappa_{\calA} & \geq \int_{\ol{\Omega}} \frac{1}{nN}\Hof{ \beta^{\omega}, \calC_{nN}} \, \diff\bfP(\omega) - \bigO{\eta} \\ 
	& \geq \int_{\Omega} \frac{1}{nN}\Hof{ \beta^{\omega}, \calC_{nN}} \, \diff\bfP(\omega) - \bigO{\eta} \\
	& = \frac{1}{nN} \Hof{ \beta, \calC_{nN} \mid \wh{\calA} } - \bigO{\eta} & \by{\eqref{eq:int-omega-e}}\\
	& \geq \hRWphiA - \bigO{\eta}. & \by{\eqref{eq:def-hRW-calA}}
\end{align*}
Letting $ \eta \to  0 $ shows that $ \kappa_{\calA} \geq \hRWphiA $. Then by \eqref{eq:def-kappa-A} and \autoref{lem:fJ-UB},
\begin{equation*}
	\dim \calA \geq d-1 + \frac{ \hRWphiA - \sum_{j=1}^{d-1} \chi_{j}}{\chi_{d}} \geq \min\left\{d, f_{\Phi}(\hRWphiA) \right\} .
\end{equation*}
This finishes the proof of the final case, and so \autoref{thm:main-final}.
\end{proof}


\begin{thebibliography}{57}
	
	\bibitem[Algom et~al.(2022)Algom, Baker, and Shmerkin]{AlgomEtAl2022}
	Amir Algom, Simon Baker, and Pablo Shmerkin.
	\newblock On normal numbers and self-similar measures.
	\newblock \emph{Adv. Math.}, 399:\penalty0 Paper No. 108276, 17, 2022.
	
	\bibitem[Baker and Banaji(2024)]{BakerBanaji2024}
	Simon Baker and Amlan Banaji.
	\newblock Polynomial {F}ourier decay for fractal measures and their
	pushforwards.
	\newblock {\em arXiv preprint
		\href{http://arxiv.org/abs/2401.01241}{arXiv:2401.01241}}, 2024.
	\newblock To appear in {\em Math. Ann.}
	
	\bibitem[Bara\'{n}ski(2007)]{Baranski2007}
	Krzysztof Bara\'{n}ski.
	\newblock Hausdorff dimension of the limit sets of some planar geometric
	constructions.
	\newblock \emph{Adv. Math.}, 210\penalty0 (1):\penalty0 215--245, 2007.
	
	\bibitem[B\'{a}r\'{a}ny and K\"{a}enm\"{a}ki(2017)]{BaranyKaeenmaeki2017}
	Bal\'{a}zs B\'{a}r\'{a}ny and Antti K\"{a}enm\"{a}ki.
	\newblock Ledrappier-{Y}oung formula and exact dimensionality of self-affine
	measures.
	\newblock \emph{Adv. Math.}, 318:\penalty0 88--129, 2017.
	
	\bibitem[B\'{a}r\'{a}ny et~al.(2016)B\'{a}r\'{a}ny, Rams, and
	Simon]{BaranyEtAl2016}
	Bal\'{a}zs B\'{a}r\'{a}ny, Micha\l{} Rams, and K\'{a}roly Simon.
	\newblock On the dimension of self-affine sets and measures with overlaps.
	\newblock \emph{Proc. Amer. Math. Soc.}, 144\penalty0 (10):\penalty0
	4427--4440, 2016.
	
	\bibitem[B\'{a}r\'{a}ny et~al.(2019)B\'{a}r\'{a}ny, Hochman, and
	Rapaport]{BaranyEtAl2019}
	Bal\'{a}zs B\'{a}r\'{a}ny, Michael Hochman, and Ariel Rapaport.
	\newblock Hausdorff dimension of planar self-affine sets and measures.
	\newblock \emph{Invent. Math.}, 216\penalty0 (3):\penalty0 601--659, 2019.
	
	\bibitem[B\'{a}r\'{a}ny et~al.(2023)B\'{a}r\'{a}ny, Simon, and
	Solomyak]{BaranyEtAl2023a}
	Bal\'{a}zs B\'{a}r\'{a}ny, K\'{a}roly Simon, and Boris Solomyak.
	\newblock \emph{Self-similar and self-affine sets and measures}, volume 276 of
	\emph{Mathematical Surveys and Monographs}.
	\newblock American Mathematical Society, Providence, RI, 2023.
	
	\bibitem[Barral and Feng(2011)]{BarralFeng2011}
	Julien Barral and De-Jun Feng.
	\newblock Non-uniqueness of ergodic measures with full {H}ausdorff dimensions
	on a {G}atzouras-{L}alley carpet.
	\newblock \emph{Nonlinearity}, 24\penalty0 (9):\penalty0 2563--2567, 2011.
	
	\bibitem[Bedford(1984)]{Bedford1984}
	Tim Bedford.
	\newblock Crinkly curves, markov partitions and dimension.
	\newblock \emph{University of Warwick}, 1984.
	
	\bibitem[Bárány et~al.(2023)Bárány, Käenmäki, Pyörälä, and
	Wu]{BaranyEtAl2023}
	Balázs Bárány, Antti Käenmäki, Aleksi Pyörälä, and Meng Wu.
	\newblock Scaling limits of self-conformal measures.
	\newblock {\em arXiv preprint
		\href{http://arxiv.org/abs/2308.11399}{arXiv:2308.11399}}, 2023.
	
	\bibitem[Das and Simmons(2017)]{DasSimmons2017}
	Tushar Das and David Simmons.
	\newblock The {H}ausdorff and dynamical dimensions of self-affine sponges: a
	dimension gap result.
	\newblock \emph{Invent. Math.}, 210\penalty0 (1):\penalty0 85--134, 2017.
	
	\bibitem[Einsiedler and Ward(2011)]{EinsiedlerWard2011}
	Manfred Einsiedler and Thomas Ward.
	\newblock \emph{Ergodic theory with a view towards number theory}, volume 259
	of \emph{Graduate Texts in Mathematics}.
	\newblock Springer-Verlag London, Ltd., London, 2011.
	
	\bibitem[Falconer(1988)]{Falconer1988}
	Kenneth~J. Falconer.
	\newblock The {H}ausdorff dimension of self-affine fractals.
	\newblock \emph{Math. Proc. Cambridge Philos. Soc.}, 103\penalty0 (2):\penalty0
	339--350, 1988.
	
	\bibitem[Falconer(2003)]{Falconer2003}
	Kenneth~J. Falconer.
	\newblock \emph{Fractal geometry: Mathematical foundations and applications}.
	\newblock John Wiley \& Sons, Inc., Hoboken, NJ, second edition, 2003.
	\newblock Mathematical foundations and applications.
	
	\bibitem[Falconer and Jin(2014)]{FalconerJin2014}
	Kenneth~J. Falconer and Xiong Jin.
	\newblock Exact dimensionality and projections of random self-similar measures
	and sets.
	\newblock \emph{J. Lond. Math. Soc. (2)}, 90\penalty0 (2):\penalty0 388--412,
	2014.
	
	\bibitem[Falconer and Kempton(2017)]{FalconerKempton2017}
	Kenneth~J. Falconer and Tom Kempton.
	\newblock The dimension of projections of self-affine sets and measures.
	\newblock \emph{Ann. Acad. Sci. Fenn. Math.}, 42\penalty0 (1):\penalty0
	473--486, 2017.
	
	\bibitem[Feng(2023)]{Feng2023a}
	De-Jun Feng.
	\newblock Dimension of invariant measures for affine iterated function systems.
	\newblock \emph{Duke Math. J.}, 172\penalty0 (4):\penalty0 701--774, 2023.
	
	\bibitem[Feng and Feng(2024)]{FengFeng2024}
	De-Jun Feng and Zhou Feng.
	\newblock Dimension of homogeneous iterated function systems with algebraic
	translations.
	\newblock {\em arXiv preprint
		\href{http://arxiv.org/abs/2405.03124}{arXiv:2405.03124}}, 2024.
	
	\bibitem[Feng and Hu(2009)]{FengHu2009}
	De-Jun Feng and Huyi Hu.
	\newblock Dimension theory of iterated function systems.
	\newblock \emph{Comm. Pure Appl. Math.}, 62\penalty0 (11):\penalty0 1435--1500,
	2009.
	
	\bibitem[Feng and Wang(2005)]{FengWang2005}
	De-Jun Feng and Yang Wang.
	\newblock A class of self-affine sets and self-affine measures.
	\newblock \emph{J. Fourier Anal. Appl.}, 11\penalty0 (1):\penalty0 107--124,
	2005.
	
	\bibitem[Ferguson et~al.(2015)Ferguson, Fraser, and Sahlsten]{FergusonEtAl2015}
	Andrew Ferguson, Jonathan~M. Fraser, and Tuomas Sahlsten.
	\newblock Scaling scenery of {$(\times m,\times n)$} invariant measures.
	\newblock \emph{Adv. Math.}, 268:\penalty0 564--602, 2015.
	
	\bibitem[Fraser(2012)]{Fraser2012a}
	Jonathan~M. Fraser.
	\newblock On the packing dimension of box-like self-affine sets in the plane.
	\newblock \emph{Nonlinearity}, 25\penalty0 (7):\penalty0 2075--2092, 2012.
	
	\bibitem[Fraser(2015)]{Fraser2015a}
	Jonathan~M. Fraser.
	\newblock Remarks on the analyticity of subadditive pressure for products of
	triangular matrices.
	\newblock \emph{Monatsh. Math.}, 177\penalty0 (1):\penalty0 53--65, 2015.
	
	\bibitem[Galicer et~al.(2016)Galicer, Saglietti, Shmerkin, and
	Yavicoli]{GalicerEtAl2016}
	Daniel Galicer, Santiago Saglietti, Pablo Shmerkin, and Alexia Yavicoli.
	\newblock {$L^q$} dimensions and projections of random measures.
	\newblock \emph{Nonlinearity}, 29\penalty0 (9):\penalty0 2609--2640, 2016.
	
	\bibitem[Gatzouras and Peres(1996)]{GatzourasPeres1996}
	Dimitrios Gatzouras and Yuval Peres.
	\newblock The variational principle for {H}ausdorff dimension: a survey.
	\newblock In \emph{Ergodic theory of {${\bf Z}^d$} actions ({W}arwick,
		1993--1994)}, volume 228 of \emph{London Math. Soc. Lecture Note Ser.}, pages
	113--125. Cambridge Univ. Press, Cambridge, 1996.
	
	\bibitem[Hochman(2014)]{Hochman2014}
	Michael Hochman.
	\newblock On self-similar sets with overlaps and inverse theorems for entropy.
	\newblock \emph{Ann. of Math. (2)}, 180\penalty0 (2):\penalty0 773--822, 2014.
	
	\bibitem[Hochman(2017)]{Hochman2017}
	Michael Hochman.
	\newblock On self-similar sets with overlaps and inverse theorems for entropy
	in $\mathbb{R}^d$.
	\newblock {\em arXiv preprint
		\href{http://arxiv.org/abs/1503.09043}{arXiv:1503.09043}}, 2017.
	\newblock To appear in {\em Mem. Amer. Math. Soc}.
	
	\bibitem[Hochman(2018)]{Hochman2018}
	Michael Hochman.
	\newblock Dimension theory of self-similar sets and measures.
	\newblock In \emph{Proceedings of the {I}nternational {C}ongress of
		{M}athematicians---{R}io de {J}aneiro 2018. {V}ol. {III}. {I}nvited
		lectures}, pages 1949--1972. World Sci. Publ., Hackensack, NJ, 2018.
	
	\bibitem[Hochman and Rapaport(2022)]{HochmanRapaport2022}
	Michael Hochman and Ariel Rapaport.
	\newblock Hausdorff dimension of planar self-affine sets and measures with
	overlaps.
	\newblock \emph{J. Eur. Math. Soc. (JEMS)}, 24\penalty0 (7):\penalty0
	2361--2441, 2022.
	
	\bibitem[Hochman and Shmerkin(2012)]{HochmanShmerkin2012}
	Michael Hochman and Pablo Shmerkin.
	\newblock Local entropy averages and projections of fractal measures.
	\newblock \emph{Ann. of Math. (2)}, 175\penalty0 (3):\penalty0 1001--1059,
	2012.
	
	\bibitem[Hutchinson(1981)]{Hutchinson1981}
	John~E. Hutchinson.
	\newblock Fractals and self-similarity.
	\newblock \emph{Indiana Univ. Math. J.}, 30\penalty0 (5):\penalty0 713--747,
	1981.
	
	\bibitem[Jordan et~al.(2007)Jordan, Pollicott, and Simon]{JordanEtAl2007}
	Thomas Jordan, Mark Pollicott, and K\'{a}roly Simon.
	\newblock Hausdorff dimension for randomly perturbed self affine attractors.
	\newblock \emph{Comm. Math. Phys.}, 270\penalty0 (2):\penalty0 519--544, 2007.
	
	\bibitem[K\"{a}enm\"{a}ki(2004)]{Kaeenmaeki2004}
	Antti K\"{a}enm\"{a}ki.
	\newblock On natural invariant measures on generalised iterated function
	systems.
	\newblock \emph{Ann. Acad. Sci. Fenn. Math.}, 29\penalty0 (2):\penalty0
	419--458, 2004.
	
	\bibitem[K\"{a}enm\"{a}ki and Orponen(2023)]{KaeenmaekiOrponen2023}
	Antti K\"{a}enm\"{a}ki and Tuomas Orponen.
	\newblock Absolute continuity in families of parametrised non-homogeneous
	self-similar measures.
	\newblock \emph{J. Fractal Geom.}, 10\penalty0 (1-2):\penalty0 169--207, 2023.
	
	\bibitem[Kenyon and Peres(1996)]{KenyonPeres1996a}
	Richard Kenyon and Yuval Peres.
	\newblock Measures of full dimension on affine-invariant sets.
	\newblock \emph{Ergodic Theory Dynam. Systems}, 16\penalty0 (2):\penalty0
	307--323, 1996.
	
	\bibitem[Lalley and Gatzouras(1992)]{LalleyGatzouras1992}
	Steven~P. Lalley and Dimitrios Gatzouras.
	\newblock Hausdorff and box dimensions of certain self-affine fractals.
	\newblock \emph{Indiana Univ. Math. J.}, 41\penalty0 (2):\penalty0 533--568,
	1992.
	
	\bibitem[Ledrappier and Young(1985)]{LedrappierYoung1985}
	Fran\c{c}ois Ledrappier and Lai-Sang Young.
	\newblock The metric entropy of diffeomorphisms. {I}. {C}haracterization of
	measures satisfying {P}esin's entropy formula. {II}. {R}elations between
	entropy, exponents and dimension.
	\newblock \emph{Ann. of Math. (2)}, 122\penalty0 (3):\penalty0 509--574, 1985.
	
	\bibitem[Maker(1940)]{Maker1940}
	Philip~T. Maker.
	\newblock The ergodic theorem for a sequence of functions.
	\newblock \emph{Duke Math. J.}, 6:\penalty0 27--30, 1940.
	
	\bibitem[McMullen(1984)]{McMullen1984}
	Curt McMullen.
	\newblock The {H}ausdorff dimension of general {S}ierpi\'{n}ski carpets.
	\newblock \emph{Nagoya Math. J.}, 96:\penalty0 1--9, 1984.
	
	\bibitem[Morris and Sert(2022)]{MorrisSert2022}
	Ian~D. Morris and {\c{C}}a{\u{g}}r{\i} Sert.
	\newblock A converse statement to hutchinson's theorem and a dimension gap for
	self-affine measures.
	\newblock \emph{J. Eur. Math. Soc. (JEMS)}, 2022.
	
	\bibitem[Morris and Sert(2023)]{MorrisSert2023}
	Ian~D. Morris and {\c{C}}a{\u{g}}r{\i} Sert.
	\newblock A variational principle relating self-affine measures to self-affine
	sets.
	\newblock {\em arXiv preprint
		\href{http://arxiv.org/abs/2303.03437}{arXiv:2303.03437}}, 2023.
	
	\bibitem[Morris and Shmerkin(2019)]{MorrisShmerkin2019}
	Ian~D. Morris and Pablo Shmerkin.
	\newblock On equality of {H}ausdorff and affinity dimensions, via self-affine
	measures on positive subsystems.
	\newblock \emph{Trans. Amer. Math. Soc.}, 371\penalty0 (3):\penalty0
	1547--1582, 2019.
	
	\bibitem[Parry(1981)]{Parry1981}
	William Parry.
	\newblock \emph{Topics in ergodic theory}, volume~75 of \emph{Cambridge Tracts
		in Mathematics}.
	\newblock Cambridge University Press, Cambridge-New York, 1981.
	
	\bibitem[Peres and Shmerkin(2009)]{PeresShmerkin2009}
	Yuval Peres and Pablo Shmerkin.
	\newblock Resonance between {C}antor sets.
	\newblock \emph{Ergodic Theory Dynam. Systems}, 29\penalty0 (1):\penalty0
	201--221, 2009.
	
	\bibitem[Py\"{o}r\"{a}l\"{a}(2025)]{Pyoeraelae2025}
	Aleksi Py\"{o}r\"{a}l\"{a}.
	\newblock The dimension of projections of planar diagonal self-affine measures.
	\newblock \emph{Ann. Fenn. Math.}, 50\penalty0 (1):\penalty0 59–78, 2025.
	
	\bibitem[Rapaport(2023)]{Rapaport2023}
	Ariel Rapaport.
	\newblock Dimension of diagonal self-affine sets and measures via non-conformal
	partitions.
	\newblock {\em arXiv preprint
		\href{http://arxiv.org/abs/2309.03985}{arXiv:2309.03985}}, 2023.
	
	\bibitem[Rapaport(2024{\natexlab{a}})]{Rapaport2024}
	Ariel Rapaport.
	\newblock On self-affine measures associated to strongly irreducible and
	proximal systems.
	\newblock \emph{Adv. Math.}, 449:\penalty0 Paper No. 109734, 116,
	2024{\natexlab{a}}.
	
	\bibitem[Rapaport(2024{\natexlab{b}})]{Rapaport2024a}
	Ariel Rapaport.
	\newblock Dimension of self-conformal measures associated to an exponentially
	separated analytic {IFS} on $\mathbb{R}$.
	\newblock {\em arXiv preprint
		\href{http://arxiv.org/abs/2412.16753}{arXiv:2412.16753}},
	2024{\natexlab{b}}.
	
	\bibitem[Rapaport and Ren(2024)]{RapaportRen2024}
	Ariel Rapaport and Haojie Ren.
	\newblock Dimension of bernoulli convolutions in $\mathbb{R}^{d}$.
	\newblock {\em arXiv preprint
		\href{http://arxiv.org/abs/2406.05495}{arXiv:2406.05495}}, 2024.
	
	\bibitem[Rohlin(1952)]{Rohlin1952}
	Vladimir~A. Rohlin.
	\newblock On the fundamental ideas of measure theory.
	\newblock \emph{Amer. Math. Soc. Translation}, 1952\penalty0 (71):\penalty0 55,
	1952.
	
	\bibitem[Saglietti et~al.(2018)Saglietti, Shmerkin, and
	Solomyak]{SagliettiEtAl2018}
	Santiago Saglietti, Pablo Shmerkin, and Boris Solomyak.
	\newblock Absolute continuity of non-homogeneous self-similar measures.
	\newblock \emph{Adv. Math.}, 335:\penalty0 60--110, 2018.
	
	\bibitem[Solomyak(1998)]{Solomyak1998}
	Boris Solomyak.
	\newblock Measure and dimension for some fractal families.
	\newblock \emph{Math. Proc. Cambridge Philos. Soc.}, 124\penalty0 (3):\penalty0
	531--546, 1998.
	
	\bibitem[Solomyak(2022)]{Solomyak2022}
	Boris Solomyak.
	\newblock Fourier decay for homogeneous self-affine measures.
	\newblock \emph{J. Fractal Geom.}, 9\penalty0 (1-2):\penalty0 193--206, 2022.
	
	\bibitem[Solomyak and Śpiewak(2023)]{SolomyakSpiewak2023}
	Boris Solomyak and Adam Śpiewak.
	\newblock Absolute continuity of self-similar measures on the plane.
	\newblock {\em arXiv preprint
		\href{http://arxiv.org/abs/2301.10620}{arXiv:2301.10620}}, 2023.
	\newblock To appear in {\em Indiana Univ. Math. J.}
	
	\bibitem[Varj\'{u}(2023)]{Varju2023}
	P\'{e}ter~P. Varj\'{u}.
	\newblock Self-similar sets and measures on the line.
	\newblock In \emph{I{CM}---{I}nternational {C}ongress of {M}athematicians.
		{V}ol. {V}. {S}ections 9--11}, pages 3610--3634. EMS Press, Berlin, 2023.
	
	\bibitem[Walters(1982)]{Walters1982}
	Peter Walters.
	\newblock \emph{An introduction to ergodic theory}, volume~79 of \emph{Graduate
		Texts in Mathematics}.
	\newblock Springer-Verlag, New York-Berlin, 1982.
	
	\bibitem[Young(1982)]{Young1982}
	Lai~Sang Young.
	\newblock Dimension, entropy and {L}yapunov exponents.
	\newblock \emph{Ergodic Theory Dynam. Systems}, 2\penalty0 (1):\penalty0
	109--124, 1982.
	
\end{thebibliography}
\end{document}